\documentclass[oneside,english]{amsart}
\usepackage[T1]{fontenc}
\usepackage[latin9]{inputenc}
\usepackage{mathtools}
\usepackage{amsbsy}
\usepackage{amstext}
\usepackage{amsthm}
\usepackage{amssymb}
\usepackage{stmaryrd}
\usepackage{esint}

\makeatletter
\numberwithin{equation}{section}
\numberwithin{figure}{section}
\theoremstyle{plain}
\newtheorem{thm}{\protect\theoremname}[section]
  \theoremstyle{definition}
  \newtheorem{defn}[thm]{\protect\definitionname}
  \theoremstyle{remark}
  \newtheorem{rem}[thm]{\protect\remarkname}
  \theoremstyle{plain}
  \newtheorem{lem}[thm]{\protect\lemmaname}
  \theoremstyle{plain}
  \newtheorem{prop}[thm]{\protect\propositionname}
  \theoremstyle{plain}
  \newtheorem{cor}[thm]{\protect\corollaryname}

\makeatother

\usepackage{babel}
  \providecommand{\corollaryname}{Corollary}
  \providecommand{\definitionname}{Definition}
  \providecommand{\lemmaname}{Lemma}
  \providecommand{\propositionname}{Proposition}
  \providecommand{\remarkname}{Remark}
\providecommand{\theoremname}{Theorem}

\begin{document}

\title[Extension of 2D MCF with free boundary]{Extension of two-dimensional mean curvature flow with free boundary}

\author{Siao-Hao Guo}
\begin{abstract}
Given a mean curvature flow of compact, embedded $C^{2}$ surfaces
satisfying Neumann free boundary condition on a mean convex, smooth
support surface in 3-dimensional Euclidean space, we show that it
can be extended as long as its mean curvature and perimeter stay uniformly
bounded along the flow. 
\end{abstract}

\maketitle

\section{Introduction}

Mean curvature flow (MCF) of hypersurfaces in Euclidean space is a
one-parameter family of hypersurfaces $\left\{ \Sigma_{t}\right\} _{a<t<b}$
whose normal velocity equals the mean curvature vector. Namely, 
\[
\left(\partial_{t}X_{t}\right)^{\perp}=\overrightarrow{H}_{\Sigma_{t}},
\]
where $X_{t}$ is the position vector of $\Sigma_{t}$ and $\overrightarrow{H}_{\Sigma_{t}}=\triangle_{\Sigma_{t}}X_{t}$
is the mean curvature vector. Given a closed, embedded, $C^{2}$ hypersurface
$\Sigma_{0}$, there is a unique MCF $\left\{ \Sigma_{t}\right\} $
starting from it. After a finite time, say as $t\nearrow T$, the
flow will become singular. One can see this by comparing the flow
with the evolution of any enclosing sphere using the maximum principle.
Moreover, singularities of the flow always accompany the blow up of
the second fundamental form (cf. \cite{E}), i.e. 
\[
\limsup_{t\nearrow T}\left\Vert A_{\Sigma_{t}}\right\Vert _{L^{\infty}}=\infty,
\]
where $A_{\Sigma_{t}}=\nabla_{\Sigma_{t}}^{2}X_{t}\cdot N_{\Sigma_{t}}$
is the second fundamental form of $\Sigma_{t}$ and $N_{\Sigma_{t}}$
is the unit normal vector. A natural question is whether the mean
curvature must also blow up at the first singular time. If so, singularities
can be interpreted as the blow up of the normal speed. 

Under various assumptions, the mean curvature does blow up at the
first singular time. For instance, in \cite{HS} it is shown that
if the initial hypersurface is closed and mean convex (i.e. the mean
curvature vector points to the inward direction), the second fundamental
form will be bounded by a multiple of the mean curvature, which forces
the mean curvature to blow up at the first singular time. More examples
can be seen in \cite{C,LS,LW,LZZ}. In particular, Li-Wang in \cite{LW}
show that the MCF of closed surfaces in $\mathbb{\mathbb{R}}^{3}$
can be extended so long as the mean curvature stays uniformly bounded.
This tells us that singularities of MCF of closed surfaces in $\mathbb{\mathbb{R}}^{3}$
can be characterized by the blow-up of the mean curvature. 

On the other hand, people also consider MCF subject to various boundary
conditions, in particular the ``free boundary condition''. Given
an open, connected domain $U$ in Euclidean space with smooth boundary,
a flow $\left\{ \Sigma_{t}\right\} $ living in $U$ satisfies the
``free boundary condition'' provided that it meets $\Gamma=\partial U$
orthogonally at every time. Namely, 
\[
N_{\Sigma_{t}}\cdot\nu=0\qquad\textrm{on}\quad\partial\Sigma_{t}=\Sigma_{t}\cap\Gamma,
\]
where $\nu$ is the inward, unit normal vector of $\Gamma$. As in
the boundaryless case, given a compact, embedded, $C^{2}$ hypersurface
$\Sigma_{0}$ meeting $\Gamma$ orthogonally, there is a unique MCF
$\left\{ \Sigma_{t}\right\} $ which starts from it and has free boundary
on $\Gamma$ (cf. \cite{S}). The flow may have long-time existence
under certain conditions (for instance, see \cite{H1}). However,
if the flow becomes singular in finite time, the second fundamental
form must blow up (cf. \cite{S}).

In this paper, we investigate the extension problem for MCF of surfaces
with free boundary so as to gain insight into the development of singularities.
The support surface $\Gamma=\partial U$ is assumed to be smooth and
mean convex (i.e. the mean curvature vector of $\Gamma$ points toward
$U$), which is used to ensure the ``boundary approaches boundary''
condition (see \cite{K} for more details). In Theorem \ref{main thm}
we show that any MCF $\left\{ \Sigma_{t}\right\} _{0\leq t<T}$ moving
freely in $U$ (i.e. satisfying the free boundary condition) can be
extended as long as its mean curvature and perimeter stay uniformly
bounded, i.e.
\[
\sup_{0\leq t<T}\left(\left\Vert H_{\Sigma_{t}}\right\Vert _{L^{\infty}}+\mathcal{H}^{1}\left(\partial\Sigma_{t}\right)\right)<\infty,
\]
where $H_{\Sigma_{t}}=\overrightarrow{H}_{\Sigma_{t}}\cdot N_{\Sigma_{t}}$
is the mean curvature of $\Sigma_{t}$, $\mathcal{H}^{k}$ denotes
$k$-dimensional Hausdorff measure, and $T>0$ is a finite number.
This can be thought of as a generalization of \cite{LW}, in which
surfaces are assume to be closed. In fact, our proof parallels that
of Li-Wang in \cite{LW}. We adapt their argument to the free boundary
setting by using the method of reflection and many crucial properties
for MCF with free boundary (cf. \cite{B,GJ,K,S}). The reason for
assuming the perimeter bound (compared with the assumption in \cite{LW})
comes from the use of the Gauss-Bonnet theorem (see Section \ref{Hypotheses}).
A key part of the proof is to show that every limit point $P$ of
the flow $\left\{ \Sigma_{t}\right\} _{0\leq t<T}$ as $t\nearrow T$
satisfies
\begin{equation}
\Theta_{\left\{ \Sigma_{t}\right\} }\left(P,T\right)=\left\{ \begin{array}{c}
1,\qquad\textrm{if}\quad P\in U\\
\frac{1}{2},\qquad\textrm{if}\quad P\in\Gamma
\end{array}\right.,\label{unit Huisken's density}
\end{equation}
where $\Theta_{\left\{ \Sigma_{t}\right\} }$ is Huisken's density
(see \cite{E,K} or Lemma \ref{Huisken's density} for the definition).
However, unlike \cite{LW}, where one can apply White's regularity
theorem (cf. \cite{Wh}) to conclude that the flow is smooth up to
time $T$, here we need an additional argument to handle possible
boundary singularities (see Section \ref{Proof of the Main Theorem}). 

The organization of the paper is as follows. In Section \ref{Preliminaries},
fundamental tools for dealing with surfaces with free boundary are
introduced. These include the parametrization of the tubular neighborhood
of a surface and the reflection principle for parabolic equations
with homogeneous Neumann boundary condition. In addition, the crucial
monotonicity formulas for MCF (cf. \cite{H2,E,B}) and the compactness
theorem in the free boundary setting are also included. In Section
\ref{Li-Wang Pseudolocality Theorem}, key estimates for MCF with
free boundary are derived. More precisely, Li-Wang's pseudolocality
theorem and small energy theorem are adapted to the free boundary
setting. Then Section \ref{Hypotheses} - Section \ref{Proof of the Main Theorem}
are devoted to the proof of the main theorem. Specifically, in Section
\ref{Hypotheses}, we formulate the assumptions and prove the condensation
compactness theorem for sequences of parabolic rescaling of the flow.
In Section \ref{Multiplicity One Convergence}, we follow \cite{LW}
to prove (\ref{unit Huisken's density}). Finally, in Section \ref{Proof of the Main Theorem}
we complete the proof. The arguments in Section \ref{Li-Wang Pseudolocality Theorem}
- Section \ref{Multiplicity One Convergence} parallel that of \cite{LW}.

\section*{Acknowledgement}

The author is grateful to Peter Sternberg for suggesting the problem,
for insightful discussions and for providing helpful comments on this
paper. 

\section{\label{Preliminaries}Preliminaries}

In this section we will set up the tools which are required for studying
MCF with free boundary. These can be roughly divided into four different
topics as follows.
\begin{enumerate}
\item Complementary balls, the $\kappa-$graph condition of surfaces, and
the parametrization of a tubular neighborhood.
\item Huisken's monotonicity formula for MCF and Buckland's adaption to
MCF with free boundary. 
\item Stahl's gradient estimates for MCF with free boundary, parametrization
of MCF near the free boundary, the reflection principle for parabolic
equations with Neumann boundary condition, and the $C^{2,1}$ estimates
for MCF with free boundary.
\item Topology and compactness of the space of MCF with free boundary.
\end{enumerate}
Note that the discussion of the above will be in the form of surfaces
in $\mathbb{R}^{3}$ for definiteness; however, most of them work
in higher dimensions as well. 

Let's begin with the first topic. Given an embedded, smooth surface
$\Gamma$ in $\mathbb{R}^{3}$, let $X$ be a point which lies in
a tubular neighborhood of $\Gamma$. Recall that the projection of
$X$ on $\Gamma$ is given by
\[
\mathring{X}\coloneqq X-d_{\Gamma}\left(X\right)\nabla d_{\Gamma}\left(X\right),
\]
where $d_{\Gamma}\left(X\right)=\textrm{dist}\left(X,\Gamma\right)$,
and the reflection of $X$ with respect to $\Gamma$ is given by
\[
\tilde{X}\coloneqq X-2d_{\Gamma}\left(X\right)\nabla d_{\Gamma}\left(X\right)=2\mathring{X}-X.
\]
Below we follow \cite{GJ} to define the complementary balls. 
\begin{defn}
\label{complementary ball}(Complementary Balls)

Let $U$ be an open, connected subset of $\mathbb{R}^{3}$ with an
embedded, smooth boundary surface $\Gamma$. Given $r>0$ and a point
$P\in U$ which lies in the tubular neighborhood of $\Gamma$, define
\[
\tilde{B}_{r}\left(P\right)=\left\{ X\in U\left|\,\tilde{X}\in B_{r}\left(P\right)\setminus U\right.\right\} ,
\]
where $B_{r}\left(P\right)$ is the open ball in $\mathbb{R}^{3}$
centered at $P$ with radius $r$. Note that $\tilde{B}_{r}\left(P\right)=\emptyset$
if $B_{r}\left(P\right)\subset U$.
\end{defn}

In other words, $\tilde{B}_{r}\left(P\right)$ is the reflection of
$B_{r}\left(P\right)\setminus U$ with respect to $\Gamma$. The prototype
of reflection is with respect to a plane. What follows is a simple
remark on the reflection and complementary balls with respect to a
plane.
\begin{rem}
\label{compensation}If 
\[
U=\mathbb{R}_{+}^{3}=\left\{ \left.\left(x_{1},x_{2},x_{3}\right)\right|\,x_{2}>0\right\} 
\]
is a half-space, then 
\[
\Gamma=\partial U=\left\{ \left.\left(x_{1},x_{2},x_{3}\right)\right|\,x_{2}=0\right\} =x_{1}x_{3}-\textrm{plane}
\]
and the reflection with respect to $\Gamma$ is given by
\[
\widetilde{\left(x_{1},x_{2},x_{3}\right)}=\left(x_{1},-x_{2},x_{3}\right).
\]
Let $\Sigma$ be a surface in $U$ which meets $\Gamma$ orthogonally.
Define $\tilde{\Sigma}$ to be the reflection of $\Sigma$ with respect
to $\Gamma$. Namely, 
\[
\tilde{\Sigma}=\left\{ X\left|\,\tilde{X}\in\Sigma\right.\right\} .
\]
Given $P\in\Sigma$ and $r>0$, note that $\Sigma\cap\tilde{B}_{r}\left(P\right)$
is the reflection of $\tilde{\Sigma}\cap B_{r}\left(P\right)$ with
respect to $\Gamma$, and hence
\[
\mathcal{H}^{2}\left(\Sigma\cap\tilde{B}_{r}\left(P\right)\right)=\mathcal{H}^{2}\left(\tilde{\Sigma}\cap B_{r}\left(P\right)\right).
\]
Let $\bar{\Sigma}=\Sigma\cup\tilde{\Sigma}$ be the extension of $\Sigma$
across $\Gamma$, then we have 
\[
\mathcal{H}^{2}\left(\bar{\Sigma}\cap B_{r}\left(P\right)\right)=\mathcal{H}^{2}\left(\Sigma\cap B_{r}\left(P\right)\right)+\mathcal{H}^{2}\left(\tilde{\Sigma}\cap B_{r}\left(P\right)\right)
\]
\[
=\mathcal{H}^{2}\left(\Sigma\cap B_{r}\left(P\right)\right)+\mathcal{H}^{2}\left(\Sigma\cap\tilde{B}_{r}\left(P\right)\right).
\]
In addition, suppose that $\Sigma$ is a $\delta-$Lipschitz graph
for some $\delta>0$, say 
\[
\Sigma=\left\{ \left.\left(x_{1},x_{2},u\left(x_{1},x_{2}\right)\right)\,\right|\,x_{1}\in\mathbb{R},\,x_{2}\geq0\right\} ,
\]
where the function satisfies $\partial_{2}u\left(x_{1},0\right)=0$
(since $\Sigma$ meets $\Gamma$ orthogonally) and $\left|\nabla u\right|\leq\delta$
(due to the $\delta-$Lipschitz graph condition). Then there holds
\[
\frac{\mathcal{H}^{2}\left(\Sigma\cap B_{r}\left(P\right)\right)+\mathcal{H}^{2}\left(\Sigma\cap\tilde{B}_{r}\left(P\right)\right)}{\pi r^{2}}\leq\sqrt{1+\delta^{2}}.
\]
\end{rem}

Note that the above area ratio estimate also holds for non-flat $\Gamma$,
as long as $r$ is sufficiently small (depending on the second fundamental
form of $\Gamma$ and $\Sigma$, see Lemma \ref{unit area ratio of surface}). 

We then proceed to introduce the $\kappa-$graph condition. Throughout
the paper, we use $\nu$ to denote the unit normal vector of $\Gamma$.
Moreover, if $\Gamma=\partial U$, where $U$ is an open, connected
subset of $\mathbb{R}^{3}$, then $\nu$ is assumed to be the inward
normal vector (i.e. pointing toward $U$). 
\begin{defn}
\label{kappa graph condition}($\kappa-$Graph Condition)

A properly embedded $C^{3,1}$ surface $\Gamma$ in $\mathbb{R}^{3}$
satisfies the ``$\kappa-$graph condition'' for some $\kappa>0$
if the following property holds. Choose an arbitrary point on $\Gamma$;
without loss of generality, we may assume that the chosen point is
$O$ and $\nu\left(O\right)=\left(0,1,0\right)$. Then $\Gamma\cap B_{\kappa^{-1}}\left(O\right)$
is a graph of a $C^{3,1}$ function $\varphi$ defined in $T_{O}\Gamma=x_{1}x_{3}-$plane,
i.e. $x_{2}=\varphi\left(x_{1},x_{3}\right)$ and 
\[
\varphi\left(0,0\right)=0,\quad\nabla\varphi\left(0,0\right)=\left(0,0\right).
\]
Moreover, the function $\varphi$ satisfies 
\[
\left|\nabla^{2}\varphi\right|\leq\kappa,\quad\left|\nabla^{3}\varphi\right|\leq\kappa^{2},\quad\left[\nabla^{3}\varphi\right]_{1}\leq\kappa^{3},
\]
where $\left[\cdot\right]_{1}$ is the Lipschitz norm defined by 
\[
\left[v\right]_{1}=\sup_{x\neq x'}\frac{\left|v\left(x\right)-v\left(x'\right)\right|}{\left|x-x'\right|}.
\]
Note that the mean value theorem then yields
\[
\left|\frac{\varphi\left(\xi\right)}{\xi}\right|\,+\,\left|\nabla\varphi\left(\xi\right)\right|\,\lesssim\,\kappa\left|\xi\right|
\]
for $\left|\xi\right|\lesssim\kappa^{-1}$.
\end{defn}

Any embedded, smooth surface $\Gamma$ would locally satisfy the above
condition. More precisely, given $R>0$, there exists $\kappa>0$
so that $\Gamma\cap B_{R}\left(O\right)$ satisfies the $\kappa-$graph
condition. Consequently, if the involved domain is bounded, then for
simplicity we may just assume that $\Gamma$ satisfies the $\kappa-$graph
condition for some $\kappa>0$. 

In follows we discuss the parametrization of a tubular neighborhood
of $\Gamma$. This will be used repeatedly to study the regularity
of surfaces near the boundary.
\begin{defn}
\label{tubular nbd}(Parametrization of Tubular Neighborhood) 

Let $U$ be an open, connected subset of $\mathbb{R}^{3}$ whose boundary
$\Gamma=\partial U$ satisfies the $\kappa-$graph condition for some
$\kappa>0$. Choose an arbitrary point on $\Gamma$; without loss
of generality, we may assume that the chosen point is $O$, $\nu\left(O\right)=\left(0,1,0\right)$,
and that $\Gamma\cap B_{\kappa^{-1}}\left(O\right)$ is defined by
$x_{2}=\varphi\left(x_{1},x_{3}\right)$ as in Definition \ref{kappa graph condition}.
Then we can parametrize $B_{\kappa^{-1}}\left(O\right)$ by 
\[
\left(x_{1},x_{2},x_{3}\right)=\Phi\left(y_{1},y_{2},y_{3}\right)\coloneqq\left(y_{1},\varphi\left(y_{1},y_{3}\right),y_{3}\right)+y_{2}\,\nu,
\]
where
\[
\nu=\nu\left(y_{1},y_{3}\right)=\frac{\left(-\partial_{1}\varphi,1,-\partial_{3}\varphi\right)}{\sqrt{1+\left|\nabla\varphi\right|^{2}}}
\]
is the inward, unit normal vector of $\Gamma\cap B_{\kappa^{-1}}\left(O\right)$.
Note that $\Phi$ is a $C^{2,1}$ diffeomorphism from a neighborhood
of $O$ onto $B_{\kappa^{-1}}\left(O\right)$ and that 
\[
U\cap B_{\kappa^{-1}}\left(O\right)=\Phi\left\{ y_{2}>0\right\} ,\quad\Gamma\cap B_{\kappa^{-1}}\left(O\right)=\Phi\left\{ y_{2}=0\right\} .
\]
Moreover, we have 
\[
\mathring{X}=\mathring{X}\left(y_{1},y_{3}\right)=\left(y_{1},\varphi\left(y_{1},y_{3}\right),y_{3}\right),\quad d_{\Gamma}\left(X\right)=y_{2},\quad\nabla d_{\Gamma}\left(X\right)=\nu\left(\mathring{X}\right),
\]
where $\mathring{X}$ is the projection of $X$ on $\Gamma$ and $d_{\Gamma}\left(X\right)=\textrm{dist}\left(X,\Gamma\right)$. 
\end{defn}

Following the notations in Definition \ref{tubular nbd}, we have
\[
\Phi\left(Y\right)-Y=\left(0,\varphi\left(y_{1},y_{3}\right),0\right)+y_{2}\left\{ \nu-\left(0,1,0\right)\right\} ,
\]
\[
\partial_{1}\Phi=\left(1,\,\partial_{1}\varphi,\,0\right)+y_{2}\,\partial_{1}\nu,\quad\partial_{2}\Phi=\nu\left(\mathring{X}\right),\quad\partial_{3}\Phi=\left(0,\,\partial_{3}\varphi,\,1\right)+y_{2}\,\partial_{3}\nu,
\]
which implies, by Definition \ref{kappa graph condition}, that $\Phi\left(O\right)=O$,
$D\Phi\left(O\right)=I$ (the identity $3\times3$ matrix), and
\begin{equation}
\left|\Phi\left(Y\right)-Y\right|\lesssim\kappa\left|Y\right|^{2},\quad\left|D\Phi\left(Y\right)-I\right|\lesssim\kappa\left|Y\right|\label{approximate identity}
\end{equation}
\[
\left|D^{2}\Phi\left(Y\right)\right|\lesssim\kappa,\quad\left[D^{2}\Phi\left(Y\right)\right]_{1}\lesssim\kappa^{2}.
\]
Notice that $\Phi$ not only maps $\left\{ y_{2}=0\right\} $ onto
$\Gamma\cap B_{\kappa^{-1}}\left(O\right)$ but also preserves their
unit normal vectors and the reflection with respect to them, i.e.
\[
\left(0,1,0\right)\,\overset{D\Phi}{\mapsto}\,\nu,
\]
\begin{equation}
\Phi\left(\tilde{Y}\right)=\widetilde{\Phi\left(Y\right)}\,\Leftrightarrow\,\Phi^{-1}\left(\tilde{X}\right)=\widetilde{\Phi^{-1}\left(X\right)},\label{reflection preserving}
\end{equation}
where $\tilde{Y}$ is the reflection of $Y$ with respect to the plane
$\left\{ y_{2}=0\right\} $ (see Remark \ref{compensation}) and $\widetilde{\Phi\left(Y\right)}$
is the reflection of $\Phi\left(Y\right)$ with respect to $\Gamma$;
$\tilde{X}$ and $\widetilde{\Phi^{-1}\left(X\right)}$ are defined
anaglously. Furthermore, let 
\begin{equation}
h_{ij}\left(Y\right)=\partial_{i}\Phi\left(Y\right)\cdot\partial_{j}\Phi\left(Y\right),\quad i,j\in\left\{ 1,2,3\right\} \label{pull-back metric}
\end{equation}
be the pull-back metric, and
\begin{equation}
\varGamma_{ij}^{k}\left(Y\right)=\frac{1}{2}h^{kl}\left(Y\right)\left\{ \partial_{i}h_{jl}\left(Y\right)+\partial_{j}h_{il}\left(Y\right)-\partial_{l}h_{ij}\left(Y\right)\right\} \label{connection}
\end{equation}
be the Levi-Civita connection, where $h^{kl}\left(Y\right)$ is the
inverse of $h_{kl}\left(Y\right)$. Then we have 
\[
\left|h_{ij}-\boldsymbol{\delta}_{ij}\right|\lesssim\kappa\left|Y\right|,\quad\left|\varGamma_{ij}^{k}\right|\lesssim\kappa,\quad\left[\varGamma_{ij}^{k}\right]_{1}\lesssim\kappa^{2},
\]
\begin{equation}
h_{22}\left(y_{1},y_{2},y_{3}\right)=1,\quad h_{12}\left(y_{1},0,y_{3}\right)=0=h_{32}\left(y_{1},0,y_{3}\right),\label{orthogonality}
\end{equation}
where $\boldsymbol{\delta}_{ij}$ is Kronecker delta.

We conclude the first topic with the following remark on the scaling
property of the parametrization in Definition \ref{tubular nbd}. 
\begin{rem}
\label{scale-preserving}Following the notations of Definition \ref{tubular nbd},
for each $\lambda>0$ we define $\Gamma^{\lambda}=\frac{1}{\lambda}\Gamma$.
Apparently $\Gamma^{\lambda}$ satisfies the $\lambda\kappa-$graph
condition and $\Gamma^{\lambda}\cap B_{\left(\lambda\kappa\right)^{-1}}\left(O\right)$
is a graph of
\[
\varphi^{\lambda}\left(y_{1},y_{3}\right)\coloneqq\frac{1}{\lambda}\,\varphi\left(\lambda y_{1},\lambda y_{3}\right).
\]
Let $\Phi^{\lambda}$ be the map corresponding to to $\Gamma^{\lambda}$
(as in Definition \ref{tubular nbd}), then it's not hard to see that
\[
X=\Phi\left(Y\right)\,\Leftrightarrow\,\frac{1}{\lambda}X=\Phi^{\lambda}\left(\frac{1}{\lambda}Y\right).
\]
In other words, the pull-back of $B_{\left(\lambda\kappa\right)^{-1}}\left(O\right)$
by $\Phi^{\lambda}$ corresponds to the pull-back of $B_{\kappa^{-1}}\left(O\right)$
by $\Phi$ via scaling. 
\end{rem}

Next, we begin the second topic with the fundamental monotonicity
formulas for MCF. The formula was first discovered by Huisken (cf.
\cite{H2}) for complete MCF, then Ecker found the localized version
(cf. \cite{E}). Afterward, Buckland modified the formula to work
for MCF with free boundary (cf. \cite{B}). 

Note that for a given a surface $\Sigma$, we denote its second fundamental
form by $A_{\Sigma}$, mean curvature by $H_{\Sigma}$ and unit normal
vector by $N_{\Sigma}$. 
\begin{lem}
\label{monotonicity formula for MCF}(Monotonicity Formulas for MCF)

Let $\left\{ \Sigma_{t}\right\} _{0\leq t<T}$ be a properly embedded
$C^{2}$ MCF moving freely in $U\subset\mathbb{R}^{3}$, where $T>0$
is a finite number. Assume that $\Gamma=\partial U$ satisfies the
$\kappa-$graph condition for some $\kappa>0$. Then we have the following.
\begin{itemize}
\item Given $P\in U$ and $0<r<\frac{1}{2\sqrt{5}}d_{\Gamma}\left(P\right)$,
there holds
\[
\frac{d}{dt}\int_{\Sigma_{t}}\psi_{r;P,T}\,\Psi_{P,T}\left(X,t\right)\,d\mathcal{H}^{2}\left(X\right)
\]
\[
\leq-\int_{\Sigma_{t}}\left(H_{\Sigma_{t}}-\nabla\ln\Psi_{P,T}\cdot N_{\Sigma_{t}}\right)^{2}\psi_{r;P,T}\,\Psi_{P,T}\left(X,t\right)\,d\mathcal{H}^{2}\left(X\right)
\]
for $T-r^{2}\leq t<T$, where 
\[
\psi_{r;P,T}\left(X,t\right)=\left(1-\frac{\left|X-P\right|^{2}-4\left(T-t\right)}{r^{2}}\right)_{+}^{3}
\]
is a localization function and 
\[
\Psi_{P,T}\left(X,t\right)=\frac{1}{4\pi\left(T-t\right)}\exp\left(-\frac{\left|X-P\right|^{2}}{4\left(T-t\right)}\right)
\]
is the backward heat kernel centered at $\left(P,T\right)$. Note
that on the right side of the formula, the first factor in the integrand
can be written explicitly as 
\[
H_{\Sigma_{t}}-\nabla\ln\Psi_{P,T}\cdot N_{\Sigma_{t}}=H_{\Sigma_{t}}+\frac{\left(X-P\right)\cdot N_{\Sigma_{t}}}{2\left(T-t\right)}.
\]
\item Given $P\in\Gamma$, there holds
\[
\frac{d}{dt}\int_{\Sigma_{t}}e^{85\left(\kappa^{2}\left(T-t\right)\right)^{\frac{2}{5}}}\eta_{\Gamma;P,T}\,\Psi_{\Gamma;P,T}\left(X,t\right)\,d\mathcal{H}^{2}\left(X\right)
\]
\[
\leq-\int_{\Sigma_{t}}\left(H_{\Sigma_{t}}-\nabla\ln\Psi_{\Gamma;P,T}\cdot N_{\Sigma_{t}}\right)^{2}\,e^{85\left(\kappa^{2}\left(T-t\right)\right)^{\frac{2}{5}}}\eta_{\Gamma;P,T}\,\Psi_{\Gamma;P,T}\left(X,t\right)\,d\mathcal{H}^{2}\left(X\right)
\]
for $\max\left\{ T-\frac{1}{2}\left(\frac{3}{320}\right)^{5}\kappa^{-2},0\right\} \leq t<T$,
where 
\[
\eta_{\Gamma;P,T}\left(X,t\right)=\left(1-\frac{\left|X-P\right|^{2}+\left|\tilde{X}-P\right|^{2}-80\left(T-t\right)}{\left(\frac{1}{2}\left(\kappa^{2}\left(T-t\right)\right)^{\frac{2}{5}}\kappa^{-1}\right)^{2}}\right)_{+}^{4}
\]
is a localization function and 
\[
\Psi_{\Gamma;P,T}\left(X,t\right)=\frac{1}{4\pi\left(T-t\right)}\exp\left(-\frac{\frac{1}{2}\left(\left|X-P\right|^{2}+\left|\tilde{X}-P\right|^{2}\right)}{4\left(1+16\left(\kappa^{2}\left(T-t\right)\right)^{\frac{2}{5}}\right)\left(T-t\right)}\right)
\]
is the modified backward heat kernel centered at $\left(P,T\right)$.
Note that $\tilde{X}$ is the reflection of $X$ with respect to $\Gamma$
and that on the right side of the formula, the first factor in the
integrand can be written explicitly as
\[
H_{\Sigma_{t}}-\nabla\ln\Psi_{\Gamma;P,T}\cdot N_{\Sigma_{t}}
\]
\[
=H_{\Sigma_{t}}+\frac{\left(X-P\right)\cdot N_{\Sigma_{t}}-\left(\left(X-P\right)\cdot\nabla d_{\Gamma}-d_{\Gamma}\right)\nabla d_{\Gamma}\cdot N_{\Sigma_{t}}-d_{\Gamma}\nabla^{2}d_{\Gamma}\left(X-P,N_{\Sigma_{t}}\right)}{2\left(1+16\left(\kappa^{2}\left(T-t\right)\right)^{\frac{2}{5}}\right)\left(T-t\right)}.
\]
\end{itemize}
\end{lem}

As a corollary, the integrals 
\[
\int_{\Sigma_{t}}\psi_{r;P,T}\,\Psi_{P,T}\,d\mathcal{H}^{2}\quad\textrm{for}\;P\in U
\]
and 
\[
\int_{\Sigma_{t}}e^{85\left(\kappa^{2}\left(T-t\right)\right)^{\frac{2}{5}}}\eta_{\Gamma;P,T}\,\Psi_{\Gamma;P,T}\,d\mathcal{H}^{2}\quad\textrm{for}\;P\in\Gamma
\]
are non-negative and non-increasing in $t$, so they must converge
as $t\nearrow T$. The limit is called Huisken's density of the flow
$\left\{ \Sigma_{t}\right\} $ at $\left(P,T\right)$, and it can
be shown to be independent of the choice $r$ (cf. \cite{E}). We
finish the second topic by listing some important properties of Huisken's
density as follows (cf. \cite{E,K}).
\begin{lem}
\label{Huisken's density}(Huisken's Density)

Following the notations in Lemma \ref{monotonicity formula for MCF},
let 
\[
\Theta_{\left\{ \Sigma_{t}\right\} }\left(P,T\right)=\left\{ \begin{array}{c}
\lim_{t\nearrow T}\int_{\Sigma_{t}}\psi_{r;P,T}\,\Psi_{P,T}\,d\mathcal{H}^{2},\quad\textrm{if}\;P\in U\\
\lim_{t\nearrow T}\int_{\Sigma_{t}}e^{85\left(\kappa^{2}\left(T-t\right)\right)^{\frac{2}{5}}}\eta_{\Gamma;P,T}\,\Psi_{\Gamma;P,T}\,d\mathcal{H}^{2},\quad\textrm{if}\;P\in\Gamma
\end{array}\right..
\]
If the flow $\left\{ \Sigma_{t}\right\} _{0\leq t<T}$ is regular
up to time $T$, let $\Sigma_{T}=\lim_{t\nearrow T}\,\Sigma_{t}$.
Then we have

\[
\Theta_{\left\{ \Sigma_{t}\right\} }\left(P,T\right)=\left\{ \begin{array}{c}
1,\quad\textrm{if}\quad P\in\Sigma_{T}\cap U\\
\frac{1}{2},\quad\textrm{if}\quad P\in\Sigma_{T}\cap\Gamma\\
0,\quad\textrm{if}\quad P\notin\Sigma_{T}
\end{array}\right..
\]
Moreover, $\Theta_{\left\{ \Sigma_{t}\right\} }\left(P,T\right)$
is upper semi-continuous. Namely, given a sequence $t_{i}\nearrow T$
and points $P_{i}\in\Sigma_{t_{i}}$ so that 
\begin{itemize}
\item $P_{i}\rightarrow P$ as $i\rightarrow\infty$,
\item If $P\in\Gamma$, then $P_{i}\in\Gamma$ for $i\gg1$,
\end{itemize}
then we have
\[
\Theta_{\left\{ \Sigma_{t}\right\} }\left(P,T\right)\geq\limsup_{i\rightarrow\infty}\,\Theta_{\left\{ \Sigma_{t}\right\} }\left(P_{i},t_{i}\right).
\]
 
\end{lem}

The main point of the third topic is about the $C^{2,1}$ estimates
for MCF with free boundary. Since the interior estimate can be found
in \cite{E}, our focus will be on the boundary regularity. To achieve
that, we will first use the map in Definition \ref{tubular nbd} to
pull back the flow to a half-space, where the pull-back flow has free
boundary on the boundary plane. Then we locally write the flow as
a graph of a time-dependent function, which satisfies a parabolic
equation and homogeneous Neumann boundary condition. Using the method
of reflection, we can extend the function across the boundary and
derive the boundary estimate. 

Throughout the paper, the following notation will be adopted. Given
a surface $\Sigma$ and a point $P\in\Sigma$, we denote by $\left(\Sigma\right)_{P}$
the path-connected component of $\Sigma$ containing $P$, i.e.
\begin{equation}
\left(\Sigma\right)_{P}=\left\{ Q\in\Sigma\left|\,\exists\,\,\textrm{continuous path}\,\gamma:\left[0,1\right]\rightarrow\Sigma\,\,\textrm{s.t.}\,\,\gamma\left(0\right)=P,\,\gamma\left(1\right)=Q\right.\right\} .\label{connected component}
\end{equation}
The following gradient estimate for MCF with free boundary is due
to Stahl, who derived the estimate by using the maximum principle
(cf. \cite{EH,S}).
\begin{lem}
\label{Stahl gradient}(Stahl's Gradient Estimate for MCF)

Given $\varepsilon>0$, there exists $\delta>0$ with the following
property. Let $\left\{ \Sigma_{t}\right\} _{a\leq t\leq b}$ be a
properly embedded $C^{2}$ MCF moving freely in $U\subset\mathbb{R}^{3}$,
where $a\leq0\leq b$ are constants. Assume that $O\in\Sigma_{0}$
and that 
\begin{itemize}
\item Either $B_{1}\left(O\right)\subset U$, or $O\in\Gamma$ and $\Gamma=\partial U$
satisfies the $\kappa-$graph condition for some $0<\kappa\leq1$;
\item There holds 
\[
\sup_{a\leq t\leq b}\left\Vert A_{\Sigma_{t}}\right\Vert _{L^{\infty}\left(B_{1}\left(O\right)\right)}\leq K
\]
for some $K\geq1$. 
\end{itemize}
Then we have 
\[
N_{\Sigma_{t}}\left(P\right)\cdot N_{\Sigma_{0}}\left(O\right)\,\geq\,1-\varepsilon
\]
for all $P\in\left(\Sigma_{t}\cap B_{\frac{\delta}{K}}\left(O\right)\right)_{O_{t}}$,
$\breve{a}\leq t\leq\breve{b}$, where 
\[
\breve{a}=\max\left\{ a,-\frac{\delta}{K^{2}}\right\} ,\quad\breve{b}=\min\left\{ b,\frac{\delta}{K^{2}}\right\} ,
\]
and $O_{t}$ is the normal trajectory of $O$ along the flow, i.e.
\[
O_{t}\in\Sigma_{t}\quad\textrm{and}\quad\partial_{t}O_{t}\,\bot\,T_{O_{t}}\Sigma_{t}.
\]
 
\end{lem}

As a consequence of the above lemma, near the point $O$ and around
the time $0$, the flow $\left\{ \Sigma_{t}\right\} $ is the graph
of a time-dependent function over $T_{O}\Sigma_{0}$ with small gradient.
However, if $O\in\Gamma$ and $U$ is not a cylinder, the domain of
the function might be time-dependent. To solve this issue, we will
use the map in Definition \ref{tubular nbd} to parametrize the tubular
neighborhood of $\Gamma$ so as to make the pull-back of the domain
(of the defining function) independent of time. Below we will carry
out the details of realizing this idea. 

Note that Einstein summation convention will be adopted throughout
the paper.
\begin{lem}
\label{parametrization of MCF near the boundary}(Parametrization
of MCF Near the Boundary)

Fix $\iota\in\left\{ 0,1\right\} $. Given $\varepsilon>0$, there
exists $\delta>0$ with the following property. Let $\left\{ \Sigma_{t}\right\} _{-1\leq t\leq\iota}$
be a properly embedded $C^{2}$ MCF moving freely in $U\subset\mathbb{R}^{3}$.
Assume that $O\in\Gamma=\partial U$, $\nu\left(O\right)=\left(0,1,0\right)$
and that 
\begin{itemize}
\item $\Gamma$ satisfies the $\kappa-$graph condition for some $0<\kappa\leq1$;
\item The second fundamental form of $\left\{ \Sigma_{t}\right\} $ satisfies
\[
\sup_{-1\leq t\leq\iota}\left\Vert A_{\Sigma_{t}}\right\Vert _{L^{\infty}\left(B_{1}\left(O\right)\right)}\leq K
\]
 for some $K\geq1$;
\item There exists $P\in\partial\Sigma_{0}\cap B_{\frac{\delta}{K}}\left(O\right)$
so that $N_{\Sigma_{0}}\left(P\right)\cdot\left(0,0,1\right)\geq1-\delta$.
\end{itemize}
Then $\left(\Sigma_{t}\cap B_{\frac{\varepsilon}{K}}\left(O\right)\right)_{P_{t}}$
can be parametrized by 
\[
X=\Phi\left(y_{1},y_{2},u\left(y_{1},y_{2},\,t\right)\right)
\]
 for 
\[
\sqrt{y_{1}^{2}+y_{2}^{2}}<r,\;y_{2}\geq0,\;-\frac{\delta}{K^{2}}\leq t\leq\frac{\delta}{K^{2}}\iota,
\]
where $P_{t}$ is the normal trajectory of $P$ along the flow (i.e.
$P_{t}\in\Sigma_{t}$ and $\partial_{t}P_{t}\perp T_{P_{t}}\Sigma_{t}$)
and $\Phi$ is the map in Definition \ref{tubular nbd}. 

Moreover, the function $u\left(y,t\right)$ satisfies
\begin{equation}
\partial_{t}u=g^{ij}\left(y,u,\nabla u\right)\partial_{ij}^{2}u+f\left(y,u,\nabla u\right),\label{eq of graph}
\end{equation}
\begin{equation}
\partial_{2}u\left(y_{1},0,t\right)=0,\label{homogeneous Neumann boundary condition}
\end{equation}
where $g^{ij}\left(y,u,\nabla u\right)$ is the inverse of $g_{ij}\left(y,u,\nabla u\right)$,
and $g_{ij}\left(y,u,\nabla u\right)$, $f\left(y,u,\nabla u\right)$
are defined in (\ref{induced metric}), (\ref{inhomogeneous term}),
respectively. 

Furthermore, we have the following estimates: 
\[
K\left|u\left(y,t\right)\right|\,+\,\left|\nabla u\left(y,t\right)\right|\leq\varepsilon,\quad\left|\nabla^{2}u\left(y,t\right)\right|\lesssim K,
\]
\[
\left|g^{ij}\left(y,u,\nabla u\right)-\boldsymbol{\delta}^{ij}\right|\leq\varepsilon,\quad\left[g^{ij}\left(y,u,\nabla u\right)\right]_{1,\textrm{spacial}}\lesssim K,
\]
\begin{equation}
\left.g^{12}\left(y,u,\nabla u\right)\right|_{x_{2}=0}=0,\label{reflection condition}
\end{equation}
\[
\left|f\left(y,u,\nabla u\right)\right|\lesssim K,\quad\left[f\left(y,u,\nabla u\right)\right]_{1,\textrm{spatial}}\lesssim K^{2},
\]
where $\left[\cdot\right]_{1,\textrm{spacial}}$ is the Lipschitz
norm with respect to the spacial variable defined by 
\[
\left[v\right]_{1,\textrm{spacial}}=\sup_{y\neq y'}\frac{\left|v\left(y,t\right)-v\left(y',t\right)\right|}{\left|y-y'\right|}.
\]
\end{lem}

\begin{proof}
By Lemma \ref{Stahl gradient} and the assumption regarding $P$ and
$N_{\Sigma_{0}}\left(P\right)$, if $\delta>0$ is sufficiently small,
the flow
\[
\left\{ \left(\Sigma_{t}\cap B_{\frac{\delta}{K}}\left(O\right)\right)_{P_{t}}\right\} _{-\frac{\delta}{K^{2}}\leq t\leq\frac{\delta}{K^{2}}\iota}
\]
are graphs with small gradients with respect to the direction $\left(0,0,1\right)$.
Using Definition \ref{tubular nbd} and (\ref{approximate identity}),
the preimage of the flow under $\Phi$ are also graphs with respect
to the direction $\left(0,0,1\right)$ and 
\begin{equation}
N_{\Sigma_{t}}\cdot\partial_{3}\Phi>0.\label{rectified graph condition}
\end{equation}
Consequently, the flow can be parametrized by 
\[
X\left(y,t\right)=\Phi\left(y,u\left(y,t\right)\right)
\]
for some function $u\left(y,t\right)$ which is defined on $\sqrt{y_{1}^{2}+y_{2}^{2}}<\frac{\delta}{K}$,
$y_{2}\geq0$, $-\frac{\delta}{K^{2}}\leq t\leq\frac{\delta}{K^{2}}\iota$
and satisfies 
\[
K\left|u\left(y,t\right)\right|\,+\,\left|\nabla u\left(y,t\right)\right|\leq\varepsilon,\quad K^{-1}\left|\nabla^{2}u\left(y,t\right)\right|\lesssim1.
\]
Note that $\delta>0$ is chosen so that that the above conditions
can be met. 

Next, in order to get more information on $u\left(y,t\right)$, let's
compute 
\[
\partial_{t}X=\partial_{t}u\,\partial_{3}\Phi\left(y,u\right);
\]
\[
\partial_{i}X=\partial_{i}\Phi\left(y,u\right)+\partial_{i}u\,\partial_{3}\Phi\left(y,u\right),\quad i\in\left\{ 1,2\right\} ;
\]
\[
D_{i}\partial_{j}X=\partial_{ij}^{2}\Phi+\partial_{ij}^{2}u\,\partial_{3}\Phi+\partial_{i}u\,\partial_{j3}^{2}\Phi+\partial_{j}u\,\partial_{i3}^{2}\Phi+\partial_{i}u\,\partial_{j}u\,\partial_{33}^{2}\Phi,\quad i,j\in\left\{ 1,2\right\} .
\]
Note that since $\Sigma_{t}$ is orthogonal to $\Gamma$ along the
boundary and $N_{\Sigma_{t}}\cdot\partial_{2}X=0$, we have
\[
0=\left.N_{\Sigma_{t}}\cdot\nu\right|_{\partial\Sigma_{t}}=\left.N_{\Sigma_{t}}\cdot\partial_{2}\Phi\right|_{\partial\Sigma_{t}}=-\partial_{2}u\left(y_{1},0,t\right)\,\left.N_{\Sigma_{t}}\cdot\partial_{3}\Phi\right|_{\partial\Sigma_{t}},
\]
which, together with (\ref{rectified graph condition}), yields (\ref{homogeneous Neumann boundary condition}).
Thus $u\left(y,t\right)$ satisfies the homogeneous Neumann boundary
condition. Moreover, the induced metric and the second fundamental
form of $\Sigma_{t}$ are respectively given by
\begin{equation}
g_{\Sigma_{t}}:\;g_{ij}\left(y,u,\nabla u\right)=\partial_{i}X\cdot\partial_{j}X\label{induced metric}
\end{equation}
\[
=h_{ij}\left(y,u\right)+h_{i3}\left(y,u\right)\partial_{j}u+h_{j3}\left(y,u\right)\partial_{i}u+h_{33}\left(y,u\right)\partial_{i}u\,\partial_{j}u
\]
and 
\[
A_{\Sigma_{t}}:\;A_{ij}\left(y,u,\nabla u,\nabla^{2}u\right)=D_{i}\partial_{j}X\cdot N_{\Sigma_{t}}
\]
\[
=\partial_{3}\Phi\cdot N_{\Sigma_{t}}\left\{ \varGamma_{ij}^{3}\left(y,u\right)+\partial_{ij}^{2}u+Q_{ij}\left(y,u,\nabla u\right)\right\} 
\]
for $i,j\in\left\{ 1,2\right\} $, where $h_{ij}$ and $\varGamma_{ij}^{k}$
are defined in (\ref{pull-back metric}) and (\ref{connection}),
respectively, and 
\begin{equation}
Q_{ij}\left(y,u,\nabla u\right)=\varGamma_{i3}^{3}\left(y,u\right)\,\partial_{j}u+\varGamma_{j3}^{3}\left(y,u\right)\,\partial_{i}u-\varGamma_{ij}^{k}\left(y,u\right)\,\partial_{k}u\label{quadratic term}
\end{equation}
\[
+\varGamma_{33}^{3}\left(y,u\right)\,\partial_{i}u\,\partial_{j}u-\varGamma_{i3}^{k}\left(y,u\right)\,\partial_{j}u\,\partial_{k}u-\varGamma_{j3}^{k}\left(y,u\right)\,\partial_{i}u\,\partial_{k}u-\varGamma_{33}^{k}\left(y,u\right)\,\partial_{i}u\,\partial_{j}u\,\partial_{k}u.
\]
Note that by (\ref{orthogonality}), (\ref{homogeneous Neumann boundary condition})
and (\ref{induced metric}), we get
\[
\left.g_{12}\left(y,u,\nabla u\right)\right|_{y_{2}=0}=0,
\]
from which, (\ref{reflection condition}) follows immediately. 

Using the MCF equation 
\[
\partial_{t}X\cdot N_{\Sigma_{t}}=H_{\Sigma_{t}}=g^{ij}\left(y,u,\nabla u\right)A_{ij},
\]
we derive 
\[
\partial_{t}u=g^{ij}\left(y,u,\nabla u\right)\partial_{ij}^{2}u+f\left(y,u,\nabla u\right)
\]
where 
\begin{equation}
f\left(y,u,\nabla u\right)=g^{ij}\left(y,u,\nabla u\right)\left\{ \varGamma_{ij}^{3}\left(y,u\right)+Q_{ij}\left(y,u,\nabla u\right)\right\} \label{inhomogeneous term}
\end{equation}
By (\ref{approximate identity}), (\ref{pull-back metric}), (\ref{connection})
and the $C^{2}$ estimate of $u\left(y,t\right)$, it's not hard to
see that
\[
\left|g^{ij}\left(y,u,\nabla u\right)-\boldsymbol{\delta}^{ij}\right|\leq\varepsilon,\quad K^{-1}\left[g^{ij}\left(y,u,\nabla u\right)\right]_{1,\textrm{spacial}}\lesssim1,
\]
and 
\[
K^{-1}\left|f\left(y,u,\nabla u\right)\right|+K^{-2}\left[f\left(y,u,\nabla u\right)\right]_{1,\textrm{spacial}}\lesssim1.
\]
\end{proof}
The following is the reflection principle for parabolic equations
with homogeneous Neumann boundary condition. The proof is omitted
since it is based on a simple calculation.
\begin{lem}
\label{reflection principle}(Reflection Principle)

Let $u\left(x,t\right)\in C^{2}\left(B_{r}^{+}\left(O\right)\times\left(-r^{2},0\right]\right)$
be a solution of 
\[
\partial_{t}u=a^{ij}\left(x,t\right)\partial_{ij}^{2}u+f\left(x,t\right),
\]
where $r>0$ is a constant, $B_{r}^{+}\left(O\right)=\left\{ \left(x_{1},x_{2}\right)\left|\,\sqrt{x_{1}^{2}+x_{2}^{2}}<r,\,x_{2}\geq0\right.\right\} $
is the upper half ball in $\mathbb{R}^{2}$, $\left\{ a^{ij}\left(x,t\right)\right\} _{i,j\in\left\{ 1,2\right\} }$
and $f\left(x,t\right)$ are continuous functions on $B_{r}^{+}\left(O\right)\times\left(-r^{2},0\right]$,
$a^{12}\left(x,t\right)=a^{21}\left(x,t\right)$.

Suppose that
\[
\partial_{2}u\left(x_{1},0,t\right)=0\quad\textrm{and}\quad a^{12}\left(x_{1},0,t\right)=0
\]
 for $x_{1}\in\left(-r,r\right),\,t\in\left(-r^{2},0\right]$. Define
the even extension of $u\left(x,t\right)$ as
\[
\bar{u}\left(x,t\right)=\left\{ \begin{array}{c}
u\left(x,t\right),\;x_{2}\geq0\\
u\left(\tilde{x},t\right),\;x_{2}<0
\end{array}\right.,
\]
where $\tilde{x}=\widetilde{\left(x_{1},x_{2}\right)}=\left(x_{1},-x_{2}\right)$
is the reflection of $x$ with respect to the line $\left\{ x_{2}=0\right\} $
in $\mathbb{R}^{2}$. Then $\bar{u}\left(x,t\right)\in C^{2}\left(B_{r}\left(O\right)\times\left(-r^{2},0\right]\right)$
and it satisfies 
\[
\partial_{t}\bar{u}=\bar{a}^{ij}\left(x,t\right)\partial_{ij}^{2}\bar{u}+\bar{f}\left(x,t\right),
\]
where 
\[
\bar{a}^{ij}\left(x,t\right)=\left\{ \begin{array}{c}
a^{ij}\left(x,t\right),\;x_{2}\geq0\\
\left(-1\right)^{i+j}a^{ij}\left(\tilde{x},t\right),\;x_{2}<0
\end{array}\right.,\quad i,j\in\left\{ 1,2\right\} ;
\]
\[
\bar{f}\left(x,t\right)=\left\{ \begin{array}{c}
f\left(x,t\right),\;x_{2}\geq0\\
f\left(\tilde{x},t\right),\;x_{2}<0
\end{array}\right..
\]
Note that $\left\{ \bar{a}^{ij}\left(x,t\right)\right\} _{i,j\in\left\{ 1,2\right\} }$,
$\bar{f}\left(x,t\right)$ are continuous on $B_{r}\left(O\right)\times\left(-r^{2},0\right]$
and that $\bar{a}^{12}\left(x,t\right)=\bar{a}^{21}\left(x,t\right)$.
\end{lem}

In the next lemma, we show how to use the given the spatial bounds
and the comparison principle to derive the H$\ddot{\textrm{o}}$lder
estimate with respect to the time variables (cf. \cite{An}).
\begin{lem}
\label{temporal regularity}(Regularity in the Time Variable)

Let $u\left(x,t\right)\in C^{2}\left(B_{r}\left(O\right)\times\left(-r^{2},0\right]\right)$
satisfy 
\[
\frac{\left|u\left(x,t\right)\right|}{r}\,+\,\left|\nabla u\left(x,t\right)\right|\,+\,r\left|\nabla^{2}u\left(x,t\right)\right|\lesssim1,
\]
where $r>0$ is a constant and $B_{r}\left(O\right)$ is the open
ball in $\mathbb{R}^{2}$ centered at $O$ with radius $r$. 

Suppose that $u\left(x,t\right)$ satisfies the equation
\[
\partial_{t}u=a^{ij}\left(x,t\right)\partial_{ij}^{2}u+f\left(x,t\right),
\]
where $\left\{ a^{ij}\left(x,t\right)\right\} _{i,j\in\left\{ 1,2\right\} }$
and $f\left(x,t\right)$ are continuous functions satisfying 
\[
\frac{1}{2}\boldsymbol{\delta}^{ij}\leq a^{ij}\left(x,t\right)\leq2\boldsymbol{\delta}^{ij},\quad a^{12}\left(x,t\right)=a^{21}\left(x,t\right),\quad r\left[a^{ij}\right]_{1,spacial}\lesssim1,
\]
\[
r\left|f\right|\,+\,r^{2}\left[f\right]_{1,\textrm{spacial}}\lesssim1
\]
on $B_{r}\left(O\right)\times\left(-r^{2},0\right]$ for $i,j\in\left\{ 1,2\right\} $,
where $\left[\cdot\right]_{1,\textrm{spacial}}$ is the Lipschitz
norm with respect to the spacial variable (see Lemma \ref{parametrization of MCF near the boundary})
. Then we have 
\[
\frac{\left|\nabla u\left(x,t\right)-\nabla u\left(x,t'\right)\right|}{\left|t-t'\right|^{\frac{1}{2}}}\lesssim\frac{1}{r}
\]
for $x\in B_{\frac{r}{4}}\left(O\right)$, $0<t-t'\ll r^{2}$.
\end{lem}

\begin{proof}
Fix $x_{0}\in B_{\frac{r}{4}}\left(O\right)$, $t_{0}<0$ with $\left|t_{0}\right|\ll r^{2}$
(to be determined), and $k\in\left\{ 1,2\right\} $. Below we will
show that 
\[
\left|\partial_{k}u\left(x_{0},t\right)-\partial_{k}u\left(x_{0},t_{0}\right)\right|\,\lesssim\,\frac{\sqrt{t-t_{0}}}{r}
\]
for $t_{0}<t\leq0$.

For each $0<h<\frac{r}{4}$, the difference quotient of $u\left(x,t\right)$
is defined by 
\[
\triangle_{k}^{h}u\left(x,t\right)=\frac{1}{h}\left(u\left(x+he_{k},t\right)-u\left(x,t\right)\right).
\]
Using the equation of $u\left(x,t\right)$ and the assumption that
\[
\left|\nabla u\left(x,t\right)\right|\,+\,r\left|\nabla^{2}u\right|\,+\,r\left[a^{ij}\right]_{1,\textrm{spacial}}\,+\,r^{2}\left[f\right]_{1,\textrm{spacial}}\lesssim1,
\]
there is $C>0$ so that 
\begin{equation}
\left|\triangle_{k}^{h}u\left(x,t\right)\right|\,\leq\,\frac{C}{4}\label{temporal regularity 1}
\end{equation}
and 
\begin{equation}
\partial_{t}\triangle_{k}^{h}u=a^{ij}\left(x+he_{k},t\right)\partial_{ij}^{2}\triangle_{k}^{h}u+\left(\triangle_{k}^{h}a^{ij}\right)\partial_{ij}^{2}u+\triangle_{k}^{h}f\label{temporal regularity 2}
\end{equation}
\[
\leq a^{ij}\left(x+he_{k},t\right)\partial_{ij}^{2}\triangle_{k}^{h}u+\frac{C}{r^{2}}.
\]
By the mean value theorem, we may assume that for the same constant
$C>0$, there holds
\begin{equation}
\triangle_{k}^{h}u\left(x,t_{0}\right)\,\leq\,\triangle_{k}^{h}u\left(x_{0},t_{0}\right)+\frac{C}{r}\left|x-x_{0}\right|\label{temporal regularity 3}
\end{equation}
\[
\leq\,\triangle_{k}^{h}u\left(x_{0},t_{0}\right)+\varepsilon+\frac{1}{4\varepsilon}\left(\frac{C}{r}\right)^{2}\left|x-x_{0}\right|^{2}\qquad\forall\,\varepsilon>0,
\]
in which we use Young's inequality to get the second line. 

For each $\varepsilon>0$, consider an auxiliary function 
\[
v_{\varepsilon}\left(x,t\right)\coloneqq\triangle_{k}^{h}u\left(x_{0},t_{0}\right)+\varepsilon+\frac{1}{4\varepsilon}\left(\frac{C}{r}\right)^{2}\left|x-x_{0}\right|^{2}+\left(\frac{2}{\varepsilon}\left(\frac{C}{r}\right)^{2}+\frac{C}{r^{2}}\right)\left(t-t_{0}\right).
\]
Obviously, by (\ref{temporal regularity 3}), we have $v_{\varepsilon}\left(x,t_{0}\right)\geq\triangle_{k}^{h}u\left(x,t_{0}\right)$
for $x\in B_{\frac{r}{2}}\left(x_{0}\right)$. Also, (\ref{temporal regularity 2})
and the condition that $a^{ij}\leq2\boldsymbol{\delta}^{ij}$ yield
\[
\left(\partial_{t}-a^{ij}\left(x+he_{k},t\right)\partial_{ij}^{2}\right)v_{\varepsilon}\,\geq\,\frac{C}{r^{2}}\,\geq\,\left(\partial_{t}-a^{ij}\left(x+he_{k},t\right)\partial_{ij}^{2}\right)\triangle_{k}^{h}u.
\]
In addition, by (\ref{temporal regularity 1}), for any $\left(x,t\right)$
satisfying $\left|x-x_{0}\right|=\frac{r}{2}$ and $t_{0}\leq t\leq0$,
we have 
\[
v_{\varepsilon}\left(x,t\right)\,\geq\,\triangle_{k}^{h}u\left(x_{0},t_{0}\right)+\varepsilon+\frac{1}{4\varepsilon}\left(\frac{C}{r}\right)^{2}\left|x-x_{0}\right|^{2}\,\geq\,\triangle_{k}^{h}u\left(x_{0},t_{0}\right)+\frac{C}{r}\left|x-x_{0}\right|
\]
\[
\geq\,-\frac{C}{4}+\frac{C}{r}\left|x-x_{0}\right|\,=\,\frac{C}{4}\,\geq\,\triangle_{k}^{h}u\left(x,t\right).
\]
The comparison principle for parabolic equations then gives
\[
\triangle_{k}^{h}u\left(x,t\right)\,\leq\,v_{\varepsilon}\left(x,t\right)
\]
for $\left|x-x_{0}\right|\leq\frac{r}{2},\,t_{0}\leq t\leq0$. In
particular, we get 
\begin{equation}
\triangle_{k}^{h}u\left(x_{0},t\right)\,\leq\,\triangle_{k}^{h}u\left(x_{0},t_{0}\right)+\varepsilon+\left(\frac{C}{r^{2}}+\frac{1}{\varepsilon}\left(\frac{C}{r}\right)^{2}\right)\left(t-t_{0}\right)\label{temporal regularity 4}
\end{equation}
for any $t\in\left[t_{0},0\right],\,\varepsilon>0$. 

For each $t\in\left(t_{0},t_{0}+\frac{r^{2}}{C^{2}}\right]$, choose
$\varepsilon=\frac{C}{r}\sqrt{t-t_{0}}$ , where $C>0$ is the aforementioned
constant. By (\ref{temporal regularity 4}), we get 
\[
\triangle_{k}^{h}u\left(x_{0},t\right)-\triangle_{k}^{h}u\left(x_{0},t_{0}\right)\,\leq\,2\frac{C}{r}\sqrt{t-t_{0}}+\frac{C}{r^{2}}\left(t-t_{0}\right)\,\leq\,\left(2C+\frac{1}{C}\right)\frac{\sqrt{t-t_{0}}}{r}.
\]
As $h\searrow0$, we get 
\[
\partial_{k}u\left(x_{0},t\right)-\partial_{k}u\left(x_{0},t_{0}\right)\,\leq\,\left(2C+\frac{1}{C}\right)\frac{\sqrt{t-t_{0}}}{r}.
\]
Lastly, choose $-\frac{r^{2}}{C^{2}}\leq t_{0}<0$ so that the above
hold for $t\in\left(t_{0},0\right]$. 

Replacing $u$ by $-u$ and following the same procedure gives
\[
-\partial_{k}u\left(x_{0},t\right)+\partial_{k}u\left(x_{0},t_{0}\right)\,\leq\,\left(2C+\frac{1}{C}\right)\frac{\sqrt{t-t_{0}}}{r}
\]
for $t_{0}<t\leq0$.
\end{proof}
With the help of Lemma \ref{reflection principle} and Lemma \ref{temporal regularity},
we can improve Lemma \ref{parametrization of MCF near the boundary}
as follows. 
\begin{prop}
\label{local graph thm}($C^{2,1}$ Estimates for MCF with Free Boundary)

The function $u\left(y,t\right)$ in Lemma \ref{parametrization of MCF near the boundary}
satisfies 
\[
K\left|u\left(y,t\right)\right|\ll1,\quad\left[\nabla^{2}u\right]_{1}+\left[\partial_{t}u\right]_{1}\,\lesssim\,K^{2}
\]
for $\left|x\right|+\left|t\right|^{\frac{1}{2}}\ll\frac{1}{K}$,
where $\left[\cdot\right]_{1}$ is the ``parabolic'' Lipschitz norm
defined by 
\[
\left[v\right]_{1}=\sup_{\left|y-y'\right|\,+\,\left|t-t'\right|>0}\,\,\frac{\left|v\left(y,t\right)-v\left(y',t'\right)\right|}{\left|y-y'\right|\,+\,\left|t-t'\right|^{\frac{1}{2}}}.
\]
\end{prop}

\begin{proof}
By (\ref{homogeneous Neumann boundary condition}) and (\ref{reflection condition}),
Lemma \ref{reflection principle} implies that the even extension
$\bar{u}\left(y,t\right)$ of $u\left(y,t\right)$ across $\left\{ y_{2}=0\right\} $
satisfies 
\[
\partial_{t}\bar{u}=\bar{g}^{ij}\left(y,t\right)\partial_{ij}^{2}\bar{u}+\bar{f}\left(x,t\right),
\]
where 
\[
\bar{g}^{ij}\left(y,t\right)=\left\{ \begin{array}{c}
g^{ij}\left(y,u\left(y,t\right),\nabla u\left(x,t\right)\right),\;y_{2}\geq0\\
\left(-1\right)^{i+j}g^{ij}\left(\tilde{y},u\left(\tilde{y},t\right),\nabla u\left(\tilde{y},t\right)\right),\;y_{2}<0
\end{array}\right.,\quad i,j\in\left\{ 1,2\right\} ;
\]
\[
\bar{f}\left(y,t\right)=\left\{ \begin{array}{c}
f\left(y,u\left(y,t\right),\nabla u\left(y,t\right)\right),\;y_{2}\geq0\\
f\left(\tilde{y},u\left(\tilde{y},t\right),\nabla u\left(\tilde{y},t\right)\right),\;y_{2}<0
\end{array}\right..
\]
Moreover, we have 
\[
K\left|\bar{u}\left(y,t\right)\right|\,+\,\left|\nabla\bar{u}\left(y,t\right)\right|\,+\,K^{-1}\left|\nabla^{2}\bar{u}\left(y,t\right)\right|\,\lesssim\,1,
\]
\[
\left|\bar{g}^{ij}-\boldsymbol{\delta}^{ij}\right|\ll1,\quad\left[\bar{g}^{ij}\right]_{1,spacial}\lesssim K,
\]
\[
\left|f\right|\lesssim K,\quad\left[f\right]_{1,spacial}\lesssim K^{2}.
\]
which, by Lemma \ref{temporal regularity}, implies that 
\[
\left|\nabla\bar{u}\left(y,t\right)-\nabla\bar{u}\left(y',t'\right)\right|\,\lesssim\,K\left(\left|y-y'\right|+\left|t-t'\right|^{\frac{1}{2}}\right).
\]
By definitions of $\bar{g}^{ij}\left(y,t\right)$ and $\bar{f}\left(y,t\right)$(and
interpolation inequalities), we then get
\[
\left|\bar{g}^{ij}\left(y,t\right)-\bar{g}^{ij}\left(y',t'\right)\right|\,\lesssim\,K\left(\left|y-y'\right|+\left|t-t'\right|^{\frac{1}{2}}\right),
\]
\[
\left|\bar{f}\left(y,t\right)-\bar{f}\left(y',t'\right)\right|\,\lesssim\,K^{2}\left(\left|y-y'\right|+\left|t-t'\right|^{\frac{1}{2}}\right).
\]
The conclusion follows immediately by applying Schauder estimates
to the equation of $\bar{u}\left(y,t\right)$.
\end{proof}
Lastly, we conclude this section with the fourth topic: the compactness
of the space of MCF with free boundary. Let's first give a definition
of the topology of the space of MCF with free boundary.
\begin{defn}
\label{topology}(Topology of the Space of MCF with Free Boundary)

Fix $\iota\in\left\{ 0,1\right\} $. Let $\left\{ \left\{ \Sigma_{t}^{i}\right\} _{a<t<b}\right\} _{i\in\mathbb{N}}$
be a sequence of MCF which moves freely in $\left\{ U_{i}\subset\mathbb{R}^{3}\right\} _{i\in\mathbb{N}}$
respectively. Also, let $\left\{ \Sigma_{t}\right\} _{a<t<b}$ be
a properly embedded $C^{2}$  MCF moving freely in $U\subset\mathbb{R}^{3}$.
We say that $\left\{ \Sigma_{t}^{i}\right\} $ $\textrm{converges}^{\iota}$
to $\left\{ \Sigma_{t}\right\} $ with finite multiplicity as $i\rightarrow\infty$
if the following hold.
\begin{enumerate}
\item $U_{i}$ converges to $U$ in the sense that $U$ consists of all
limit points of $U_{i}$ and $\partial U_{i}=\Gamma_{i}\stackrel{C^{3}}{\rightarrow}\Gamma=\partial U$
as $i\rightarrow\infty$;
\item For each $t_{*}\in\left(a,b\right)$, $\Sigma_{t_{*}}$ consists of
all limit points of $\left\{ \Sigma_{t_{*}}^{i}\right\} _{i\in\mathbb{N}}$.
In addition, for every $P\in\Sigma_{t_{*}}$, there exist $r>0$ and
$m\in\mathbb{N}$ so that
\end{enumerate}
\begin{itemize}
\item There is a $C^{2}$ diffeomorphism $\Phi$ from an open neighborhood
of $O$ in $\mathbb{R}^{3}$ onto $B_{r}\left(P\right)$ so that $\Phi^{-1}\left(\Sigma_{t}\cap B_{r}\left(P\right)\right)$
is a graph of a function $u\left(y,t\right)$ defined for $y\in\Omega\subset\mathbb{R}^{2}$,
$t\in\left[t_{*}-r^{2},t_{*}+\iota r^{2}\right]$, where 
\[
\Omega=\left\{ \begin{array}{c}
\left\{ \left(y_{1},y_{2}\right)\left|\,\sqrt{y_{1}^{2}+y_{2}^{2}}<r\right.\right\} ,\quad\textrm{if}\quad P\in U\\
\left\{ \left(y_{1},y_{2}\right)\left|\,\sqrt{y_{1}^{2}+y_{2}^{2}}<r,\,y_{2}\geq0\right.\right\} ,\quad\textrm{if}\quad P\in\Gamma
\end{array}\right..
\]
Note that if $P\in\Gamma$, we require $\Gamma\cap B_{r}\left(P\right)=\Phi\left\{ y_{2}=0\right\} $
and $\Phi^{-1}\left(\Sigma_{t}\cap B_{r}\left(P\right)\right)$ meets
$\left\{ y_{2}=0\right\} $ orthogonally.
\item For each $i\gg1$, there is a $C^{2}$ diffeomorphism $\Phi_{i}$
from an open neighborhood of $O$ in $\mathbb{R}^{3}$ onto $B_{r}\left(P\right)$
so that $\Phi_{i}^{-1}\left(\Sigma_{t}^{i}\cap B_{r}\left(P\right)\right)$
consists of $m$ graphs of functions
\[
\left\{ u_{i}^{1}\left(y,t\right),\cdots,u_{i}^{m}\left(y,t\right)\right\} 
\]
defined for $y\in\Omega$, $t_{*}-r^{2}\leq t\leq t_{*}+\iota r^{2}$;
moreover, if $P\in\Gamma$, $\Gamma_{i}\cap B_{r}\left(P\right)=\Phi_{i}\left\{ y_{2}=0\right\} $
and $\Phi_{i}^{-1}\left(\Sigma_{t}^{i}\cap B_{r}\left(P\right)\right)$
meets $\left\{ y_{2}=0\right\} $ orthogonally. Furthermore, as $i\rightarrow\infty$
we have 
\[
\Phi_{i}\stackrel{C^{2}}{\rightarrow}\Phi,
\]
\[
u_{i}^{j}\stackrel{C^{2}}{\rightarrow}u\qquad\forall\,j\in\left\{ 1,\cdots,m\right\} .
\]
\end{itemize}
\end{defn}

Before coming to the compactness theorem, let's introduce the following
notation in order to simplify the notation in the proof.
\begin{defn}
(Representation of a Local Graph)

Given a point $Q\in\mathbb{R}^{3}$, an orthonormal basis $\omega=\left\{ e_{1},e_{2},e_{3}\right\} $
in $\mathbb{R}^{3}$, and a function $u:\,\Omega\subset\mathbb{R}^{2}\rightarrow\mathbb{R}$,
let $\left[Q,\omega,u\right]$ be the graph of $u$ with respect to
the orientation $\omega$ and centered at $Q$. That is,
\[
\left[Q,\omega,u\right]=\left\{ \left.X\left(y_{1},y_{2}\right)=Q+y_{1}e_{1}+y_{2}e_{2}+u\left(y_{1},y_{2}\right)e_{3}\,\right|\left(y_{1},y_{2}\right)\in\Omega\right\} .
\]
\end{defn}

What follows is the the compactness theorem that will be used frequently
in the later sections (cf. \cite{PR}).
\begin{prop}
\label{compactness}(Compactness of the Space of MCF with Free Boundary)

Fix $\iota\in\left\{ 0,1\right\} $. Let $\left\{ \left\{ \Sigma_{t}^{i}\right\} _{a<t<b}\right\} _{i\in\mathbb{N}}$
be a sequence of connected, properly embedded $C^{2}$  MCF moving
freely in $\left\{ U_{i}\subset\mathbb{R}^{3}\right\} _{i\in\mathbb{N}}$
respectively. Suppose that
\begin{itemize}
\item There is $\kappa>0$ so that each $\Gamma_{i}=\partial U_{i}$ is
mean convex (i.e. the mean curvature vector of $\Gamma_{i}$ points
toward $U_{i}$) and satisfies the $\kappa-$graph condition. Note
that by passing to a subsequence, we may assume that $U_{i}\rightarrow U$
as $i\rightarrow\infty$ (in the sense that $U$ consists of all limit
points of $U_{i}$ and $\Gamma_{i}=\partial U_{i}\stackrel{C^{3}}{\rightarrow}\Gamma=\partial U$);
\item The areas and second fundamental forms of the flows are locally uniformly
$\textrm{bounded}^{\iota}$. Namely, for each $t_{*}\in\left(a,b\right)$
and $P\in\mathbb{R}^{3}$, there is $r>0$ so that for every $i\in\mathbb{N}$,
there holds
\[
\sup_{t_{*}-r^{2}\leq t\leq t_{*}+\iota r^{2}}\left(\mathcal{H}^{2}\left(\Sigma_{t}^{i}\cap B_{r}\left(P\right)\right)+\left\Vert A_{\Sigma_{t}^{i}}\right\Vert _{L^{\infty}\left(B_{r}\left(P\right)\right)}\right)\leq C\left(P,t_{*},r\right)<\infty;
\]
\item There is a time $t_{0}$ so that every subsequence of $\left\{ \Sigma_{t_{0}}^{i}\right\} _{i\in\mathbb{N}}$
has limit points which are not in $\Gamma$.
\end{itemize}
Then $\left\{ \Sigma_{t}^{i}\right\} _{t\geq t_{0}}$ $\textrm{converges}^{\iota}$
to a MCF moving freely in $U$ as $i\rightarrow\infty$ (in the sense
of Definition \ref{topology}).
\end{prop}

\begin{proof}
Due to the the uniform $\kappa-$graph condition (see Definition \ref{kappa graph condition}),
we may assume, by passing to a subsequence, that $U_{i}\rightarrow U$
as $i\rightarrow\infty$. Moreover, the convergence implies that $\Gamma$
is mean convex and satisfies the $\kappa-$graph condition as well.

Fix $P\in U\cup\Gamma$ and $t_{*}\in\mathbb{R}$. If $P$ is not
a limit point of $\left\{ \Sigma_{t_{*}}^{i}\right\} _{i\in\mathbb{N}}$,
by the local bound on the second fundamental form (which controls
the normal speed of the flow), we may assume, by passing to a subsequence,
that there is $r>0$ so that $\Sigma_{t}^{i}\cap B_{r}\left(P\right)=\emptyset$
for $t_{*}-r^{2}\leq t\leq t_{*}+\iota r^{2}$, $i\in\mathbb{N}$.
In case $P$ is a limit point of $\left\{ \Sigma_{t_{*}}^{i}\right\} _{i\in\mathbb{N}}$,
below we divide into two cases to consider: 
\begin{itemize}
\item \textit{Case 1}: $P\in U$;
\item \textit{Case 2}: $P\in\Gamma$.
\end{itemize}
$\mathbf{Case\;1}$ ($P\in U$): 

By passing to a subsequence, we may assume that there exist $P_{i}\in\Sigma_{t_{*}}^{i}\cap U_{i}$
which converge to $P$ as $i\rightarrow\infty$. The properly embeddedness
and the uniform bound on the second fundamental forms imply that there
exists $r>0$ with the following property. For each $i\in\mathbb{N}$,
there are a number of time-dependent local graphs 
\[
\left\{ \left[Q_{i}^{j},\:\omega_{i}^{j},\:u_{i}^{j}:B_{3r}\left(O\right)\times\left[t_{*}-9r^{2},t_{*}+9\iota r^{2}\right]\rightarrow\mathbb{R}\right]\right\} _{j=1,\cdots,m_{i}}
\]
whose union covers $\left\{ \Sigma_{t}^{i}\cap B_{r}\left(P_{i}\right)\right\} _{t_{*}-r^{2}\leq t\leq t_{*}+\iota r^{2}}$
and which are mutually disjoint in $B_{2r}\left(P_{i}\right)$ for
$t_{*}-4r^{2}\leq t\leq t_{*}+4\iota r^{2}$. The uniform bound for
the local areas yields that 
\[
\sup_{i\in\mathbb{N}}\,m_{i}\leq C\left(P,t_{*},r\right),
\]
so we may assume, by passing to a subsequence, that $m_{i}=m$ for
$i\in\mathbb{N}$. By the smooth estimate for MCF (see \cite{E} for
analogous results of Proposition \ref{local graph thm} for the interior
case), there holds
\[
\max_{1\leq j\leq m}\left\Vert u_{i}^{j}\right\Vert _{C^{2,1}}\leq C\left(P,t_{*},r\right).
\]
It follows, by passing to a subsequence, that 
\[
Q_{i}^{j}\rightarrow Q^{j},\quad\omega_{i}^{j}\rightarrow\omega^{j},\quad u_{i}^{j}\stackrel{C^{2}}{\rightarrow}u\qquad\forall\,j\in\left\{ 1,\cdots,m\right\} 
\]
as $i\rightarrow\infty$. Clearly, each limiting local graph $\left[Q^{j},\omega^{j},u^{j}\right]$
is a solution of MCF. Furthermore, every two limiting local graphs
must be either disjoint or identical by the strong maximum principle.

$\mathbf{Case\;2}$ ($P\in\Gamma$): 

Passing to a subsequence, there are $P_{i}\in\Sigma_{t_{*}}^{i}$
so that $P_{i}\rightarrow P$ as $i\rightarrow\infty$. Let $\mathring{P}_{i}$
be the closest point on $\Gamma_{i}$ to $P_{i}$, then we have $\left|\mathring{P}_{i}-P_{i}\right|\rightarrow0$
as $i\rightarrow\infty$ (since $P\in\Gamma$ and $\Gamma_{i}\rightarrow\Gamma$).
By the $\kappa-$graph condition, the uniform bound on the second
fundamental forms and Lemma \ref{parametrization of MCF near the boundary},
there exists $r>0$ with the following properties. For each $i\in\mathbb{N}$,
there is a diffeomorphism $\Phi_{i}$ (see Definition \ref{tubular nbd})
which maps from a neighborhood of $O$ in $\mathbb{R}^{3}$ onto $B_{3r}\left(\mathring{P}_{i}\right)$
and 
\[
U_{i}\cap B_{3r}\left(\mathring{P}_{i}\right)=\Phi_{i}\left\{ y_{2}>0\right\} ,\quad\Gamma_{i}\cap B_{3r}\left(\mathring{P}_{i}\right)=\Phi_{i}\left\{ y_{2}=0\right\} ;
\]
moreover, there are a number of time-dependent local graphs $\left\{ \left[Q_{i}^{j},\:\omega_{i}^{j},\:u_{i}^{j}\right]\right\} _{j=1,\cdots,m_{i}+m_{i}'},$
whose union covers $\left\{ \Phi_{i}^{-1}\left(\Sigma_{t}^{i}\cap B_{r}\left(P_{i}\right)\right)\right\} _{t_{*}-r^{2}\leq t\leq t_{*}+\iota r^{2}}$
and which are mutually disjoint in $B_{2r}\left(O\right)$ for $t_{*}-4r^{2}\leq t\leq t_{*}+4\iota r^{2}$.
The number $m_{i}$ and $m_{i}'$ means the following:
\begin{itemize}
\item For $j\in\left\{ 1,\cdots,m_{i}\right\} $, the function $u_{i}^{j}\left(x,t\right)$
is defined on $B_{3r}^{+}\left(O\right)\times\left[t_{*}-9r^{2},t_{*}+9\iota r^{2}\right]$
and $\Phi_{i}\left[Q_{i}^{j},\omega_{i}^{j},u_{i}^{j}\right]$ meets
$\Gamma_{i}$ orthogonally for $j\in\left\{ 1,\cdots,m_{i}\right\} $;
\item For $j\in\left\{ m_{i}+1,\cdots,m_{i}+m_{i}'\right\} $, the function
$u_{i}^{j}\left(x,t\right)$ is defined on $B_{3r}\left(O\right)\times\left[t_{*}-9r^{2},t_{*}+9\iota r^{2}\right]$
and $\Phi_{i}\left[Q_{i}^{j},\omega_{i}^{j},u_{i}^{j}\right]$ is
contained in $U_{i}$. 
\end{itemize}
The uniform bound on areas yields 
\[
\sup_{i\in\mathbb{N}}\,\left(m_{i}+m_{i}'\right)\leq C\left(P,t_{*},r,\kappa\right),
\]
so, by passing to a subsequence, we may assume that $m_{i}=m$ and
$m'_{i}=m'$ for $i\in\mathbb{N}$. In addition, Proposition \ref{local graph thm}
implies
\[
\max_{1\leq j\leq m+m'}\left\Vert u_{i}^{j}\right\Vert _{C^{2,1}}\leq C\left(P,t_{*},r,\kappa\right).
\]
It follows, by passing to a subsequence, that $\Phi_{i}\stackrel{C^{2}}{\rightarrow}\Phi$
and
\[
Q_{i}^{j}\rightarrow Q^{j},\quad\omega_{i}^{j}\rightarrow\omega^{j},\quad u_{i}^{j}\stackrel{C^{2}}{\rightarrow}u^{j},\qquad\forall\,j\in\left\{ 1,\cdots,m+m'\right\} 
\]
as $i\rightarrow\infty$. By the convergence, $\Phi$ is a local diffeomorphism
from a neighborhood of $O$ in $\mathbb{R}^{3}$ to $B_{2r}\left(P\right)$
and 
\[
\Phi^{-1}\left(U\cap B_{2r}\left(P\right)\right)\subset\left\{ \left.\left(y_{1},y_{2},y_{3}\right)\right|\,y_{2}>0\right\} ,
\]
\[
\Phi^{-1}\left(\Gamma\cap B_{2r}\left(P\right)\right)\subset\left\{ \left.\left(y_{1},y_{2},y_{3}\right)\right|\,y_{2}=0\right\} .
\]
By the comparison principle (more precisely, using (\ref{eq of graph}),
Lemma \ref{reflection principle} and \cite{Wa}), every two limiting
graphs must be either disjoint or identical. In addition, each $\Phi\left[Q^{j},\omega^{j},u^{j}\right]$
is clearly a solution of MCF. For $j\in\left\{ 1,\cdots,m\right\} $,
$\Phi\left[Q^{j},\omega^{j},u^{j}\right]$ meets $\Gamma$ orthogonally.
For $j\in\left\{ m+1,\cdots,m+m'\right\} $, due to the strong maximum
principle and the mean convexity of $\Gamma$, either $\Phi\left[Q^{j},\omega^{j},u^{j}\right]$
and $\Gamma$ are disjoint, or $\Phi\left[Q^{j},\omega^{j},u^{j}\right]\subset\Gamma$. 

Lastly, take a countable dense subset $\left\{ \left(P_{k},t_{k}\right)\right\} _{k\in\mathbb{N}}$
of $\left(U\cup\Gamma\right)\times\left[t_{0},\infty\right)$. Applying
the above argument for the sequence of points one by one successively
and using Cantor's diagonal argument, we can extract a subsequence
as claimed. Note that by the above argument, if the limiting flow
$\left\{ \Sigma_{t}\right\} $ intersects $\Gamma$ at interior points
for $t_{1}\geq t_{0}$, then $\Sigma_{t}\subset\Gamma$ by the strong
maximum principle for all $t\in\left[t_{0},t_{1}\right]$, which contradicts
with the assumption that $\Sigma_{t_{0}}$ is not contained in $\Gamma$.
\end{proof}

\section{\label{Li-Wang Pseudolocality Theorem}Area Ratio and Curvature Estimates
for MCF }

In this section we follow closely the ideas in \cite{LW} to estimate
the area ratio and second fundamental form of MCF with free boundary
and uniformly bounded mean curvature. It consists of the following
two parts.
\begin{enumerate}
\item The area ratio of a surface in a sufficiently small ball (which would
be modified near the boundary, see Lemma \ref{unit area ratio of surface}
for instance) stays close to one along MCF with free boundary and
uniformly bounded mean curvature (see Proposition \ref{unit area ratio of MCF}). 
\item Smallness of the $L^{2}$ norm of the initial second fundamental form
yields the bound on the second fundamental form for MCF with free
boundary and uniformly bounded mean curvature (see Proposition \ref{small energy implies regularity}).
\end{enumerate}
The proof of the first part proceeds as follows. We first show that
the modified area ratio for the initial surface $\Sigma_{0}$ is close
to one, provided that the radius is sufficiently small (depending
on the curvatures of $\Sigma$ and the boundary support surface $\Gamma$).
Then we show that the area of $\Sigma_{t}$ change slightly (forward
and backward) in time if the mean curvature stays uniformly bounded.
Lastly, by appealing to the monotonicity of area ratio for each $\Sigma_{t}$,
we show that the modified area ratio stays close to one along the
flow.

Recall that in (\ref{connected component}), we denote by $\left(\Sigma\right)_{P}$
the path-connected component of a surface $\Sigma$ containing $P$.
Also, $\tilde{B}_{r}\left(P\right)$ stands for the reflection of
$B_{r}\left(P\right)\setminus U$ with respect to $\Gamma$ (see Definition
\ref{complementary ball}). 
\begin{lem}
\label{unit area ratio of surface}Given $\varepsilon>0$, there exists
$\delta>0$ with the following property. Let $\Sigma$ be a properly
embedded $C^{2}$  surface in $U\subset\mathbb{R}^{3}$ which meets
$\Gamma=\partial U$ orthogonally. Suppose that
\begin{itemize}
\item Either $B_{1}\left(O\right)\subset U$ , or $O\in\Gamma$ and $\Gamma$
satisfies the $\kappa-$graph condition for some $0<\kappa\leq1$;
\item There is $K\geq1$ so that $\left(\Sigma\cap B_{\frac{\delta}{K}}\left(P\right)\right)_{P}$
is a $\delta-$Lipschitz graph for any $P\in\Sigma\cap B_{\frac{1}{2}}\left(O\right)$
(which holds, for instance, when $\left\Vert A_{\Sigma}\right\Vert _{L^{\infty}\left(B_{1}\left(O\right)\right)}\leq K$).
\end{itemize}
Then for any $P\in\Sigma\cap B_{\delta}\left(O\right)$ and $\rho\in\left[0,\frac{\delta}{K}\right)$,
we have
\[
\frac{\mathcal{H}^{2}\left(\Sigma\cap B_{\rho}\left(P\right)\right)_{P}+\mathcal{H}^{2}\left(\left(\Sigma\cap B_{\rho}\left(P\right)\right)_{P}\cap\tilde{B}_{\rho}\left(P\right)\right)}{\pi\rho^{2}}\,\leq\,1+\varepsilon.
\]
 
\end{lem}

\begin{proof}
Fix $P\in\Sigma$ and $\rho>0$. If $\rho\leq d_{\Gamma}\left(P\right)$,
then $\tilde{B}_{\rho}\left(P\right)=\emptyset$ (see Definition \ref{complementary ball})
and the conclusion follows directly from the small gradient graph
condition. In the case when $\rho>d_{\Gamma}\left(P\right)$, one
can first apply the map $\Phi$ in Definition \ref{tubular nbd} to
pull back $\Sigma$ to a half-space and then use (\ref{approximate identity}),
(\ref{reflection preserving}), Remark \ref{compensation} and the
small gradient graph condition to get the conclusion.
\end{proof}
Below we show that the area changes slightly along MCF within a time
which is inversely proportional to the mean curvature.
\begin{lem}
\label{pre unit area ratio preserving}Given $\varepsilon>0$, there
exists $\delta>0$ with the following property. Let $\left\{ \Sigma_{t}\right\} _{a\leq t\leq b}$
be a properly embedded $C^{2}$  MCF which moves freely in $U\subset\mathbb{R}^{3}$
and has the following parametrization:
\[
X_{t}=X\left(\cdot,t\right):M^{2}\times\left[a,b\right]\rightarrow U\subset\mathbb{R}^{3},\qquad\partial_{t}X_{t}=\overrightarrow{H}_{\Sigma_{t}},
\]
where $a\leq0\leq b$ are constants. Suppose that 
\begin{itemize}
\item Either $B_{1}\left(O\right)\subset U$ , or $O\in\Gamma=\partial U$
and $\Gamma$ satisfies the $\kappa-$graph condition for some $0<\kappa\leq1$;
\item There is $0<\Lambda\leq1$ so that $\sup_{a\leq t\leq b}\left\Vert H_{\Sigma_{t}}\right\Vert _{L^{\infty}\left(B_{1}\left(O\right)\right)}\leq\Lambda$.
\end{itemize}
Then for any $\rho\in\left(0,\delta\right]$, $t\in\left[\breve{a},\breve{b}\right]$
and $P\in\Sigma_{t}\cap B_{\delta}\left(O\right)$, we have
\[
\frac{\mathcal{H}^{2}\left(\Sigma_{t}\cap B_{\rho}\left(P\right)\right)_{P}}{\pi\rho^{2}}\,\leq\,e^{\Lambda^{2}\left|t\right|}\left(1+\frac{2\Lambda\left|t\right|}{\rho}\right)^{2}\,\frac{\mathcal{H}^{2}\left(\Sigma_{0}\cap B_{\rho+2\Lambda\left|t\right|}\left(P_{0}\right)\right)_{P_{0}}}{\pi\left(\rho+2\Lambda\left|t\right|\right)^{2}},
\]

\[
\frac{\mathcal{H}^{2}\left(\left(\Sigma_{t}\cap B_{\rho}\left(P\right)\right)_{P}\cap\tilde{B}_{\rho}\left(P\right)\right)}{\pi\rho^{2}}
\]
\[
\leq\,e^{\Lambda^{2}\left|t\right|}\left(1+\frac{2\left(1+\varepsilon\right)\Lambda\left|t\right|}{\rho}\right)^{2}\,\frac{\mathcal{H}^{2}\left(\left(\Sigma_{0}\cap B_{\rho+2\left(1+\varepsilon\right)\Lambda\left|t\right|}\left(P_{0}\right)\right)_{P_{0}}\cap\tilde{B}_{\rho+2\left(1+\varepsilon\right)\Lambda\left|t\right|}\left(P_{0}\right)\right)}{\pi\left(\rho+2\left(1+\varepsilon\right)\Lambda\left|t\right|\right)^{2}},
\]
where 
\[
\breve{a}=\max\left\{ a,\,-\frac{\delta}{\Lambda}\right\} ,\quad\breve{b}=\min\left\{ b,\,\frac{\delta}{\Lambda}\right\} ,\quad P_{0}=X_{0}\circ X_{t}^{-1}\left(P\right).
\]
\end{lem}

\begin{proof}
Given $t$ and $P\in\Sigma_{t}$, let $P_{0}=X_{0}\circ X_{t}^{-1}\left(P\right)$.
Since 
\[
\partial_{t}\left|X\left(p,t\right)-X\left(p,0\right)\right|\leq\Lambda,\quad\partial_{t}\left|X\left(p,t\right)-X\left(q,t\right)\right|\leq2\Lambda
\]
for any $p,q\in M$, we get $\left|P-P_{0}\right|\leq\Lambda\left|t\right|$
and
\[
X_{t}^{-1}\left(\Sigma_{t}\cap B_{\rho}\left(P\right)\right)_{P}\,\subset\,X_{0}^{-1}\left(\Sigma_{0}\cap B_{\rho+2\Lambda\left|t\right|}\left(P_{0}\right)\right)_{P_{0}}.
\]
Moreover, by the evolution formula $\partial_{t}\,d\mu_{t}=-H_{\Sigma_{t}}^{2}d\mu_{t}$,
where $d\mu_{t}$ is the induced measure of $\Sigma_{t}$ on $M$,
it follows that
\[
\mathcal{H}^{2}\left(\Sigma_{t}\cap B_{R}\left(P\right)\right)_{P}\,\leq\,\mathcal{H}^{2}\left(X_{t}\circ X_{0}^{-1}\left(\Sigma_{0}\cap B_{R+2\Lambda\left|t\right|}\left(P_{0}\right)\right)_{P_{0}}\right)
\]
\[
\leq e^{\Lambda^{2}\left|t\right|}\mathcal{H}^{2}\left(\Sigma_{0}\cap B_{R+2\Lambda\left|t\right|}\left(P_{0}\right)_{P_{0}}\right)
\]
This proves the first inequality.

Likewise, for the second inequality, it suffices to show that
\[
X_{t}^{-1}\left(\left(\Sigma_{t}\cap B_{\rho}\left(P\right)\right)_{P}\cap\tilde{B}_{\rho}\left(P\right)\right)\,\subset\,X_{0}^{-1}\left(\left(\Sigma_{0}\cap B_{\rho+2\left(1+\varepsilon\right)\Lambda\left|t\right|}\left(P_{0}\right)\right)_{P_{0}}\cap\tilde{B}_{\rho+2\left(1+\varepsilon\right)\Lambda\left|t\right|}\left(P_{0}\right)\right).
\]
To see that, fix $Q\in\left(\Sigma_{t}\cap B_{\rho}\left(P\right)\right)_{P}\cap\tilde{B}_{\rho}\left(P\right)$
and let $Q_{0}=X_{0}\circ X_{t}^{-1}\left(Q\right)$. Notice that
$\tilde{Q}\in B_{\rho}\left(P\right)$ (see Definition \ref{complementary ball})
and $Q_{0}\in\left(\Sigma_{0}\cap B_{\rho+2\Lambda\left|t\right|}\left(P_{0}\right)\right)_{P_{0}}$.
Given $\varepsilon>0$, by (\ref{approximate identity}) and (\ref{reflection preserving}),
there is $\delta>0$ so that
\[
\left|\widetilde{Q_{0}}-\tilde{Q}\right|\,\leq\,\left(1+\varepsilon\right)\left|Q_{0}-Q\right|
\]
as long as $P\in B_{\delta}\left(O\right)$, $0<\rho\leq\delta$ and
$\Lambda\left|t\right|\leq\delta$. It follows that 
\[
\widetilde{Q_{0}}\,\in\,B_{\rho+\left(1+\varepsilon\right)\Lambda\left|t\right|}\left(P\right)\,\subset\,B_{\rho+2\left(1+\varepsilon\right)\Lambda\left|t\right|}\left(P_{0}\right),
\]
which, by Definition \ref{complementary ball}, yields $Q_{0}\in\tilde{B}_{\rho+2\left(1+\varepsilon\right)\Lambda\left|t\right|}\left(P_{0}\right)$.
\end{proof}
The following is an immediate consequence of Lemma \ref{pre unit area ratio preserving}. 
\begin{cor}
\label{unit area ratio preserving}Given $\varepsilon>0$, there exists
$\delta>0$ with the following property. Let $\left\{ \Sigma_{t}\right\} _{a\leq t\leq b}$
be as in Lemma \ref{pre unit area ratio preserving}. Then for any
$\rho\in\left(0,\delta\right]$, $t\in\left[\breve{a},\breve{b}\right]$
and $P\in\Sigma_{t}\cap B_{\delta}\left(O\right)$, there holds
\[
\frac{\mathcal{H}^{2}\left(\Sigma_{t}\cap B_{\rho}\left(P\right)\right)_{P}+\mathcal{H}^{2}\left(\left(\Sigma_{t}\cap B_{\rho}\left(P\right)\right)_{P}\cap\tilde{B}_{\rho}\left(P\right)\right)}{\pi\rho^{2}}
\]
\[
\leq\left(1+\varepsilon\right)\frac{\mathcal{H}^{2}\left(\Sigma_{0}\cap B_{\left(1+\delta\right)\rho}\left(P_{0}\right)\right)_{P_{0}}+\mathcal{H}^{2}\left(\left(\Sigma_{0}\cap B_{\left(1+\delta\right)\rho}\left(P_{0}\right)\right)_{P_{0}}\cap\tilde{B}_{\left(1+\delta\right)\rho}\left(P_{0}\right)\right)}{\pi\left(\left(1+\delta\right)\rho\right)^{2}},
\]
where 
\[
\breve{a}=\max\left\{ a,\,-\frac{\delta\rho}{\Lambda}\right\} ,\quad\breve{b}=\min\left\{ b,\,\frac{\delta\rho}{\Lambda}\right\} ,\quad P_{0}=X_{0}\circ X_{t}^{-1}\left(P\right).
\]
\end{cor}

There is one drawback of Corollary \ref{unit area ratio preserving},
which is that the time-span $\left[\breve{a},\breve{b}\right]$ depends
on the radius $\rho$ and degenerates as $\rho\searrow0$. To fix
the problem, we appeal to the following lemma from \cite{GJ}. Loosely
speaking, it says that the modified area ratio of a surface with free
boundary is non-decreasing in radius provided that its mean curvature
(and also the curvature of the boundary support surface) is bounded.
\begin{lem}
\label{monotonicity of area ratio}(Monotonicity of Area Ratio)

There is a universal constant $C>0$ with the following property.
Let $\Sigma$ be a properly embedded $C^{2}$ surface in $U\subset\mathbb{R}^{3}$
which meets $\Gamma=\partial U$ orthogonally. Suppose that 
\begin{itemize}
\item Either $B_{1}\left(O\right)\subset U$, or $O\in\Gamma$ and $\Gamma$
satisfies the $\kappa-$graph condition for some $0<\kappa\leq1$;
\item There is $\Lambda\geq0$ so that $\left\Vert H_{\Sigma}\right\Vert _{L^{\infty}\left(B_{1}\left(O\right)\right)}\leq\Lambda$.
\end{itemize}
Then for any $P\in\Sigma\cap B_{1}\left(O\right)$, the function
\[
r\,\mapsto\,e^{C\left(\Lambda+\kappa\right)r}\left(\frac{\mathcal{H}^{2}\left(\Sigma\cap B_{r}\left(P\right)\right)_{P}+\mathcal{H}^{2}\left(\left(\Sigma\cap B_{r}\left(P\right)\right)_{P}\cap\tilde{B}_{r}\left(P\right)\right)}{\pi r^{2}}\right)
\]
is non-decreasing for $0<r\leq1-\left|P\right|$.
\end{lem}

Combining Lemma \ref{unit area ratio of surface}, Corollary \ref{unit area ratio preserving}
and Lemma \ref{monotonicity of area ratio}, we then get the following
area ratio estimates for MCF with free boundary and uniformly bounded
mean curvature.
\begin{prop}
\label{unit area ratio of MCF}(Modified Area Ratio Staying Close
to One)

Given $\varepsilon>0$, there exists $\delta>0$ with the following
property. Let $\left\{ \Sigma_{t}\right\} _{a\leq t\leq b}$ be a
properly embedded $C^{2}$  MCF in $U\subset\mathbb{R}^{3}$ with
free boundary on $\Gamma=\partial U$, where $a\leq0\leq b$ are constants.
Suppose that 
\begin{itemize}
\item Either $B_{1}\left(O\right)\subset U$, or $O\in\Gamma$ and $\Gamma$
satisfies the $\kappa-$graph condition for some $0<\kappa\leq1$;
\item There is $K\geq1$ so that $\left(\Sigma_{0}\cap B_{\frac{\delta}{K}}\left(P\right)\right)_{P}$
is a $\delta-$Lipschitz graph for any $P\in\Sigma_{0}\cap B_{\frac{1}{2}}\left(O\right)$
(which holds, for instance, when $\left\Vert A_{\Sigma_{0}}\right\Vert _{L^{\infty}\left(B_{1}\left(O\right)\right)}\leq K$);
\item There is $0<\Lambda\leq1$ so that $\sup_{a\leq t\leq b}\left\Vert H_{\Sigma_{t}}\right\Vert _{L^{\infty}\left(B_{1}\left(O\right)\right)}\leq\Lambda$.
\end{itemize}
Then for any $t\in\left[\breve{a},\breve{b}\right]$, $P\in\Sigma_{t}\cap B_{\delta}\left(O\right)$,
there holds 
\[
\sup_{0<r\leq\frac{\delta}{K}}\frac{\mathcal{H}^{2}\left(\Sigma_{t}\cap B_{r}\left(P\right)\right)_{P}+\mathcal{H}^{2}\left(\left(\Sigma_{t}\cap B_{r}\left(P\right)\right)_{P}\cap\tilde{B}_{r}\left(P\right)\right)}{\pi r^{2}}\leq1+\varepsilon,
\]
where $\breve{a}=\max\left\{ a,\,-\frac{\delta}{\Lambda K}\right\} $,
$\breve{b}=\min\left\{ b,\,\frac{\delta}{\Lambda K}\right\} $.
\end{prop}

Next, let's begin the second part of the section with the following
lemma. It says that given a complete minimal surface $\Sigma$ in
$\mathbb{R}_{+}^{3}$ with free boundary and bounded second fundamental
form, if the modified area ratio is sufficiently close to one, it
must be a half-plane (cf. \cite{LW}). 
\begin{lem}
\label{rigidity}There is a universal constant $\vartheta>0$ with
the following property. Let $\Sigma$ be a properly embedded $C^{2}$
minimal surface satisfying $\left\Vert A_{\Sigma}\right\Vert _{L^{\infty}}\leq K$
for some $K>0$. Suppose that
\begin{itemize}
\item Either $\Sigma$ is complete surface in $\mathbb{R}^{3}$ without
boundary and 
\[
\sup_{r>0}\frac{\mathcal{H}^{2}\left(\Sigma\cap B_{r}\left(P\right)\right)_{P}}{\pi r^{2}}\leq1+\vartheta
\]
for any $P\in\Sigma$;
\item Or $\Sigma$ is a complete surface in $\mathbb{R}_{+}^{3}$ with free
boundary on $\partial\mathbb{R}_{+}^{3}\simeq\mathbb{R}^{2}$ and
\[
\sup_{r>0}\frac{\mathcal{H}^{2}\left(\Sigma\cap B_{r}\left(P\right)\right)_{P}+\mathcal{H}^{2}\left(\left(\Sigma\cap B_{r}\left(P\right)\right)_{P}\cap\tilde{B}_{r}\left(P\right)\right)}{\pi r^{2}}\leq1+\vartheta
\]
for any $P\in\Sigma$.
\end{itemize}
Then $\Sigma$ is flat, i.e. $A_{\Sigma}\equiv0$. 
\end{lem}

\begin{proof}
It suffices to prove the first case (which is a lemma in \cite{LW}),
since the second case can be reduced to the first case by the method
of reflection (see Remark \ref{compensation}, (\ref{eq of graph})
and Lemma \ref{reflection principle}). Below we sketch the argument
in \cite{LW} for the convenience of the reader.

By rescaling, we may assume that $K=1$. For the sake of contradiction,
let's suppose that there is a sequence of non-flat, complete, properly
embedded $C^{2}$ minimal surfaces $\left\{ \Sigma_{i}\right\} _{i\in\mathbb{N}}$
satisfying $0<\left\Vert A_{\Sigma_{i}}\right\Vert _{L^{\infty}}\leq1$
and 
\[
\sup_{P\in\Sigma_{i},\,r>0}\frac{\mathcal{H}^{2}\left(\Sigma_{i}\cap B_{r}\left(P\right)\right)_{P}}{\pi r^{2}}\leq1+\frac{1}{i}
\]
for all $i$. For each $i\in\mathbb{N}$, choose $P_{i}\in\Sigma_{i}$
so that 
\[
2\left|A_{\Sigma_{i}}\left(P_{i}\right)\right|\geq\left\Vert A_{\Sigma_{i}}\right\Vert _{L^{\infty}}\coloneqq A_{i}>0.
\]
Let $\hat{\Sigma}_{i}=A_{i}\left(\Sigma_{i}-P_{i}\right)$, then $\hat{\Sigma}_{i}$
is a properly embedded $C^{2}$  minimal surface satisfying $\left\Vert A_{\hat{\Sigma}_{i}}\right\Vert _{L^{\infty}}\leq1$,
$\left|A_{\hat{\Sigma}_{i}}\left(O\right)\right|\geq\frac{1}{2}$
and 
\[
\sup_{r>0}\frac{\mathcal{H}^{2}\left(\hat{\Sigma}_{i}\cap B_{r}\left(O\right)\right)_{O}}{\pi r^{2}}\leq1+\frac{1}{i}.
\]
By the compactness theorem for the space of minimal surfaces, it follows
that a subsequence of $\left\{ \hat{\Sigma}_{i}\right\} _{i\in\mathbb{N}}$
converges in the $C^{2}$ topology to a complete minimal surface $\hat{\Sigma}$,
which satisfies $\left\Vert A_{\hat{\Sigma}_{i}}\right\Vert _{L^{\infty}}=1$,
$\left|A_{\hat{\Sigma}}\left(O\right)\right|\geq\frac{1}{2}$ and
\[
\sup_{r>0}\frac{\mathcal{H}^{2}\left(\hat{\Sigma}\cap B_{r}\left(O\right)\right)_{O}}{\pi r^{2}}\leq1
\]
By the monotonicity formula of minimal surfaces (cf. \cite{Al}),
it follows that $\left(\hat{\Sigma}\cap B_{1}\left(O\right)\right)_{O}$
must be a $C^{2}$ minimal cone (and hence a plane), which contradicts
with $\left|A_{\hat{\Sigma}}\left(O\right)\right|\geq\frac{1}{2}$.
\end{proof}
Now we are in a position to establish Li-Wang's pseudolocality theorem
for MCF (cf. \cite{LW}) in the free boundary setting. The proof is
based on a rescaling argument, with the help of Proposition \ref{unit area ratio of MCF}
and Lemma \ref{rigidity}. To simplify the notation in the proof,
let's denote the interior norm for the second fundamental form of
MCF by 
\begin{equation}
\left\llbracket A_{\left\{ \Sigma_{t}\right\} }\right\rrbracket _{B_{R}\left(O\right)\times B_{\rho}\left(0\right)}\label{interior norm}
\end{equation}
\[
=\sup\left\{ r\left|A_{\Sigma_{t_{0}}}\left(P\right)\right|\,:\,t_{0}\in B_{\rho}\left(0\right),\,P\in\Sigma_{t_{0}},\,B_{r}\left(P\right)\times B_{r^{2}}\left(t_{0}\right)\subset B_{R}\left(O\right)\times B_{\rho}\left(0\right)\right\} .
\]
\begin{prop}
\label{Li-Wang curvature estimate}(Li-Wang's Curvature Estimate for
MCF)

There exist $\delta>0$ and $C>0$ with the following property. Let
$\left\{ \Sigma_{t}\right\} _{-1\leq t\leq1}$ be a properly embedded
$C^{2}$  MCF in $U\subset\mathbb{R}^{3}$ with free boundary on $\Gamma=\partial U$.
Suppose that 
\begin{itemize}
\item Either $B_{1}\left(O\right)\subset U$, or $O\in\Gamma$ and $\Gamma$
is mean convex and satisfies the $\kappa-$graph condition for some
$0<\kappa\leq1$;
\item There is $K\geq1$ so that $\left(\Sigma_{0}\cap B_{\frac{\delta}{K}}\left(P\right)\right)_{P}$
is a $\delta-$Lipschitz graph for any $P\in\Sigma_{0}\cap B_{\frac{1}{2}}\left(O\right)$
(which holds, for instance, when $\left\Vert A_{\Sigma_{0}}\right\Vert _{L^{\infty}\left(B_{1}\left(O\right)\right)}\leq K$);
\item There is $0<\Lambda\leq1$ so that $\sup_{\left|t\right|\leq1}\left\Vert H_{\Sigma_{t}}\right\Vert _{L^{\infty}\left(B_{1}\left(O\right)\right)}\leq\Lambda$.
\end{itemize}
Then we have
\[
\left\llbracket A_{\left\{ \Sigma_{t}\right\} }\right\rrbracket _{B_{\frac{\delta}{K}}\left(O\right)\times B_{\min\left\{ 1,\,\frac{\delta}{\Lambda K}\right\} }\left(0\right)}\leq C;
\]
in particular, there holds
\[
\sup_{\left|t\right|\leq\min\left\{ \frac{1}{4},\,\frac{\delta}{4\Lambda K}\right\} }\left\Vert A_{\Sigma_{t}}\right\Vert _{L^{\infty}\left(B_{\frac{\delta}{2K}}\left(O\right)\right)}\leq\frac{2C}{\delta}K.
\]
\end{prop}

\begin{proof}
Let $\vartheta>0$ be the constant in Lemma \ref{rigidity} and $\delta>0$
be the constant in Proposition \ref{unit area ratio of MCF} with
the choice $\varepsilon=\vartheta$.

Suppose that the proposition does not hold. Then for each $i\in\mathbb{N}$,
we can find a MCF $\left\{ \Sigma_{t}^{i}\right\} _{-1\leq t\leq1}$
in $U_{i}$ with free boundary on $\Gamma_{i}=\partial U_{i}$ so
that it satisfies the hypotheses and
\[
C_{i}\coloneqq\left\llbracket A_{\left\{ \Sigma_{t}^{i}\right\} }\right\rrbracket _{B_{\frac{\delta}{K}}\left(O\right)\times B_{\min\left\{ 1,\,\frac{\delta}{\Lambda K}\right\} }\left(0\right)}\rightarrow\infty.
\]
For each $i\in\mathbb{N}$, choose $\left(P_{i},t_{i},r_{i}\right)$
so that $r_{i}\left|A_{\Sigma_{t}^{i}}\left(P_{i}\right)\right|\geq\frac{C_{i}}{2}$.
Let $A_{i}=\left|A_{\Sigma_{t}^{i}}\left(P_{i}\right)\right|$, then
we have $A_{i}\geq r_{i}A_{i}\rightarrow\infty$ as $i\rightarrow\infty$.
Moreover, by (\ref{interior norm}), we have 
\[
\sup_{\left|t-t_{i}\right|<\left(\frac{r_{i}}{2}\right)^{2}}\,\frac{r_{i}}{2}\left\Vert A_{\Sigma_{t}}\right\Vert _{L^{\infty}\left(B_{\frac{r_{i}}{2}}\left(P_{i}\right)\right)}\,\leq\,C_{i}\,\leq\,2r_{i}A_{i},
\]
which implies
\[
\sup_{\left|t-t_{i}\right|<\frac{r_{i}^{2}}{4}}\,\left\Vert A_{\Sigma_{t}}\right\Vert _{L^{\infty}\left(B_{\frac{r_{i}}{2}}\left(P_{i}\right)\right)}\,\leq\,4A_{i}.
\]
Here we have two cases to consider:
\begin{itemize}
\item \textit{Case 1}: $\limsup_{i\rightarrow\infty}A_{i}\,d_{\Gamma_{i}}\left(P_{i}\right)=\infty$;
\item \textit{Case 2}: $\limsup_{i\rightarrow\infty}A_{i}\,d_{\Gamma_{i}}\left(P_{i}\right)<R$
for some $R>0$.
\end{itemize}
$\mathbf{Case\;1}$ ($\limsup_{i\rightarrow\infty}A_{i}\,d_{\Gamma_{i}}\left(P_{i}\right)=\infty$):

By passing to a subsequence, we may assume that $\lim_{i\rightarrow\infty}A_{i}\,d_{\Gamma_{i}}\left(P_{i}\right)=\infty$.
Let 
\[
\hat{\Sigma}_{\tau}^{i}=\left(A_{i}\left(\Sigma_{t_{i}+\frac{\tau}{A_{i}^{2}}}^{i}-P_{i}\right)\cap B_{\frac{1}{2}r_{i}A_{i}}\left(O\right)\right)_{O_{\tau}},\quad\left|\tau\right|\leq\frac{1}{4}\left(r_{i}A_{i}\right)^{2},
\]
where $O_{\tau}$ is the ``normal trajectory'' of $O$ along the
flow at time $\tau$. Then we have 
\[
\left\Vert A_{\hat{\Sigma}_{\tau}^{i}}\right\Vert _{L^{\infty}}\leq4,\quad\left|A_{\hat{\Sigma}_{0}^{i}}\left(O\right)\right|=1,
\]
\[
\left\Vert H_{\hat{\Sigma}_{\tau}^{i}}\right\Vert _{L^{\infty}}\leq\frac{\Lambda}{A_{i}}\rightarrow0,
\]
\[
\sup_{\left|\tau\right|<\frac{1}{4}\left(r_{i}A_{i}\right)^{2}}\,\,\sup_{Q\in\hat{\Sigma}_{\tau}^{i}}\,\,\sup_{0<r<\textrm{dist}\left(Q,\,\partial\hat{\Sigma}_{\tau}^{i}\right)}\frac{\mathcal{H}^{2}\left(\hat{\Sigma}_{\tau}^{i}\cap B_{r}\left(Q\right)\right)_{Q}}{\pi r^{2}}\leq1+\vartheta,
\]
in which the last inequality comes from Proposition \ref{unit area ratio of MCF}.
It follows, by Proposition \ref{compactness}, that a subsequence
of the rescaled flows converges to a complete, properly embedded minimal
surface $\hat{\Sigma}$, which satisfies 
\[
\left\Vert A_{\hat{\Sigma}}\right\Vert _{L^{\infty}}\leq4,\quad\left|A_{\hat{\Sigma}}\left(O\right)\right|=1,
\]
\[
\sup_{Q\in\hat{\Sigma}}\,\,\sup_{r>0}\frac{\mathcal{H}^{2}\left(\hat{\Sigma}\cap B_{r}\left(Q\right)\right)_{Q}}{\pi r^{2}}\leq1+\vartheta.
\]
This contradicts Lemma \ref{rigidity}.

$\mathbf{Case\;2}$ ($\limsup_{i\rightarrow\infty}A_{i}\,d_{\Gamma_{i}}\left(P_{i}\right)<R$
for some $R>0$):

Let's first choose $\mathring{P}_{i}\in\Gamma_{i}$ so that $\left|\mathring{P}_{i}-P_{i}\right|=d_{\Gamma_{i}}\left(P_{i}\right)$.
For for $i\gg1$ (so that $r_{i}A_{i}>R$), let 
\[
\mathcal{P}_{i}=A_{i}\left(P_{i}-\mathring{P}_{i}\right)\in B_{R}\left(O\right);
\]
\[
\hat{\Sigma}_{\tau}^{i}=\left(A_{i}\left(\Sigma_{t_{i}+\frac{\tau}{A_{i}^{2}}}^{i}-\mathring{P}_{i}\right)\cap B_{\frac{1}{2}r_{i}A_{i}}\left(\mathcal{P}_{i}\right)\right)_{\mathcal{P}_{i}\left(\tau\right)},\quad\left|\tau\right|\leq\frac{1}{4}\left(r_{i}A_{i}\right)^{2};
\]
\[
\hat{U}_{i}=A_{i}\left(U_{i}-\mathring{P}_{i}\right),\quad\partial\hat{U}_{i}=\hat{\Gamma}_{i}=A_{i}\left(\Gamma_{i}-\mathring{P}_{i}\right).
\]
where $\mathcal{P}_{i}\left(\tau\right)$ is the ``normal trajectory''
of $\mathcal{P}_{i}$ along the flow at time $\tau$. Then $\hat{\Gamma}_{i}$
satisfies $\kappa_{i}-$graph condition, where $\kappa_{i}=\kappa A_{i}^{-1}\rightarrow0$,
and
\[
\left\Vert A_{\hat{\Sigma}_{\tau}^{i}}\right\Vert _{L^{\infty}}\leq4,\quad\left|A_{\hat{\Sigma}_{0}^{i}}\left(\mathcal{P}_{i}\right)\right|=1,
\]
\[
\left\Vert H_{\hat{\Sigma}_{\tau}^{i}}\right\Vert _{L^{\infty}}\leq\frac{\Lambda}{A_{i}}\rightarrow0,
\]
and
\[
\frac{\mathcal{H}^{2}\left(\hat{\Sigma}_{\tau}^{i}\cap B_{r}\left(Q\right)\right)_{Q}+\mathcal{H}^{2}\left(\left(\hat{\Sigma}_{\tau}^{i}\cap B_{r}\left(Q\right)\right)_{Q}\cap\tilde{B}_{r}\left(Q\right)\right)}{\pi r^{2}}\leq1+\vartheta
\]
for all $\left(Q,\tau,r\right)$ satisfying
\[
\left|\tau\right|<\frac{1}{4}\left(r_{i}A_{i}\right)^{2},\quad Q\in\hat{\Sigma}_{\tau}^{i},\quad B_{r}\left(Q\right)\subset B_{\frac{1}{2}r_{i}A_{i}}\left(\mathcal{P}_{i}\right).
\]
Note that the last inequality comes from Proposition \ref{unit area ratio of MCF}.
Passing to a subsequence, we may assume $\mathcal{P}_{i}\rightarrow\mathcal{P}\in B_{R}\left(O\right)$
and (by Proposition \ref{compactness}) that the rescaled flows converge
to a complete, properly embedded minimal surface $\hat{\Sigma}$ in
$\mathbb{R}_{+}^{3}$ with free boundary on $\partial\mathbb{R}_{+}^{3}\simeq\mathbb{R}^{2}$.
The limiting minimal surface satisfies 
\[
\left\Vert A_{\hat{\Sigma}}\right\Vert _{L^{\infty}}\leq4,\quad\left|A_{\hat{\Sigma}}\left(\mathcal{P}\right)\right|=1,
\]
\[
\sup_{Q\in\hat{\Sigma}}\,\,\sup_{r>0}\frac{\mathcal{H}^{2}\left(\hat{\Sigma}\cap B_{r}\left(Q\right)\right)_{Q}+\mathcal{H}^{2}\left(\left(\hat{\Sigma}\cap B_{r}\left(Q\right)\right)_{Q}\cap\tilde{B}_{r}\left(Q\right)\right)}{\pi r^{2}}\leq1+\vartheta,
\]
which contradicts Lemma \ref{rigidity}.
\end{proof}
To apply Proposition \ref{Li-Wang curvature estimate}, one of the
conditions to be satisfied is that we need to know in what scale can
we write the surface as a local graph with small gradient. Sometimes
this is not known in advance, especially when proving our main theorem.
Instead, we would like to replace this condition by the smallness
of $L^{2}$ norm of the second fundamental form, which is what we
called the small energy theorem. The key to making the transition
is through the following lemma and its corollary (see Corollary \ref{energy concentration coro}).
\begin{lem}
\label{energy concentration}(High-Curvature and Energy Concentration)

Given $K\geq5$, there exists $\epsilon>0$ with the following property.
Let $\left\{ \Sigma_{t}\right\} _{-1\leq t\leq1}$ be a properly embedded
$C^{2}$  MCF in $U\subset\mathbb{R}^{3}$ with free boundary on $\Gamma=\partial U$.
Suppose that 
\begin{itemize}
\item Either $B_{1}\left(O\right)\subset U$, or $O\in\Gamma$ and $\Gamma$
is mean convex and satisfies the $\kappa-$graph condition for some
$0<\kappa\leq1$;
\item There is $0<\Lambda\leq1$ so that $\sup_{\left|t\right|\leq1}\left\Vert H_{\Sigma_{t}}\right\Vert _{L^{\infty}\left(B_{1}\left(O\right)\right)}\leq\Lambda$;
\item There holds $\sup\left\{ r\left|A_{\Sigma_{0}}\left(P\right)\right|\,:\,P\in\Sigma_{0},\,B_{r}\left(P\right)\subset B_{1}\left(O\right)\right\} >K$.
\end{itemize}
Then we have 
\[
\int_{\Sigma_{0}\cap B_{1}\left(O\right)}\left|A_{\Sigma_{0}}\right|^{2}d\mathcal{H}^{2}>\epsilon.
\]
\end{lem}

\begin{proof}
Suppose the contrary. Then for each $i\in\mathbb{N}$, there is a
MCF $\left\{ \Sigma_{t}^{i}\right\} _{-1\leq t\leq1}$ in $U_{i}$
with free boundary on $\Gamma_{i}=\partial U_{i}$, which satisfies
the hypotheses and
\[
\int_{\Sigma_{0}^{i}\cap B_{1}\left(O\right)}\left|A_{\Sigma_{0}^{i}}\right|^{2}d\mathcal{H}^{2}\leq\frac{1}{i}.
\]
For each $i\in\mathbb{N}$, choose $P_{i}\in\Sigma_{0}^{i}$ and $0<r_{i}<1$
so that 
\[
r_{i}\left|A_{\Sigma_{0}^{i}}\left(P_{i}\right)\right|\geq\frac{1}{2}\sup\left\{ r\left|A_{\Sigma_{0}^{i}}\left(P\right)\right|\,:\,P\in\Sigma_{0}^{i},\,B_{r}\left(P\right)\subset B_{1}\left(O\right)\right\} >\frac{K}{2}.
\]
Note that $\left\Vert A_{\Sigma_{0}^{i}}\right\Vert _{L^{\infty}\left(B_{\frac{1}{2}r_{i}}\left(P_{i}\right)\right)}\,\leq\,4A_{i}$,
where $A_{i}=\left|A_{\Sigma_{0}^{i}}\left(P_{i}\right)\right|$. 

Let $0<\alpha\leq\frac{1}{4}$ be a small number to be determined.
By passing to a subsequence, we may assume that 
\begin{itemize}
\item Either $\limsup_{i\rightarrow\infty}A_{i}\,d_{\Gamma_{i}}\left(P_{i}\right)<\alpha$
(\textit{Case 1});
\item Or $\limsup_{i\rightarrow\infty}A_{i}\,d_{\Gamma_{i}}\left(P_{i}\right)\geq\alpha$
(\textit{Case 2}).
\end{itemize}
$\mathbf{Case\;1}$ ($\limsup_{i\rightarrow\infty}A_{i}\,d_{\Gamma_{i}}\left(P_{i}\right)<\alpha$):

Choose $\mathring{P}_{i}\in\Gamma_{i}$ so that $\left|\mathring{P}_{i}-P_{i}\right|=d_{\Gamma_{i}}\left(P_{i}\right)$.
Define
\[
\mathcal{P}_{i}=A_{i}\left(P_{i}-\mathring{P}_{i}\right),\quad\hat{\Sigma}_{\tau}^{i}=A_{i}\left(\Sigma_{\tau A_{i}^{-2}}^{i}-\mathring{P}_{i}\right),\quad\hat{U}_{i}=A_{i}\left(U_{i}-\mathring{P}_{i}\right).
\]
Then $\hat{\Gamma}_{i}\coloneqq\partial\hat{U}_{i}=A_{i}\left(\Gamma_{i}-\mathring{P}_{i}\right)$
satisfies $\kappa_{i}-$graph condition, where $\kappa_{i}=\kappa A_{i}^{-1}\leq1$,
and
\[
\left\Vert A_{\hat{\Sigma}_{0}^{i}}\right\Vert _{L^{\infty}\left(B_{\frac{1}{2}r_{i}A_{i}}\left(\mathcal{P}_{i}\right)\right)}\leq4,\quad\left|A_{\hat{\Sigma}_{0}^{i}}\left(\mathcal{P}_{i}\right)\right|=1,
\]
\[
\sup_{\left|\tau\right|\leq A_{i}^{2}}\left\Vert H_{\hat{\Sigma}_{\tau}^{i}}\right\Vert _{L^{\infty}\left(B_{\frac{1}{2}r_{i}A_{i}}\left(\mathcal{P}_{i}\right)\right)}\leq\frac{\Lambda}{A_{i}}\leq1,
\]
\[
\int_{\hat{\Sigma}_{0}^{i}\cap B_{\frac{1}{2}r_{i}A_{i}}\left(\mathcal{P}_{i}\right)}\left|A_{\hat{\Sigma}_{0}^{i}}\right|^{2}d\mathcal{H}^{2}\leq\frac{1}{i}.
\]
Note that $B_{\frac{1}{2}r_{i}A_{i}}\left(\mathcal{P}_{i}\right)\supset B_{1}\left(O\right)$
since $\frac{1}{2}r_{i}A_{i}\geq\frac{K}{4}>\frac{5}{4}$ and $\mathcal{P}_{i}\in B_{\alpha}\left(O\right)$.
By Proposition \ref{Li-Wang curvature estimate}, there exist universal
constants $\delta>0$ and $C>0$ so that 

\[
\sup_{\left|\tau\right|\leq\delta^{2}}\left\Vert A_{\hat{\Sigma}_{\tau}^{i}}\right\Vert _{L^{\infty}\left(B_{\delta}\left(O\right)\right)}\leq C.
\]
It follows, by Proposition \ref{compactness} and passing to a subsequence,
that $\mathcal{P}_{i}\rightarrow\mathcal{P}\in B_{2\alpha}\left(O\right)$
and $\left\{ \hat{\Sigma}_{\tau}^{i}\right\} \rightarrow\left\{ \hat{\Sigma}_{\tau}\right\} $.
The limiting MCF $\left\{ \hat{\Sigma}_{\tau}\right\} $ satisfies
\[
\int_{\hat{\Sigma}_{0}\cap B_{\delta}\left(O\right)}\left|A_{\hat{\Sigma}_{0}}\right|^{2}d\mathcal{H}^{2}=0,\quad\left|A_{\hat{\Sigma}_{0}}\left(\mathcal{P}\right)\right|=1,
\]
which is a contradiction if we choose $\alpha<\frac{\delta}{2}$. 

$\mathbf{Case\;2}$ ($\limsup_{i\rightarrow\infty}A_{i}\,d_{\Gamma_{i}}\left(P_{i}\right)\geq\alpha$):

Let 
\[
\hat{\Sigma}_{\tau}^{i}=\frac{A_{i}}{\alpha}\left(\Sigma_{\tau\left(\frac{\alpha}{A_{i}}\right)^{2}}^{i}-P_{i}\right).
\]
Then we have 
\[
\left\Vert A_{\hat{\Sigma}_{0}^{i}}\right\Vert _{L^{\infty}\left(B_{\frac{1}{2}r_{i}\frac{A_{i}}{\alpha}}\left(O\right)\right)}\leq4\alpha,\quad\left|A_{\hat{\Sigma}_{0}^{i}}\left(O\right)\right|=\alpha,
\]
\[
\sup_{\left|\tau\right|\leq\left(\frac{A_{i}}{\alpha}\right)^{2}}\left\Vert H_{\hat{\Sigma}_{\tau}^{i}}\right\Vert _{L^{\infty}\left(B_{\frac{1}{2}r_{i}\frac{A_{i}}{\alpha}}\left(O\right)\right)}\leq\frac{\alpha\Lambda}{A_{i}}\leq1,
\]
\[
\int_{\hat{\Sigma}_{0}^{i}\cap B_{\frac{1}{2}r_{i}\frac{A_{i}}{\alpha}}\left(O\right)}\left|A_{\hat{\Sigma}_{0}^{i}}\right|^{2}d\mathcal{H}^{2}\leq\frac{1}{i}.
\]
Proposition \ref{Li-Wang curvature estimate} then implies that there
exist $\delta>0$ and $C>0$ so that 

\[
\sup_{\left|\tau\right|\leq\delta^{2}}\left\Vert A_{\hat{\Sigma}_{\tau}^{i}}\right\Vert _{L^{\infty}\left(B_{\delta}\left(O\right)\right)}\leq C.
\]
It follows, by Proposition \ref{compactness}, that a subsequence
of the flows converges to a limiting MCF $\left\{ \hat{\Sigma}_{\tau}\right\} $,
which satisfies
\[
\int_{\hat{\Sigma}_{0}\cap B_{\delta}\left(O\right)}\left|A_{\hat{\Sigma}_{0}}\right|^{2}d\mathcal{H}^{2}=0,\quad\left|A_{\hat{\Sigma}_{0}}\left(O\right)\right|=\alpha,
\]
so we get a contradiction. 
\end{proof}
Setting $K=5$ in Lemma \ref{energy concentration}, we then get the
following corollary. 
\begin{cor}
\label{energy concentration coro}There exists $\epsilon>0$ with
the following property. Let $\left\{ \Sigma_{t}\right\} _{-1\leq t\leq1}$
be a properly embedded $C^{2}$  MCF in $U\subset\mathbb{R}^{3}$
with free boundary on $\Gamma=\partial U$. Suppose that 
\begin{itemize}
\item Either $B_{1}\left(O\right)\subset U$, or $O\in\Gamma$ and $\Gamma$
is mean convex and satisfies the $\kappa-$graph condition for some
$\kappa\leq1$;
\item There is $0<\Lambda\leq1$ so that $\sup_{\left|t\right|\leq1}\left\Vert H_{\Sigma_{t}}\right\Vert _{L^{\infty}\left(B_{1}\left(O\right)\right)}\leq\Lambda$;
\item There holds $\int_{\Sigma_{0}\cap B_{1}\left(O\right)}\left|A_{\Sigma_{0}}\right|^{2}d\mathcal{H}^{2}\leq\epsilon$.
\end{itemize}
Then we have 
\[
\sup\left\{ r\left|A_{\Sigma_{0}}\left(P\right)\right|\,:\,P\in\Sigma_{0},\,B_{r}\left(P\right)\subset B_{1}\left(O\right)\right\} \leq5,
\]
which, in particular, implies $\left\Vert A_{\Sigma_{0}}\right\Vert _{L^{\infty}\left(B_{\frac{1}{2}}\left(O\right)\right)}\leq10.$
\end{cor}

Thanks to Corollary \ref{energy concentration coro}, Proposition
\ref{Li-Wang curvature estimate} can be improved as follows (cf.
\cite{LW}).
\begin{prop}
\label{small energy implies regularity}(Li-Wang's Small Energy Theorem)

There exist $\epsilon>0$ and $C>0$ with the following property.
Let $\left\{ \Sigma_{t}\right\} _{-1\leq t\leq1}$ be a properly embedded
$C^{2}$  MCF in $U\subset\mathbb{R}^{3}$ with free boundary on $\Gamma=\partial U$.
Suppose that 
\begin{itemize}
\item Either $B_{1}\left(O\right)\subset U$, or $O\in\Gamma$ and $\Gamma$
is mean convex and satisfies the $\kappa-$graph condition for some
$0<\kappa\leq1$;
\item There holds $\int_{\Sigma_{0}\cap B_{1}\left(O\right)}\left|A_{\Sigma_{0}}\right|^{2}d\mathcal{H}^{2}\leq\epsilon$;
\item There is $0<\Lambda\leq1$ so that $\sup_{\left|t\right|\leq1}\left\Vert H_{\Sigma_{t}}\right\Vert _{L^{\infty}\left(B_{1}\left(O\right)\right)}\leq\Lambda$.
\end{itemize}
Then we have 
\[
\sup_{\left|t\right|\leq\min\left\{ \frac{1}{4},\,\frac{\epsilon}{4\Lambda}\right\} }\left\Vert A_{\Sigma_{t}}\right\Vert _{L^{\infty}\left(B_{\frac{\epsilon}{2}}\left(O\right)\right)}\leq C.
\]
\end{prop}

\section{\label{Hypotheses}Hypotheses}

In this section we will specify the hypotheses of our main theorem.
From now on, let $\left\{ \boldsymbol{\Sigma}_{t}\right\} _{0\leq t<T}$
be a compact, embedded $C^{2}$  MCF in $\boldsymbol{U}\subset\mathbb{R}^{3}$
with free boundary on $\boldsymbol{\Gamma}=\partial\boldsymbol{U}$,
where $T>0$ is a finite constant. We assume that 
\begin{itemize}
\item $\boldsymbol{\Gamma}=\partial\boldsymbol{U}$ is a properly embedded
$C^{3,1}$ surface which satisfies the $\boldsymbol{\kappa}-$graph
condition for some $\boldsymbol{\kappa}>0$ (see Definition \ref{kappa graph condition})
and is mean convex, i.e. 
\begin{equation}
H_{\boldsymbol{\Gamma}}=-\nabla_{\boldsymbol{\Gamma}}\cdot\boldsymbol{\nu}\geq0,\label{mean convex barrier}
\end{equation}
where $\boldsymbol{\nu}$ is the inward, unit normal vector of $\boldsymbol{\Gamma}$; 
\item The mean curvature of $\left\{ \boldsymbol{\Sigma}_{t}\right\} _{0\leq t<T}$
is uniformly bounded, i.e. 
\begin{equation}
\sup_{0\leq t<T}\left\Vert H_{\boldsymbol{\Sigma}_{t}}\right\Vert _{L^{\infty}}\leq\boldsymbol{\Lambda}<\infty;\label{mean curvature bound}
\end{equation}
\item The perimeter of $\left\{ \boldsymbol{\Sigma}_{t}\right\} _{0\leq t<T}$
is uniformly bounded, i.e. 
\begin{equation}
\sup_{0\leq t<T}\mathcal{H}^{1}\left(\partial\boldsymbol{\Sigma}_{t}\right)\leq\boldsymbol{l}<\infty.\label{boundary length bound}
\end{equation}
\end{itemize}
Note that in order to distinguish from the generic MCF that appeared
in the previous sections, we use boldface to denote the specific MCF
in the main theorem. Also, in the proof we only need $\boldsymbol{\Gamma}$
to satisfy the $\boldsymbol{\kappa}-$graph condition in the region
where the flow exists (i.e. the support of the flow). Actually, the
support of the flow is bounded as $\boldsymbol{\Sigma}_{0}$ is compact
and its mean curvature stays uniformly bound. Since every properly
embedded $C^{3,1}$ surface locally satisfies the $\boldsymbol{\kappa}-$graph
condition (with $\boldsymbol{\kappa}$ depending on the given bounded
region), one can regard this condition as a byproduct of the other
conditions. 

The goal of this paper is to show that the second fundamental form
of $\left\{ \boldsymbol{\Sigma}_{t}\right\} _{0\leq t<T}$ is uniformly
bounded; whence, by \cite{S} the flow can be extended (see Theorem
\ref{main thm}). The proof begins in this section and will be completed
in Section \ref{Proof of the Main Theorem}. For the rest of this
section, we will first show that the $L^{2}$ norm of the second fundamental
form is uniformly bounded. Then we will use that to prove the condensation
compactness theorem for sequences of parabolic rescaling of the flow.
\begin{lem}
(Uniformly Bounded Energy)

The $L^{2}$ norm of the second fundamental form is uniformly bounded.
More precisely, there holds
\begin{equation}
\int_{\boldsymbol{\Sigma}_{t}}\left|A_{\boldsymbol{\Sigma}_{t}}\right|^{2}\,d\mathcal{H}^{2}\,\leq\,C\left(\boldsymbol{\Lambda},\boldsymbol{\kappa},\boldsymbol{l},\mathcal{H}^{2}\left(\boldsymbol{\Sigma}_{0}\right),\chi\left(\boldsymbol{\Sigma}_{0}\right)\right)\label{energy bound}
\end{equation}
for $0\leq t<T$.
\end{lem}

\begin{proof}
Using the normal parametrization of the flow, i.e. 
\[
X_{t}=X\left(\cdot,t\right):M^{2}\times\left[0,T\right)\rightarrow\boldsymbol{U}\subset\mathbb{R}^{3},\qquad\partial_{t}X_{t}=\overrightarrow{H}_{\boldsymbol{\Sigma}_{t}},
\]
we have
\[
\partial_{t}\,d\boldsymbol{\mu}_{t}=-H_{\boldsymbol{\Sigma}_{t}}^{2}d\boldsymbol{\mu}_{t},
\]
where $d\boldsymbol{\mu}_{t}$ is the pull-back measure of $\boldsymbol{\Sigma}_{t}$
on $M$. It follows that
\begin{equation}
\sup_{0\leq t<T}\mathcal{H}^{2}\left(\boldsymbol{\Sigma}_{t}\right)\leq\mathcal{H}^{2}\left(\boldsymbol{\Sigma}_{0}\right)<\infty.\label{area bound}
\end{equation}
Next, for each $P\in\boldsymbol{U}\cup\boldsymbol{\Gamma}$, let 
\[
r_{P}=\left\{ \begin{array}{c}
d_{\boldsymbol{\Gamma}}\left(P\right),\quad\textrm{if}\quad P\in\boldsymbol{U}\\
\frac{1}{\sqrt{2}}\left(\frac{3}{320}\right)^{\frac{5}{2}}\boldsymbol{\kappa}^{-1},\quad\textrm{if}\quad P\in\boldsymbol{\Gamma}
\end{array}\right..
\]
Lemma \ref{monotonicity formula for MCF} yields the following area
ratio estimate (cf. \cite{E,K}):
\begin{equation}
\frac{\mathcal{H}^{2}\left(\boldsymbol{\Sigma}_{t}\cap B_{R\sqrt{T-t}}\left(P\right)\right)}{R^{2}\left(T-t\right)}\leq C\left(\boldsymbol{\kappa}\right)\frac{\mathcal{H}^{2}\left(\boldsymbol{\Sigma}_{T-\frac{5}{14}r_{P}^{2}}\cap B_{r_{P}}\left(P\right)\right)}{r_{P}^{2}}\leq C\left(\boldsymbol{\kappa},\mathcal{H}^{2}\left(\boldsymbol{\Sigma}_{0}\right),r_{P}\right)\label{area ratio bound}
\end{equation}
for $R>1$, $T-\frac{1}{10}\left(\frac{r_{P}}{R}\right)^{2}<t<T$.

By the Gauss-Bonnet theorem, for each $t\in\left[0,T\right)$ we have
\[
\int_{\boldsymbol{\Sigma}_{t}}K_{\boldsymbol{\Sigma}_{t}}\,d\mathcal{H}^{2}+\int_{\gamma_{t}}\vec{k}_{\gamma_{t}}\cdot\boldsymbol{\nu}\,d\mathcal{H}^{1}=2\pi\chi\left(\boldsymbol{\Sigma}_{t}\right),
\]
where $K_{\boldsymbol{\Sigma}_{t}}$ is the Gauss curvature of $\boldsymbol{\Sigma}_{t}$,
\[
\gamma_{t}\coloneqq\partial\boldsymbol{\Sigma}_{t}=\boldsymbol{\Sigma}_{t}\cap\boldsymbol{\Gamma}
\]
is the boundary curve, $\vec{k}_{\gamma_{t}}=D_{T_{\gamma_{t}}}T_{\gamma_{t}}$
is the curvature vector of $\gamma_{t}$ in $\mathbb{R}^{3}$, $T_{\gamma_{t}}$
is the unit tangent vector of $\gamma_{t}$, and $\chi\left(\boldsymbol{\Sigma}_{t}\right)$
is the Euler characteristic of $\boldsymbol{\Sigma}_{t}$. Note that
\[
K_{\boldsymbol{\Sigma}_{t}}=\frac{1}{2}\left(H_{\boldsymbol{\Sigma}_{t}}^{2}-\left|A_{\boldsymbol{\Sigma}_{t}}\right|^{2}\right),
\]
\[
\vec{k}_{\gamma_{t}}\cdot\boldsymbol{\nu}=D_{T_{\gamma_{t}}}T_{\gamma_{t}}\cdot\boldsymbol{\nu}=-T_{\gamma_{t}}\cdot D_{T_{\gamma_{t}}}\boldsymbol{\nu}=A_{\boldsymbol{\Gamma}}\left(T_{\gamma_{t}},T_{\gamma_{t}}\right),
\]
\[
\chi\left(\boldsymbol{\Sigma}_{t}\right)=\chi\left(\boldsymbol{\Sigma}_{0}\right).
\]
It follows from (\ref{mean curvature bound}), (\ref{boundary length bound})
and (\ref{area bound}) that 
\[
\int_{\boldsymbol{\Sigma}_{t}}\left|A_{\boldsymbol{\Sigma}_{t}}\right|^{2}\,d\mathcal{H}^{2}=\int_{\boldsymbol{\Sigma}_{t}}H_{\boldsymbol{\Sigma}_{t}}^{2}\,d\mathcal{H}^{2}+2\int_{\gamma_{t}}A_{\boldsymbol{\Gamma}}\left(T_{\gamma_{t}},T_{\gamma_{t}}\right)\,d\mathcal{H}^{1}-4\pi\chi\left(\boldsymbol{\Sigma}_{0}\right)
\]
\[
\lesssim\boldsymbol{\Lambda}^{2}\mathcal{H}^{2}\left(\boldsymbol{\Sigma}_{0}\right)+2\boldsymbol{\kappa}\boldsymbol{l}-4\pi\chi\left(\boldsymbol{\Sigma}_{0}\right).
\]
\end{proof}
The following remark points out that the energy concentrates at only
finitely many points; hence there are at most finitely many singularities
at time $T$.
\begin{rem}
\label{isolated singularities}Let 
\[
\omega_{t}=\left|A_{\boldsymbol{\Sigma}_{t}}\right|^{2}\lfloor\,d\mathcal{H}^{2}.
\]
By (\ref{energy bound}) and compactness, given any sequence $t_{i}\nearrow T$,
there is a subsequence (still denoted by $\left\{ t_{i}\right\} $
for simplicity of notations) so that $\omega_{t_{i}}\rightharpoonup\omega$
in the sense of Radon measure in $\mathbb{R}^{3}$. It follows from
the regularity of Radon measures that $\omega_{t_{i}}\left(\mathcal{B}\right)\rightarrow\omega\left(\mathcal{B}\right)$
for any bounded Borel set $\mathcal{B}$ satisfying $\omega\left(\partial\mathcal{B}\right)=0$.
As a result (together with the local finiteness of $\omega$), we
have 
\[
\omega_{t_{i}}\left(B_{r}\left(P\right)\right)\rightarrow\omega\left(B_{r}\left(P\right)\right)
\]
for every $P\in\mathbb{R}^{3}$ and almost every $r>0$. Let's define
\begin{equation}
\mathcal{S}=\left\{ P\in\mathbb{R}^{3}\left|\,\lim_{i\rightarrow\infty}\omega_{t_{i}}\left(B_{r}\left(P\right)\right)\geq\epsilon\quad\forall\;r>0\right.\right\} ,\label{singularities}
\end{equation}
where $\epsilon$ is the constant in Proposition \ref{small energy implies regularity}.
Note that 
\begin{equation}
\mathcal{H}^{0}\left(\mathcal{S}\right)\leq\frac{1}{\epsilon}\,\limsup_{t\nearrow T}\int_{\boldsymbol{\Sigma}_{t}}\left|A_{\boldsymbol{\Sigma}_{t}}\right|^{2}\,d\mathcal{H}^{2}\leq C\left(\boldsymbol{\Lambda},\boldsymbol{\kappa},\boldsymbol{l},\mathcal{H}^{2}\left(\boldsymbol{\Sigma}_{0}\right),\chi\left(\boldsymbol{\Sigma}_{0}\right)\right)\label{number of singularities}
\end{equation}
by (\ref{energy bound}). By Proposition \ref{small energy implies regularity}
and Proposition \ref{local graph thm}, one can conclude that for
any limit point $P$ of $\left\{ \boldsymbol{\Sigma}_{t}\right\} $
(as $t\nearrow T$) which is not in $\mathcal{S}$, there exists $r>0$
so that $\left\{ \boldsymbol{\Sigma}_{t}\cap B_{r}\left(P\right)\right\} $
is $C^{2,1}$ up to time $T$. Therefore, there are at most finitely
many singularities of $\left\{ \boldsymbol{\Sigma}_{t}\right\} $
as $t\nearrow T$ . 

Below we consider the parabolic rescaling of $\left\{ \boldsymbol{\Sigma}_{t}\right\} $
about a singular point on the boundary and prove the condensation
compactness theorem.
\end{rem}

\begin{prop}
\label{condensation compactness for MCF}(Condensation Compactness)

Given $P\in\mathcal{S}\cap\boldsymbol{\Gamma}$ and a sequence$\left\{ \lambda_{i}\searrow0\right\} _{i\in\mathbb{N}}$.
Let 
\begin{equation}
\boldsymbol{\Sigma}_{\tau}^{\left(P,T\right),\lambda_{i}}=\frac{1}{\lambda_{i}}\left(\boldsymbol{\Sigma}_{T+\lambda_{i}^{2}\tau}-P\right),\quad-\frac{T}{\lambda_{i}^{2}}\leq\tau<0\label{parabolic rescaling}
\end{equation}
and 
\[
\boldsymbol{U}^{P,\lambda_{i}}=\frac{1}{\lambda_{i}}\left(\boldsymbol{U}-P\right),\quad\boldsymbol{\Gamma}^{P,\lambda_{i}}=\partial\boldsymbol{U}^{P,\lambda_{i}}=\frac{1}{\lambda_{i}}\left(\boldsymbol{\Gamma}-P\right).
\]
Then there exist a half plane $\Pi$ which meets $\lim_{i\rightarrow\infty}\boldsymbol{\Gamma}^{P,\lambda_{i}}\simeq\mathbb{R}^{2}$
orthogonally, a finite set $\mathfrak{S}_{P}\subset\Pi$, and an integer
$m\in\mathbb{N}$ so that a subsequence of (\ref{parabolic rescaling})
converges to $\left\{ \Pi\right\} _{-\infty<\tau<0}$ with multiplicity
$m$ away from $\mathfrak{S}_{P}\times\left(-\infty,0\right)$. Moreover,
we have 
\[
\mathcal{H}^{0}\left(\mathfrak{S}_{P}\right)\leq C\left(\boldsymbol{\Lambda},\boldsymbol{\kappa},\boldsymbol{l},\mathcal{H}^{2}\left(\boldsymbol{\Sigma}_{0}\right),\chi\left(\boldsymbol{\Sigma}_{0}\right)\right).
\]

An analogous result holds for $P\in\mathcal{S}\cap\boldsymbol{U}$,
in which case $\Pi$ is a plane. 
\end{prop}

\begin{proof}
Throughout the proof, we will assume that $P\in\mathcal{S}\cap\boldsymbol{\Gamma}$.
The result for $P\in\mathcal{S}\cap\boldsymbol{U}$ follows from the
same argument. 

Firstly, note that $\left\{ \boldsymbol{\Sigma}_{\tau}^{\left(P,T\right),\lambda_{i}}\right\} $
is a MCF in $\boldsymbol{U}^{P,\lambda_{i}}$ with free boundary on
$\boldsymbol{\Gamma}^{P,\lambda_{i}}$, which is mean convex and satisfies
$\lambda_{i}\boldsymbol{\kappa}-$graph condition. Also, by (\ref{area ratio bound}),
(\ref{mean curvature bound}) and (\ref{energy bound}), we have
\begin{equation}
\sup_{R>1}\,\sup_{-\frac{1}{10}\left(\frac{r_{p}}{\lambda_{i}R}\right)^{2}<\tau<0}\frac{\mathcal{H}^{2}\left(\boldsymbol{\Sigma}_{\tau}^{\left(P,T\right),\lambda_{i}}\cap B_{R\sqrt{-\tau}}\left(O\right)\right)}{R^{2}\left(-\tau\right)}\leq C\left(\boldsymbol{\kappa},\mathcal{H}^{2}\left(\boldsymbol{\Sigma}_{0}\right),r_{p}\right),\label{rescaled area ratio bound}
\end{equation}
\begin{equation}
\sup_{-\frac{T}{\lambda_{i}^{2}}\leq\tau<0}\left\Vert H_{\boldsymbol{\Sigma}_{\tau}^{\left(P,T\right),\lambda_{i}}}\right\Vert _{L^{\infty}}\leq\lambda_{i}\boldsymbol{\Lambda},\label{rescaled mean curvature bound}
\end{equation}
\[
\int_{\boldsymbol{\Sigma}_{-1}^{\left(P,T\right),\lambda_{i}}}\left|A_{\boldsymbol{\Sigma}_{-1}^{\left(P,T\right),\lambda_{i}}}\right|^{2}\,d\mathcal{H}^{2}\leq C\left(\boldsymbol{\Lambda},\boldsymbol{\kappa},\boldsymbol{l},\mathcal{H}^{2}\left(\boldsymbol{\Sigma}_{0}\right),\chi\left(\boldsymbol{\Sigma}_{0}\right)\right).
\]
By a similar argument as in Remark \ref{isolated singularities},
the set
\begin{equation}
\mathfrak{S}_{P}\coloneqq\left\{ Q\in\mathbb{R}^{3}\left|\,\lim_{i\rightarrow\infty}\int_{\boldsymbol{\Sigma}_{-1}^{\left(P,T\right),\lambda_{i}}\cap B_{r}\left(Q\right)}\left|A_{\boldsymbol{\Sigma}_{-1}^{\left(P,T\right),\lambda_{i}}}\right|^{2}\,d\mathcal{H}^{2}\left(B_{r}\left(Q\right)\right)\geq\epsilon\quad\forall\;r>0\right.\right\} \label{energy-concentrated points}
\end{equation}
is finite, where $\epsilon$ is the constant in Proposition \ref{small energy implies regularity},
and 
\[
\mathcal{H}^{0}\left(\mathfrak{S}_{P}\right)\leq C\left(\boldsymbol{\Lambda},\boldsymbol{\kappa},\boldsymbol{l},\mathcal{H}^{2}\left(\boldsymbol{\Sigma}_{0}\right),\chi\left(\boldsymbol{\Sigma}_{0}\right)\right).
\]
It follows, by Proposition \ref{small energy implies regularity},
(\ref{rescaled area ratio bound}), (\ref{rescaled mean curvature bound})
and Proposition \ref{compactness}, that $\left\{ \boldsymbol{\Sigma}_{\tau}^{\left(P,T\right),\lambda_{i}}\right\} $
converges with finite multiplicity to a minimal surface $\left\{ \Pi\right\} _{-\infty<\tau<0}$
away from $\mathfrak{S}_{P}$. Note that
\[
\boldsymbol{U}^{P,\lambda_{i}}\rightarrow\mathbb{R}_{+}^{3},\quad\boldsymbol{\Gamma}^{P,\lambda_{i}}\rightarrow\partial\mathbb{R}_{+}^{3}\simeq\mathbb{R}^{2}
\]
and $\Pi$ is orthogonal to $\lim_{i\rightarrow\infty}\boldsymbol{\Gamma}^{P,\lambda_{i}}$.
Furthermore, Lemma \ref{monotonicity formula for MCF} implies that
the limiting surface $\Pi$ satisfies the self-shrinker equation
\[
H_{\Pi}+\frac{1}{2}X\cdot N_{\Pi}=0
\]
(cf. \cite{I,B}). Consequently, the minimal surface $\Pi$ must be
flat. 
\end{proof}

\section{\label{Multiplicity One Convergence}Unity of Huisken's Density }

The goal of this section is to show the unity of Huisken's density
of $\left\{ \boldsymbol{\Sigma}_{t}\right\} _{0\leq t<T}$ at time
$T$. By ``unity'' we mean that it is one for the interior limit
points and one half for the boundary limit points (see Proposition
\ref{unit Huisken density}). We will follow closely the procedure
in \cite{LW} to prove that. Our discussion will focus on the boundary
limit points since the arguments are similar for the interior limit
points. 

In order to prove the unity of Huisken's density, we will work with
the normalized MCF defined in (\ref{NMCF}). As a consequence of Proposition
\ref{condensation compactness for MCF}, any sequence of time-slices
of (\ref{NMCF}) would converge to a half-plane with multiplicity
away from (at most) finitely many singularities (see Lemma \ref{condensation compactness for NMCF}).
Following the idea of \cite{LW}, we will choose special sequences
and use that to prove the unity of Huisken's density by contradiction
(see Proposition \ref{delta pinching sequence}, Lemma \ref{Harnack estimate}
and Proposition \ref{unit Huisken density}). It then follows from
Allard's regularity theorem (cf. \cite{Al}) and Proposition \ref{Li-Wang curvature estimate}
that the convergence in Proposition \ref{condensation compactness for MCF}
is of multiplicity one and without singularities (see Corollary \ref{flat tangent flow}).

Let's begin our discussion with the parabolic rescaling defined in
(\ref{parabolic rescaling}). Given $P\in\mathcal{S}\cap\boldsymbol{\Gamma}$
and a sequence $\left\{ t_{i}\nearrow T\right\} _{i\in\mathbb{N}}$,
let $\lambda_{i}=\sqrt{T-t_{i}}$ and
\[
\boldsymbol{\Sigma}_{\tau}^{\left(P,T\right),\lambda_{i}}=\frac{1}{\lambda_{i}}\left(\boldsymbol{\Sigma}_{T+\lambda_{i}^{2}\tau}-P\right)\quad\textrm{for}\quad-\frac{T}{\lambda_{i}^{2}}\leq\tau<0,
\]
\[
\boldsymbol{U}^{P,\lambda_{i}}=\frac{1}{\lambda_{i}}\left(\boldsymbol{U}-P\right),\quad\boldsymbol{\Gamma}^{P,\lambda_{i}}=\partial\boldsymbol{U}^{P,\lambda_{i}}=\frac{1}{\lambda_{i}}\left(\boldsymbol{\Gamma}-P\right).
\]
Note that $\boldsymbol{\Gamma}^{P,\lambda_{i}}$ satisfies $\lambda_{i}\boldsymbol{\kappa}-$graph
condition. As usual, we will parametrize its tubular neighborhood
near $O$ by a map $\boldsymbol{\Phi}^{P,\lambda_{i}}$ (defined by
Definition \ref{tubular nbd} and Remark \ref{scale-preserving}). 

On the other hand, it is very useful to consider the following time-dependent
parabolic rescaling of $\left\{ \boldsymbol{\Sigma}_{t}\right\} $
(called ``normalized MCF'') , $\boldsymbol{U}$ and $\boldsymbol{\Gamma}$:
\begin{equation}
\Pi_{s}\coloneqq\left.\frac{1}{\sqrt{T-t}}\left(\boldsymbol{\Sigma}_{t}-P\right)\right|_{t=T-e^{-s}}\quad\textrm{for}\quad-\ln T\leq s<\infty,\label{NMCF}
\end{equation}
\[
\boldsymbol{U}_{s}\coloneqq\left.\frac{1}{\sqrt{T-t}}\left(\boldsymbol{U}-P\right)\right|_{t=T-e^{-s}},\quad\boldsymbol{\Gamma}_{s}\coloneqq\partial\boldsymbol{U}_{s}=\left.\frac{1}{\sqrt{T-t}}\left(\boldsymbol{\Gamma}-P\right)\right|_{t=T-e^{-s}}.
\]
Note that $\Pi_{s}$ has free boundary on $\boldsymbol{\Gamma}_{s}$
and that $\boldsymbol{\Gamma}_{s}$ satisfies $e^{-\frac{s}{2}}\boldsymbol{\kappa}-$graph
condition. Likewise, the tubular neighborhood of $\boldsymbol{\Gamma}_{s}$
near $O$ is parametrized by a map $\boldsymbol{\Phi}_{s}$ (defined
by Definition \ref{tubular nbd} and Remark \ref{scale-preserving}). 

As a remark, let $s_{i}=-\ln\left(T-t_{i}\right)$, then one can see
that 
\[
\Pi_{s_{i}}=\frac{1}{\sqrt{T-t_{i}}}\left(\boldsymbol{\Sigma}_{t_{i}}-P\right)=\boldsymbol{\Sigma}_{-1}^{\left(P,T\right),\lambda_{i}},
\]
\[
\boldsymbol{U}_{s_{i}}=\frac{1}{\sqrt{T-t_{i}}}\left(\boldsymbol{U}-P\right)=\boldsymbol{U}^{P,\lambda_{i}},\quad\boldsymbol{\Gamma}_{s_{i}}=\frac{1}{\sqrt{T-t_{i}}}\left(\boldsymbol{\Gamma}-P\right)=\boldsymbol{\Gamma}^{P,\lambda_{i}}.
\]
More generally, for each $\sigma\in\left[-\ln T-s_{i},\infty\right)$,
we have
\begin{equation}
\Pi_{s_{i}+\sigma}=\left.\frac{1}{\sqrt{-\tau}}\boldsymbol{\Sigma}_{\tau}^{\left(P,\boldsymbol{T}\right),\lambda_{i}}\right|_{\tau=-e^{-\sigma}},\label{drifted NMCF}
\end{equation}
\[
\boldsymbol{U}_{s_{i}+\sigma}=\left.\frac{1}{\sqrt{-\tau}}\boldsymbol{U}^{P,\lambda_{i}}\right|_{\tau=-e^{-\sigma}},\quad\boldsymbol{\Gamma}_{s_{i}+\sigma}=\partial\boldsymbol{U}_{s_{i}+\sigma}=\left.\frac{1}{\sqrt{-\tau}}\boldsymbol{\Gamma}^{P,\lambda_{i}}\right|_{\tau=-e^{-\sigma}}.
\]
The following lemma is a paraphrasing of Proposition \ref{condensation compactness for MCF}
in terms of $\left\{ \Pi_{s_{i}+\sigma}\right\} $. 
\begin{lem}
\label{condensation compactness for NMCF}Given $P\in\mathcal{S}\cap\boldsymbol{\Gamma}$
and a sequence $\left\{ s_{i}\nearrow\infty\right\} _{i\in\mathbb{N}}$,
there exist a half plane $\Pi$, a finite set $\mathfrak{S}_{P}\subset\Pi$,
and an integer $m\in\mathbb{N}$ so that, after passing to a subsequence,
$\left\{ \Pi_{s_{i}+\sigma}\right\} _{-\ln T-s_{i}\leq\sigma<\infty}$
converges to $\left\{ \Pi\right\} _{-\infty<\sigma<\infty}$ with
multiplicity $m$ away from $\underset{-\infty<\sigma<\infty}{\cup}\,e^{\frac{\sigma}{2}}\mathfrak{S}_{P}\times\left\{ \sigma\right\} $. 

The half plane $\Pi$ meets $\lim_{i\rightarrow\infty}\boldsymbol{\Gamma}_{s_{i+\sigma}}\simeq\mathbb{R}^{2}$
orthogonally. The number of $\mathfrak{S}_{P}$ is bounded by
\[
\mathcal{H}^{0}\left(\mathfrak{S}_{P}\right)\leq C\left(\boldsymbol{\Lambda},\boldsymbol{\kappa},\boldsymbol{l},\mathcal{H}^{2}\left(\boldsymbol{\Sigma}_{0}\right),\chi\left(\boldsymbol{\Sigma}_{0}\right)\right).
\]
Moreover, Huisken's densitiy of $\left\{ \boldsymbol{\Sigma}_{t}\right\} $
at $\left(P,T\right)$ (see Lemma \ref{Huisken's density}) is given
by
\[
\Theta_{\left\{ \boldsymbol{\Sigma}_{t}\right\} }\left(P,T\right)=\frac{m}{2}.
\]
\end{lem}

\begin{proof}
By Proposition \ref{condensation compactness for MCF}, there exist
a half plane $\Pi$, a finite set $\mathfrak{S}_{P}\subset\Pi$, and
an integer $m\in\mathbb{N}$ so that $\left\{ \boldsymbol{\Sigma}_{\tau}^{\left(P,T\right),\lambda_{i}}\right\} \rightarrow\left\{ \Pi\right\} $
with multiplicity $m$ away from the set $\mathfrak{S}_{P}\times\left(-\infty,0\right)$.
It follows that 
\[
\Pi_{s_{i}+\sigma}=\left.\frac{1}{\sqrt{-\tau}}\boldsymbol{\Sigma}_{\tau}^{\left(P,T\right),\lambda_{i}}\right|_{\tau=-e^{-\sigma}}\,\rightarrow\,\frac{1}{\sqrt{-\tau}}\Pi=\Pi
\]
away from $\underset{-\infty<\sigma<\infty}{\cup}\,\frac{1}{\sqrt{-\tau}}\mathfrak{S}_{P}\times\left\{ \sigma\right\} $,
where $\tau=-e^{-\sigma}$. Notice that 
\[
\left\Vert H_{\Pi_{s_{i}+\sigma}}\right\Vert _{L^{\infty}}\leq\boldsymbol{\Lambda}e^{-\frac{1}{2}\left(s_{i}+\sigma\right)}\rightarrow0\qquad\textrm{as}\quad i\rightarrow\infty
\]
and $\boldsymbol{\Gamma}_{s_{i}+\sigma}$ satisfies a $\boldsymbol{\kappa}e^{-\frac{1}{2}\left(s_{i}+\sigma\right)}-$graph
condition.

On the other hand, let $t_{i}=T-e^{-s_{i}}$, then Lemma \ref{Huisken's density}
yields 
\[
\Theta_{\left\{ \boldsymbol{\Sigma}_{t}\right\} }\left(P,T\right)=\lim_{i\rightarrow\infty}\int_{\boldsymbol{\Sigma}_{t_{i}}}e^{85\left(\boldsymbol{\kappa}^{2}\left(T-t_{i}\right)\right)^{\frac{2}{5}}}\eta_{\boldsymbol{\Gamma};P,T}\,\Psi_{\boldsymbol{\Gamma};P,T}\left(X,t_{i}\right)\,d\mathcal{H}^{2}\left(X\right)
\]
\[
=\lim_{i\rightarrow\infty}\int_{\Pi_{s_{i}}}e^{85\left(\boldsymbol{\kappa}^{2}e^{-s_{i}}\right)^{\frac{2}{5}}}\eta_{\boldsymbol{\Gamma}_{s_{i}}}\,\Psi_{\boldsymbol{\Gamma}_{s_{i}}}\left(Y\right)\,d\mathcal{H}^{2}\left(Y\right),
\]
where 
\[
\eta_{\boldsymbol{\Gamma}_{s_{i}}}\left(Y\right)=\left(1-\frac{\left|Y\right|^{2}+\left|\tilde{Y}\right|^{2}-80}{\left(\frac{1}{2}\left(\boldsymbol{\kappa}^{2}e^{-s_{i}}\right)^{\frac{2}{5}}\left(\boldsymbol{\kappa}e^{-\frac{s_{i}}{2}}\right)^{-1}\right)^{2}}\right)_{+}^{4},
\]
\[
\Psi_{\boldsymbol{\Gamma}_{s_{i}}}\left(Y\right)=\frac{1}{4\pi}\exp\left(-\frac{\frac{1}{2}\left(\left|Y\right|^{2}+\left|\tilde{Y}\right|^{2}\right)}{4\left(1+16\left(\boldsymbol{\kappa}^{2}e^{-s_{i}}\right)^{\frac{2}{5}}\right)}\right).
\]
Due to $\boldsymbol{\Gamma}_{s_{i}}\rightarrow\partial\mathbb{R}^{3}\simeq\mathbb{R}^{2}$,
$\Pi_{s_{i}}\rightarrow\Pi$ with multiplicity $m$ away from $\mathfrak{S}_{P}$,
(\ref{rescaled area ratio bound}) and Lemma \ref{monotonicity of area ratio}
(with $\Sigma$ replaced by $\Pi_{s_{i}}$), we get $\Theta_{\left\{ \boldsymbol{\Sigma}_{t}\right\} }\left(P,T\right)=\frac{m}{2}$.
\end{proof}
Next, we would like to choose carefully sequences in Lemma \ref{condensation compactness for NMCF}
in the hope that it could help to prove the unity of Huisken's density.
For that purpose, let's make the following definition. Given $\delta>0$,
for each $s\in\left[-\ln T,\infty\right)$ we define
\[
\boldsymbol{U}_{s}^{\delta-\textrm{curved}}=\left\{ Q\in\boldsymbol{U}_{s}\cap B_{\delta^{-1}}\left(O\right)\left|\,\,\exists\,Q'\in\Pi_{s}\cap B_{\delta}\left(Q\right)\;\,\textrm{s.t.}\;\,\delta\left|A_{\Pi_{s}}\left(Q'\right)\right|\geq1\right.\right\} .
\]
Let $\boldsymbol{U}_{s}^{\delta-\textrm{outer}}$ be the set consisting
of all points $Q\in\left(\boldsymbol{U}_{s}\cap B_{\delta^{-1}}\left(O\right)\right)\setminus\left(\Pi_{s}\cup\boldsymbol{U}_{s}^{\delta-\textrm{curved}}\right)$
for which there is a continuous path $\gamma:\left[0,1\right)\rightarrow\boldsymbol{U}_{s}\setminus\left(\Pi_{s}\cup\boldsymbol{U}_{s}^{\delta-\textrm{curved}}\right)$
so that $\gamma\left(0\right)=Q$ and $\textrm{dist}\left(\gamma\left(\xi\right),\Pi_{s}\right)\rightarrow\infty$
as $\xi\nearrow1$. Then we define 
\begin{equation}
\boldsymbol{U}_{s}^{\delta}=\left(\boldsymbol{U}_{s}\cap B_{\delta^{-1}}\left(O\right)\right)\setminus\left(\boldsymbol{U}_{s}^{\delta-\textrm{curved}}\cup\boldsymbol{U}_{s}^{\delta-\textrm{outer}}\right).\label{enclosed region}
\end{equation}
Loosely speaking, $\boldsymbol{U}_{s}^{\delta}$ is the region in
$B_{\delta^{-1}}\left(O\right)$ which is enclosed by $\Pi_{s}$ (especially
when it has multiple sheets) and away from points of large curvature
of $\Pi_{s}$ (cf. \cite{LW}). 

The next lemma will be used to choose the special sequences in Proposition
\ref{delta pinching sequence}.
\begin{lem}
\label{pinching estimate}Given $\delta>0$ and $i\in\mathbb{N}$,
there exists $s_{i}\geq i$ so that 
\[
\sup_{0<\sigma\leq i}\mathcal{H}^{3}\left(\boldsymbol{U}_{s_{i}+\sigma}^{\delta}\right)\leq\left(1+\frac{1}{i}\right)\,\mathcal{H}^{3}\left(\boldsymbol{U}_{s_{i}}^{\delta}\right).
\]
\end{lem}

\begin{proof}
Suppose the contrary, then for every $s\geq i$, there exists $0<\sigma_{s}\leq i$
so that 
\[
\mathcal{H}^{3}\left(\boldsymbol{U}_{s+\sigma_{s}}^{\delta}\right)>\left(1+\frac{1}{i}\right)\,\mathcal{H}^{3}\left(\boldsymbol{U}_{s}^{\delta}\right).
\]
In particular, we have 
\[
\mathcal{H}^{3}\left(\boldsymbol{U}_{i+\sigma_{i}}^{\delta}\right)>\left(1+\frac{1}{i}\right)\,\mathcal{H}^{3}\left(\boldsymbol{U}_{i}^{\delta}\right)\geq0.
\]
Define a sequence $\left\{ s_{k}\right\} _{k\in\mathbb{Z}_{+}}$ recursively
by setting $s_{k+1}=s_{k}+\sigma_{s_{k}}$ for $k\geq0$ and $s_{0}=i+\sigma_{i}$.
Then we have
\[
\mathcal{H}^{3}\left(\boldsymbol{U}_{s_{k}}^{\delta}\right)>\left(1+\frac{1}{i}\right)^{k}\,\mathcal{H}^{3}\left(\boldsymbol{U}_{s_{0}}^{\delta}\right)\qquad\forall\,\,k\in\mathbb{N},
\]
which implies $\limsup_{k\rightarrow\infty}\,\mathcal{H}^{3}\left(\boldsymbol{U}_{s_{k}}^{\delta}\right)=\infty$
(a contradiction). 
\end{proof}
In the following proposition, we use Lemma \ref{pinching estimate}
to choose a special sequence of (\ref{drifted NMCF}). By Lemma \ref{condensation compactness for NMCF}
and Definition \ref{tubular nbd}, we can parametrize each flow as
a multigraph over a half-plane. In that case, if the multiplicity
is not one, the volume of (\ref{enclosed region}) can be roughly
interpreted as the integral of the difference of the upper and lower
graphs, over which we have some uniform control within any finite
period of time (see (\ref{delta pinching of graph})). Moreover, we
also derive the equation satisfied by the difference functions.
\begin{prop}
\label{delta pinching sequence}Given $P\in\mathcal{S}\cap\boldsymbol{\Gamma}$
and $\delta>0$. Suppose that $\Theta_{\left\{ \boldsymbol{\Sigma}_{t}\right\} }\left(P,T\right)>\frac{1}{2}$,
then there exists a sequence $\left\{ s_{i}\nearrow\infty\right\} _{i\in\mathbb{N}}$,
a half plane $\Pi$, a finite set $\mathfrak{S}_{P}\subset\Pi$ and
an integer $m>1$ with the following property. If we assume (without
loss of generality) that 
\[
\lim_{i\rightarrow\infty}\boldsymbol{U}_{s_{i}}=\left\{ \left.\left(x_{1},x_{2},x_{3}\right)\right|\,x_{1},x_{3}\in\mathbb{R},\,x_{2}>0\right\} ,
\]
\[
\lim_{i\rightarrow\infty}\boldsymbol{\Gamma}_{s_{i}}=\left\{ \left.\left(x_{1},0,x_{3}\right)\right|\,x_{1},x_{3}\in\mathbb{R}\right\} ,
\]
\[
\Pi=\left\{ \left.\left(x_{1},x_{2},0\right)\right|\,x_{1}\in\mathbb{R},\,x_{2}\geq0\right\} ,
\]
then the rescaled flow $\left\{ \Pi_{s_{i}+\sigma}\right\} _{-\ln T-s_{i}\leq\sigma<\infty}$
defined in (\ref{drifted NMCF}) can be parametrized as a multigraph
\[
Y_{s_{i}+\sigma}\left(y_{1},y_{2}\right)=\boldsymbol{\Phi}_{s_{i}+\sigma}\left(y_{1},y_{2},v_{s_{i}}^{j}\left(y_{1},y_{2},\sigma\right)\right),\quad j=1,\cdots,m
\]
where $\boldsymbol{\Phi}_{s_{i}+\sigma}$ is the map in Definition
\ref{tubular nbd} which parametrizes the tubular neighborhood of
$\boldsymbol{\Gamma}_{s_{i}+\sigma}$ near $O$, and the functions
satisfy 
\[
v_{s_{i}}^{1}\left(y,\sigma\right)<\cdots<v_{s_{i}}^{m}\left(y,\sigma\right),
\]
\[
v_{s_{i}}^{j}\left(y,\sigma\right)\rightarrow0\qquad\textrm{away from}\;\,\underset{-\infty<\sigma<\infty}{\cup}\,e^{\frac{\sigma}{2}}\mathfrak{S}_{P}\times\left\{ \sigma\right\} \;\,\textrm{for}\;\,j\in\left\{ 1,\cdots,m\right\} .
\]
In addition, let $v_{s_{i}}=v_{s_{i}}^{m}-v_{s_{i}}^{1}$, then it
satisfies
\[
\partial_{\sigma}v_{s_{i}}=\partial_{k}\left(g_{s_{i}+\sigma}^{kl}\left(y,v_{s_{i}}^{m},\nabla v_{s_{i}}^{m}\right)\,\partial_{l}v_{s_{i}}\right)+\left(-\frac{1}{2}y+b_{s_{i}}\left(y,\sigma\right)\right)\cdot\nabla v_{s_{i}}+\left(\frac{1}{2}+c_{s_{i}}\left(y,\sigma\right)\right)v_{s_{i}},
\]
\begin{equation}
\left.\partial_{2}v_{s_{i}}\right|_{y_{2}=0}=0,\label{linear eq of type I graph}
\end{equation}
in which the coefficients satisfy 
\begin{equation}
g_{s_{i}+\sigma}^{kl}\left(y,v_{s_{i}}^{m},\nabla v_{s_{i}}^{m}\right)\rightarrow\boldsymbol{\delta}^{kl}\quad\textrm{and}\quad\left.g_{s_{i}+\sigma}^{12}\left(y,v_{s_{i}}^{m},\nabla v_{s_{i}}^{m}\right)\right|_{y_{2}=0}=0,\label{coefficients bound for type I eq}
\end{equation}
\[
\left|b_{s_{i}}\right|\,+\,\left|c_{s_{i}}\right|\rightarrow0.
\]
Furthermore, given $0<\varepsilon<1<\mathcal{T}<\infty$, for $i\gg1$
there holds
\begin{equation}
\sup_{0<\sigma\leq\mathcal{T}}\,\int_{\Pi\cap B_{\left(1-\varepsilon\right)\delta^{-1}}\left(O\right)\setminus\,\cup_{Q\in e^{\frac{\sigma}{2}}\mathfrak{S}_{P}}B_{\left(1+\varepsilon\right)\delta}\left(Q\right)}\,v_{s_{i}}\left(y,\sigma\right)\,dy\label{delta pinching of graph}
\end{equation}
\[
\leq\left(1+\varepsilon\right)\int_{\Pi\cap B_{\left(1+\varepsilon\right)\delta^{-1}}\left(O\right)\setminus\,\cup_{Q\in e^{\frac{\sigma}{2}}\mathfrak{S}_{P}}B_{\left(1-\varepsilon\right)\delta}\left(Q\right)}\,v_{s_{i}}\left(y,0\right)\,dy.
\]
 
\end{prop}

\begin{proof}
Given $\delta>0$, by Proposition \ref{condensation compactness for MCF},
Lemma \ref{condensation compactness for NMCF} and Lemma \ref{pinching estimate},
there exist a sequence $\left\{ t_{i}\nearrow T\right\} $, a half
plane $\Pi$ and a finite set $\mathfrak{S}_{P}\subset\Pi$, and an
integer $m\in\mathbb{N}$ with the following properties. 
\begin{itemize}
\item Let $\lambda_{i}=\sqrt{T-t_{i}}$, then the sequence of MCF defined
in (\ref{parabolic rescaling}) satisfies $\left\{ \boldsymbol{\Sigma}_{\tau}^{\left(P,T\right),\lambda_{i}}\right\} _{-\frac{T}{\lambda_{i}^{2}}\leq\tau<0}\rightarrow\left\{ \Pi\right\} _{-\infty<\tau<0}$
with finite multiplicity $m$ away from $\mathfrak{S}_{P}\times\left(-\infty,0\right)$.
The half plane $\Pi$ has free boundary on $\lim_{i\rightarrow\infty}\boldsymbol{\Gamma}^{P,\lambda_{i}}\simeq\mathbb{R}^{2}$; 
\item Let $s_{i}=-\ln\left(T-t_{i}\right)$, then the sequence of normalized
MCF defined in (\ref{drifted NMCF}) satisfies $\left\{ \Pi_{s_{i}+\sigma}\right\} _{-\ln T-s_{i}\leq\sigma<\infty}\rightarrow\left\{ \Pi\right\} _{-\infty<\tau<\infty}$
with finite multiplicity $m$ away from $\underset{-\infty<\tau<\infty}{\cup}\,e^{\frac{\sigma}{2}}\mathfrak{S}_{P}\times\left\{ \sigma\right\} $;
\item The set $\boldsymbol{U}_{s_{i}}^{\delta}$ defined in (\ref{enclosed region})
satisfies 
\begin{equation}
\sup_{0<\sigma\leq i}\mathcal{H}^{3}\left(\boldsymbol{U}_{s_{i}+\sigma}^{\delta}\right)\leq\left(1+\frac{1}{i}\right)\,\mathcal{H}^{3}\left(\boldsymbol{U}_{s_{i}}^{\delta}\right).\label{delta pinching condition}
\end{equation}
\end{itemize}
With out loss of generality, we may assume that 
\[
\lim_{i\rightarrow\infty}\boldsymbol{U}^{P,\lambda_{i}}=\left\{ \left.\left(x_{1},x_{2},x_{3}\right)\right|\,x_{1},x_{3}\in\mathbb{R},\,x_{2}>0\right\} ,
\]
\[
\lim_{i\rightarrow\infty}\boldsymbol{\Gamma}^{P,\lambda_{i}}=\left\{ \left.\left(x_{1},0,x_{3}\right)\right|\,x_{1},x_{3}\in\mathbb{R}\right\} ,
\]
\[
\Pi=\left\{ \left.\left(x_{1},x_{2},0\right)\right|\,x_{1}\in\mathbb{R},\,x_{2}\geq0\right\} .
\]
Let $\boldsymbol{\Phi}^{P,\lambda_{i}}$ be the map defined in Definition
\ref{tubular nbd}, which parametrizes the tubular neighborhood of
$\boldsymbol{\Gamma}^{P,\lambda_{i}}$ near $O$ (see also Remark
\ref{scale-preserving}). Since $\boldsymbol{\Gamma}^{P,\lambda_{i}}$
satisfies $e^{-\frac{1}{2}\left(s_{i}+\sigma\right)}\boldsymbol{\kappa}-$graph
condition, $\boldsymbol{\Phi}^{P,\lambda_{i}}$ converges to the identity
map as $i\rightarrow\infty$. By (\ref{approximate identity}) and
the normal-vector-preserving property of $\boldsymbol{\Phi}^{P,\lambda_{i}}$,
for each $-\infty<\tau<0$, $\left(\boldsymbol{\Phi}^{P,\lambda_{i}}\right)^{-1}\left(\boldsymbol{\Sigma}_{\tau}^{\left(P,T\right),\lambda_{i}}\right)$
has free boundary on 
\[
\left(\boldsymbol{\Phi}^{P,\lambda_{i}}\right)^{-1}\left(\boldsymbol{\Gamma}^{P,\lambda_{i}}\right)\subset\left\{ \left.\left(y_{1},0,y_{3}\right)\,\right|\,y_{1},y_{3}\in\mathbb{R}\right\} 
\]
and it converges to 
\[
\left\{ \left.\left(y_{1},y_{2},0\right)\,\right|\,y_{1}\in\mathbb{R},\,y_{2}\geq0\right\} \simeq\Pi
\]
with multiplicity $m$ away from $\mathfrak{S}_{P}$. It follows that
for $i\gg1$ and away from $\mathfrak{S}_{P}$, $\left(\boldsymbol{\Phi}^{P,\lambda_{i}}\right)^{-1}\left(\boldsymbol{\Sigma}_{\tau}^{\left(P,T\right),\lambda_{i}}\right)$
is a disjoint union of graphs of $u_{\lambda_{i}}^{j}\left(y_{1},y_{2},t\right)$
defined on $\Pi$ for $j=1,\cdots,m$. We may assume that $u_{\lambda_{i}}^{1}\left(y,t\right)<\cdots<u_{\lambda_{i}}^{m}\left(y,t\right)$.
Note that $\Theta_{\left\{ \boldsymbol{\Sigma}_{t}\right\} }\left(P,T\right)>\frac{1}{2}$
implies $m>1$ (see Lemma \ref{condensation compactness for NMCF}),
and that $u_{\lambda_{i}}^{j}\left(y,t\right)\rightarrow0$ away from
$\mathfrak{S}_{P}\times\left(-\infty,0\right)$ for each $j\in\left\{ 1,\cdots,m\right\} $.
Thus, we can parametrize $\boldsymbol{\Sigma}_{\tau}^{\left(P,T\right),\lambda_{i}}$
(away from $\mathfrak{S}_{P}$) as 
\[
X_{\tau}^{\left(P,T\right),\lambda_{i}}\left(y_{1},y_{2}\right)=\boldsymbol{\Phi}^{P,\lambda_{i}}\left(y_{1},y_{2},u_{\lambda_{i}}^{j}\left(y_{1},y_{2},t\right)\right),\quad j=1,\cdots,m.
\]

Using a similar argument as in Lemma \ref{parametrization of MCF near the boundary},
$u_{\lambda_{i}}^{j}\left(y,t\right)$ satisfies an analogous equation
as (\ref{eq of graph}). Namely, 
\begin{equation}
\partial_{\tau}u_{\lambda_{i}}^{j}=\boldsymbol{g}_{\lambda_{i}}^{kl}\left(y,u_{\lambda_{i}}^{j},\nabla u_{\lambda_{i}}^{j}\right)\partial_{kl}^{2}u_{\lambda_{i}}^{j}+\boldsymbol{f}_{\lambda_{i}}\left(y,u_{\lambda_{i}}^{j},\nabla u_{\lambda_{i}}^{j}\right),\label{eq of rescaled graph}
\end{equation}
where $\boldsymbol{g}_{\lambda_{i}}^{kl}\left(y,u_{\lambda_{i}}^{j},\nabla u_{\lambda_{i}}^{j}\right)$
and $\boldsymbol{f}_{\lambda_{i}}\left(y,u_{\lambda_{i}}^{j},\nabla u_{\lambda_{i}}^{j}\right)$
are defined in the same way as (\ref{induced metric}) and (\ref{inhomogeneous term})
but with $\boldsymbol{\Phi}^{P,\lambda_{i}}$ in place of $\Phi$.
More precisely, $\boldsymbol{g}_{\lambda_{i}}^{kl}\left(y,u_{\lambda_{i}}^{j},\nabla u_{\lambda_{i}}^{j}\right)$
is the inverse of
\[
\boldsymbol{g}_{kl}^{\lambda_{i}}\left(y,u_{\lambda_{i}}^{j},\nabla u_{\lambda_{i}}^{j}\right)\coloneqq\partial_{k}X_{\tau}^{\left(P,T\right),\lambda_{i}}\cdot\partial_{l}X_{\tau}^{\left(P,T\right),\lambda_{i}}
\]
\[
=\boldsymbol{h}_{kl}^{\lambda_{i}}\left(y,u_{\lambda_{i}}^{j}\right)+\boldsymbol{h}_{k3}^{\lambda_{i}}\left(y,u_{\lambda_{i}}^{j}\right)\partial_{l}u_{\lambda_{i}}^{j}+\boldsymbol{h}_{l3}^{\lambda_{i}}\left(y,u_{\lambda_{i}}^{j}\right)\partial_{k}u_{\lambda_{i}}^{j}+\boldsymbol{h}_{33}^{\lambda_{i}}\left(y,u_{\lambda_{i}}^{j}\right)\partial_{k}u_{\lambda_{i}}^{j}\,\partial_{l}u_{\lambda_{i}}^{j},
\]
where 
\[
\boldsymbol{h}_{kl}^{\lambda_{i}}\left(y_{1},y_{2},y_{3}\right)=\partial_{k}\boldsymbol{\Phi}^{P,\lambda_{i}}\left(y_{1},y_{2},y_{3}\right)\cdot\partial_{l}\boldsymbol{\Phi}^{P,\lambda_{i}}\left(y_{1},y_{2},y_{3}\right).
\]
is the pull-back metric by $\boldsymbol{\Phi}^{P,\lambda_{i}}$ (see
also (\ref{pull-back metric}) and (\ref{induced metric})). Similarly,
we define 
\[
\boldsymbol{f}_{\lambda_{i}}\left(y,u_{\lambda_{i}}^{j},\nabla u_{\lambda_{i}}^{j}\right)=\boldsymbol{g}_{\lambda_{i}}^{kl}\left(y,u_{\lambda_{i}}^{j},\nabla u_{\lambda_{i}}^{j}\right)\left\{ \boldsymbol{\varGamma}_{kl,\lambda_{i}}^{3}\left(y,u_{\lambda_{i}}^{j}\right)+\boldsymbol{Q}_{kl,,\lambda_{i}}\left(y,u_{\lambda_{i}}^{j},\nabla u_{\lambda_{i}}^{j}\right)\right\} ,
\]
where $\boldsymbol{\varGamma}_{kl,\lambda_{i}}^{3}\left(y,u_{\lambda_{i}}^{j}\right)$
and $\boldsymbol{Q}_{kl,,\lambda_{i}}\left(y,u_{\lambda_{i}}^{j},\nabla u_{\lambda_{i}}^{j}\right)$
are defined analogously as (\ref{connection}) and (\ref{quadratic term})
but with $\boldsymbol{\Phi}^{P,\lambda_{i}}$ and $\left(\boldsymbol{\Phi}^{P,\lambda_{i}},u_{\lambda_{i}}^{j},\nabla u_{\lambda_{i}}^{j}\right)$
in place of $\Phi$ and $\left(\Phi,u,\nabla u\right)$, respectively.
As in Lemma \ref{parametrization of MCF near the boundary}, we have
\begin{equation}
\left.\partial_{2}u_{\lambda_{i}}^{j}\right|_{y_{2}=0}=0,\label{Neumann boundary condition for resclaed graph}
\end{equation}
\begin{equation}
\left.\boldsymbol{g}_{\lambda_{i}}^{12}\left(y,u_{\lambda_{i}}^{j},\nabla u_{\lambda_{i}}^{j}\right)\right|_{y_{2}=0}=0.\label{reflection condition for rescaled graph}
\end{equation}
Note that $\boldsymbol{\Gamma}^{P,\lambda_{i}}$ satisfies a $\lambda_{i}\boldsymbol{\kappa}-$graph
condition.

By (\ref{drifted NMCF}) and Remark \ref{scale-preserving}, $\Pi_{s_{i}+\sigma}$
can be parametrized as a multigraph 
\[
Y_{s_{i}+\sigma}\left(y\right)=\boldsymbol{\Phi}_{s_{i}+\sigma}\left(y,v_{s_{i}}^{j}\left(y,\sigma\right)\right),\quad j=1,\cdots,m,
\]
where
\begin{equation}
v_{s_{i}}^{j}\left(y,\sigma\right)=\left.\frac{1}{\sqrt{-\tau}}\,u_{\lambda_{i}}^{j}\left(\sqrt{-\tau}\,y,\tau\right)\right|_{\tau=-e^{-\sigma}}\label{type I rescaling of graph}
\end{equation}
and $\boldsymbol{\Phi}_{s_{i}+\sigma}$ is the map defined in \ref{tubular nbd}
(which parametrizes a tubular neighborhood of $\boldsymbol{\Gamma}_{s_{i}+\sigma}$).
Note that $v_{s_{i}}^{1}\left(y,\sigma\right)<\cdots<v_{s_{i}}^{m}\left(y,\sigma\right)$
and $v_{s_{i}}^{j}\left(y,\sigma\right)\rightarrow0$ away from $\left\{ \left(y,\sigma\right)\left|\,y\in e^{\frac{\sigma}{2}}\mathfrak{S}_{P},\,\sigma\in\mathbb{R}\right.\right\} $.
By (\ref{eq of rescaled graph}), (\ref{Neumann boundary condition for resclaed graph})
and (\ref{type I rescaling of graph}), we have
\begin{equation}
\partial_{\sigma}v_{s_{i}}^{j}+\frac{1}{2}y\cdot\nabla v_{s_{i}}^{j}-\frac{1}{2}v_{s_{i}}^{j}=\boldsymbol{g}_{s_{i}+\sigma}^{kl}\left(y,v_{s_{i}}^{j},\nabla v_{s_{i}}^{j}\right)\partial_{kl}^{2}v_{s_{i}}^{j}+\boldsymbol{f}_{s_{i}+\sigma}\left(y,v_{s_{i}}^{j},\nabla v_{s_{i}}^{j}\right),\label{eq of type I rescaled graph}
\end{equation}
\[
\left.\partial_{2}v_{s_{i}}^{j}\right|_{y_{2}=0}=0,
\]
for $j=1,\cdots,m$, where $\boldsymbol{g}_{s_{i}+\sigma}^{kl}\left(y,v_{s_{i}}^{j},\nabla v_{s_{i}}^{j}\right)$
and $\boldsymbol{f}_{s_{i}+\sigma}\left(y,v_{s_{i}}^{j},\nabla v_{s_{i}}^{j}\right)$
are defined analogously as (\ref{induced metric}) and (\ref{inhomogeneous term})
but with $\boldsymbol{\Phi}_{s_{i}+\sigma}$ in place of $\Phi$.
Note that $\boldsymbol{\Gamma}_{s_{i}+\sigma}$ satisfies our $e^{-\frac{s_{i}+\sigma}{2}}\boldsymbol{\kappa}-$graph
condition.

Next, let $v_{s_{i}}=v_{s_{i}}^{m}-v_{s_{i}}^{1}$. By (\ref{eq of type I rescaled graph}),
we get 
\[
\partial_{\sigma}v_{s_{i}}+\frac{1}{2}y\cdot\nabla v_{s_{i}}-\frac{1}{2}v_{s_{i}}=\partial_{k}\left(\boldsymbol{g}_{s_{i}+\sigma}^{kl}\left(y,v_{s_{i}}^{m},\nabla v_{s_{i}}^{m}\right)\,\partial_{l}v_{s_{i}}\right)+\Xi_{s_{i}+\sigma},
\]
\[
\left.\partial_{2}v_{s_{i}}\right|_{y_{2}=0}=0,
\]
where 
\[
\Xi_{s_{i}+\sigma}=-\partial_{k}\left(\boldsymbol{g}_{s_{i}+\sigma}^{kl}\left(y,v_{s_{i}}^{m},\nabla v_{s_{i}}^{m}\right)\right)\partial_{l}v_{s_{i}}
\]
\[
+\left(\boldsymbol{g}_{s_{i}+\sigma}^{kl}\left(y,v_{s_{i}}^{m},\nabla v_{s_{i}}^{m}\right)-\boldsymbol{g}_{s_{i}+\sigma}^{kl}\left(y,v_{s_{i}}^{1},\nabla v_{s_{i}}^{1}\right)\right)\partial_{kl}^{2}v_{s_{i}}^{1}
\]
\[
+\boldsymbol{f}_{s_{i}+\sigma}\left(y,v_{s_{i}}^{m},\nabla v_{s_{i}}^{m}\right)-\boldsymbol{f}_{s_{i}+\sigma}\left(y,v_{s_{i}}^{1},\nabla v_{s_{i}}^{1}\right).
\]
From (\ref{inhomogeneous term}), we have
\[
\Xi_{s_{i}+\sigma}=b_{s_{i}}\left(y,\sigma\right)\cdot\nabla v_{s_{i}}+c_{s_{i}}\left(y,\sigma\right)v_{s_{i}}.
\]
with the vector-valued function $b_{s_{i}}$ and the scalar function
$c_{s_{i}}$ satisfying 
\[
\left|b_{s_{i}}\right|\,+\,\left|c_{s_{i}}\right|
\]
\[
\leq C\left(\left\Vert \boldsymbol{\varphi}_{s_{i}+\sigma}\right\Vert _{C^{3}},y_{2},\nabla v_{s_{i}}^{m},\nabla v_{s_{i}}^{1}\right)\,\left(\left\Vert v_{s_{i}}^{m}\right\Vert _{C^{2}}+\left\Vert v_{s_{i}}^{1}\right\Vert _{C^{2}}+\left\Vert \boldsymbol{\varGamma}_{kl,s_{i}+\sigma}^{p}\right\Vert _{L^{\infty}}+\left[\boldsymbol{\varGamma}_{kl,s_{i}+\sigma}^{p}\right]_{1,\textrm{spacial}}\right),
\]
where $\boldsymbol{\varphi}_{s_{i}+\sigma}$ is the local graph of
$\boldsymbol{\Gamma}_{s_{i}+\sigma}$ (see Definition \ref{kappa graph condition}),
$\boldsymbol{\varGamma}_{kl,s_{i}+\sigma}^{p}$ is the connection
associated with $\boldsymbol{\Phi}_{s_{i}+\sigma}$(see (\ref{connection})),
and $\left[\cdot\right]_{1,\textrm{spacial}}$ is the Lipschitz norm
with respect to the spacial variable (see Lemma \ref{parametrization of MCF near the boundary}).
Note that (\ref{reflection condition for rescaled graph}) yields
\[
\left.\boldsymbol{g}_{s_{i}+\sigma}^{12}\left(y,v_{s_{i}}^{m},\nabla v_{s_{i}}^{m}\right)\right|_{y_{2}=0}=0.
\]
Since $\boldsymbol{\Gamma}_{s_{i}+\sigma}$ satisfies $e^{-\frac{1}{2}\left(s_{i}+\sigma\right)}\boldsymbol{\kappa}-$graph
condition, we have $\boldsymbol{\Phi}_{s_{i}+\sigma}\left(Y\right)\stackrel{C^{3}}{\rightarrow}Y$
as $i\rightarrow\infty$. It follows that 
\[
\boldsymbol{g}_{s_{i}+\sigma}^{kl}\left(y,v_{s_{i}}^{m},\nabla v_{s_{i}}^{m}\right)\rightarrow\boldsymbol{\delta}^{kl},\quad\left|b_{s_{i}}^{k}\right|+\left|c_{s_{i}}\right|\rightarrow0.
\]
 Moreover, the push-forward measure $\boldsymbol{\Phi}_{s_{i}+\sigma}^{*}\mathcal{H}^{3}$
converge to $\mathcal{H}^{3}$ as $i\rightarrow\infty$. Namely, for
every Borel set $\mathcal{B}$ in $\mathbb{R}^{3}$, there holds
\[
\boldsymbol{\Phi}_{s_{i}+\sigma}^{*}\mathcal{H}^{3}\left(\mathcal{B}\right)=\mathcal{H}^{3}\left(\boldsymbol{\Phi}_{s_{i}+\sigma}^{-1}\left(\mathcal{B}\right)\right)\rightarrow\mathcal{H}^{3}\left(\mathcal{B}\right).
\]

Lastly, given $0<\varepsilon<1<\mathcal{T}<\infty$, by (\ref{enclosed region}),
for $i\gg1$ we have
\[
\sup_{\left|\sigma\right|\leq\mathcal{T}}\,\int_{\Pi\cap B_{\left(1-\varepsilon\right)\delta^{-1}}\left(O\right)\setminus\,\cup_{Q\in e^{\frac{\sigma}{2}}\mathfrak{S}_{P}}B_{\left(1+\varepsilon\right)\delta}\left(Q\right)}\,v_{s_{i}}\left(y,\sigma\right)\,dy\,\leq\left(1+\frac{\varepsilon}{2}\right)\mathcal{H}^{3}\left(\boldsymbol{U}_{s_{i}+\sigma}^{\delta}\right),
\]
\[
\int_{\Pi\cap B_{\left(1+\varepsilon\right)\delta^{-1}}\left(O\right)\setminus\,\cup_{Q\in e^{\frac{\sigma}{2}}\mathfrak{S}_{P}}B_{\left(1-\varepsilon\right)\delta}\left(Q\right)}\,v_{s_{i}}\left(y,0\right)\,dy\,\geq\left(1-\frac{\varepsilon}{2}\right)\mathcal{H}^{3}\left(\boldsymbol{U}_{s_{i}}^{\delta}\right).
\]
Using (\ref{delta pinching condition}), we get
\[
\sup_{0<\sigma\leq\mathcal{T}}\,\int_{\Pi\cap B_{\left(1-\varepsilon\right)\delta^{-1}}\left(O\right)\setminus\,\cup_{Q\in e^{\frac{\sigma}{2}}\mathfrak{S}_{P}}B_{\left(1+\varepsilon\right)\delta}\left(Q\right)}\,v_{s_{i}}\left(y,\sigma\right)\,dy
\]
\[
\leq\frac{1+\frac{\varepsilon}{2}}{1-\frac{\varepsilon}{2}}\left(1+\frac{1}{i}\right)\int_{\Pi\cap B_{\left(1+\varepsilon\right)\delta^{-1}}\left(O\right)\setminus\,\cup_{Q\in e^{\frac{\sigma}{2}}\mathfrak{S}_{P}}B_{\left(1-\varepsilon\right)\delta}\left(Q\right)}\,v_{s_{i}}\left(y,0\right)\,dy
\]
\[
\leq\left(1+\varepsilon\right)\int_{\Pi\cap B_{\left(1+\varepsilon\right)\delta^{-1}}\left(O\right)\setminus\,\cup_{Q\in e^{\frac{\sigma}{2}}\mathfrak{S}_{P}}B_{\left(1-\varepsilon\right)\delta}\left(Q\right)}\,v_{s_{i}}\left(y,0\right)\,dy
\]
for $i\gg1$.
\end{proof}
What follows are the estimates on the upper and lower bound for the
difference function. Since the function is defined on a half-plane,
we first use the method of reflection and then use the equation satisfied
by the extension to derive the estimates. 
\begin{lem}
\label{Harnack estimate}Let $\delta>0$ and $\left\{ v_{s_{i}}\right\} $
be as in Proposition \ref{delta pinching sequence}. Define $\bar{v}_{s_{i}}$
as the even extension of $v_{s_{i}}$, i.e.
\[
\bar{v}_{s_{i}}\left(y,t\right)=\left\{ \begin{array}{c}
v_{s_{i}}\left(y,\sigma\right),\;y_{2}\geq0\\
v_{s_{i}}\left(\tilde{y},\sigma\right),\;y_{2}<0
\end{array},\right.
\]
where $\tilde{y}=\widetilde{\left(y_{1},y_{2}\right)}=\left(y_{1},-y_{2}\right)$.
Then $\bar{v}_{s_{i}}$ is $C^{2}$ away from $\underset{\sigma\in\mathbb{R}}{\cup}\,e^{\frac{\sigma}{2}}\left(\overline{\mathfrak{S}_{P}}\right)\times\left\{ \sigma\right\} $,
where $\overline{\mathfrak{S}_{P}}=\mathfrak{S}_{P}\cup\widetilde{\mathfrak{S}_{P}}$,
and it satisfies 
\begin{equation}
\partial_{\sigma}\bar{v}_{s_{i}}=\partial_{k}\left(\bar{\boldsymbol{g}}_{s_{i}}^{kl}\left(y,\sigma\right)\,\partial_{l}\bar{v}_{s_{i}}\right)+\left(-\frac{1}{2}y+\bar{b}_{s_{i}}\left(y,\sigma\right)\right)\cdot\nabla v_{s_{i}}+\left(\frac{1}{2}+\bar{c}_{s_{i}}\left(y,\sigma\right)\right)v_{s_{i}},\label{linear eq of extended type I graph}
\end{equation}
where
\[
\bar{\boldsymbol{g}}_{s_{i}}^{kl}\left(y,\sigma\right)=\left\{ \begin{array}{c}
\boldsymbol{g}_{s_{i}+\sigma}^{kl}\left(y,v_{s_{i}}^{m},\nabla v_{s_{i}}^{m}\right),\;y_{2}\geq0\\
\left(-1\right)^{k+l}\boldsymbol{g}_{s_{i}+\sigma}^{kl}\left(\tilde{y},v_{s_{i}}^{m}\left(\tilde{y},\sigma\right),\nabla v_{s_{i}}^{m}\left(\tilde{y},\sigma\right)\right),\;y_{2}<0
\end{array}\right.,\quad\textrm{for}\,\;k,l\in\left\{ 1,2\right\} ;
\]
\[
\bar{b}_{s_{i}}^{k}\left(y,\sigma\right)=\left\{ \begin{array}{c}
b_{s_{i}}^{k}\left(y,\sigma\right),\;y_{2}\geq0\\
\left(-1\right)^{k+1}b_{s_{i}}^{k}\left(\tilde{y},\sigma\right)\;y_{2}<0
\end{array}\right.,\quad\textrm{for}\,\;k\in\left\{ 1,2\right\} ;
\]
\[
\bar{c}_{s_{i}}\left(y,\sigma\right)=\left\{ \begin{array}{c}
c_{s_{i}}\left(y,\sigma\right),\;y_{2}\geq0\\
c_{s_{i}}\left(\tilde{y},\sigma\right)\;y_{2}<0
\end{array}\right..
\]
The functions $\bar{\boldsymbol{g}}_{s_{i}}^{kl}\left(y,\sigma\right),\bar{c}_{s_{i}}\left(y,\sigma\right)$
are Lipschitz in $y$ , $\bar{b}_{s_{i}}\in L^{\infty}$ and 
\begin{equation}
\bar{\boldsymbol{g}}_{s_{i}}^{kl}\rightarrow\boldsymbol{\delta}^{kl},\quad\left|\bar{b}_{s_{i}}\right|\,+\,\left|\bar{c}_{s_{i}}\right|\rightarrow0\label{coefficients bound for extended type I eq}
\end{equation}

Additionally, let 
\[
\sigma_{\delta}=\max\left\{ -2\,\ln\left(\frac{\delta}{2}\,\min_{Q\in\mathfrak{S}_{P}\setminus\left\{ O\right\} }\left|Q\right|\right),1\right\} 
\]
and fix $Q_{\delta}\in\Pi\cap B_{\frac{1}{2}\delta^{-1}}\left(O\right)\setminus B_{\frac{3}{2}\delta}\left(O\right)$.
Then for any given $\mathcal{T}>\sigma_{\delta}+1$, if $i\gg1$,
we have
\begin{itemize}
\item For $y\in B_{\frac{2}{3}\delta^{-1}}\left(O\right)\setminus B_{\frac{4}{3}\delta}\left(O\right)$,
$\sigma_{\delta}+1\leq\sigma\leq\mathcal{T}$, there holds 
\[
C\left(\delta,\mathcal{T}\right)^{-1}\bar{v}_{s_{i}}\left(Q_{\delta},\sigma_{\delta}\right)\leq\,\bar{v}_{s_{i}}\left(y,\sigma\right)\,\leq C\left(\delta\right)\bar{v}_{s_{i}}\left(Q_{\delta},\sigma_{\delta}\right);
\]
\item For $y\in B_{\frac{2}{3}\delta^{-1}}\left(O\right)\setminus B_{\frac{4}{3}\delta}\left(O\right)$,
$\sigma_{\delta}+1\leq\sigma\leq\sigma_{\delta}+2$, there holds 
\[
\bar{v}_{s_{i}}\left(y,\sigma\right)\,\geq C\left(\delta\right)^{-1}\bar{v}_{s_{i}}\left(Q_{\delta},\sigma_{\delta}\right).
\]
\end{itemize}
\end{lem}

\begin{proof}
The first part follows from (\ref{linear eq of type I graph}), (\ref{coefficients bound for type I eq})
and the reflection principle (see Lemma \ref{reflection principle}).
For the second part, note that $\sigma_{\delta}\geq1$ is chosen so
that 
\[
e^{\frac{\sigma_{\delta}}{2}}\min_{Q\in\mathfrak{S}_{P}\setminus\left\{ O\right\} }\left|Q\right|\geq2\delta^{-1}
\]
if $\mathfrak{S}_{P}\setminus\left\{ O\right\} \neq\emptyset$. In
particular, $v_{s_{i}}\left(\cdot,\sigma\right)$ is $C^{2}$ in $B_{2\delta^{-1}}\left(O\right)$
away from $O$ for $\sigma\geq\sigma_{\delta}$ and $i\gg1$. 

Given $\mathcal{T}>\sigma_{\delta}+1$, by (\ref{linear eq of extended type I graph}),
(\ref{coefficients bound for extended type I eq}) and the Harnack
inequality (cf. \cite{AS}), for $i\gg1$ we have 
\begin{itemize}
\item For $y\in B_{\frac{2}{3}\delta^{-1}}\left(O\right)\setminus B_{\frac{4}{3}\delta}\left(O\right),\,\sigma_{\delta}+1\leq\sigma\leq\mathcal{T}$,
\[
\bar{v}_{s_{i}}\left(\cdot,\sigma\right)\geq C\left(\delta,\mathcal{T}\right)^{-1}\bar{v}_{s_{i}}\left(Q_{\delta},\sigma_{\delta}\right);
\]
\item For $y\in B_{\frac{2}{3}\delta^{-1}}\left(O\right)\setminus B_{\frac{4}{3}\delta}\left(O\right),\,\sigma_{\delta}+1\leq\sigma\leq\sigma_{\delta}+2$,
\[
\bar{v}_{s_{i}}\left(\cdot,\sigma\right)\geq C\left(\delta\right)^{-1}\bar{v}_{s_{i}}\left(Q_{\delta},\sigma_{\delta}\right);
\]
\item For $y\in B_{\frac{5}{4}\delta^{-1}}\left(O\right)\setminus\,\cup_{Q\in\mathfrak{S}_{P}}B_{\frac{3}{4}\delta}\left(Q\right)$,
\[
\bar{v}_{s_{i}}\left(\cdot,0\right)\leq C\left(\delta\right)\bar{v}_{s_{i}}\left(Q_{\delta},\sigma_{\delta}\right).
\]
\end{itemize}
In particular, the last one yields
\[
\int_{B_{\frac{5}{4}\delta^{-1}}\left(O\right)\setminus\,\cup_{Q\in\mathfrak{S}_{P}}B_{\frac{3}{4}\delta}\left(Q\right)}\,\bar{v}_{s_{i}}\left(y,0\right)\,dy\,\leq C\left(\delta\right)\bar{v}_{s_{i}}\left(Q_{\delta},\sigma_{\delta}\right).
\]
To derive the upper bound for $\bar{v}_{s_{i}}$, we first use (\ref{delta pinching of graph})
and the above inequality to get 
\[
\sup_{\sigma_{\delta}\leq\sigma\leq\mathcal{T}}\,\int_{B_{\frac{3}{4}\delta^{-1}}\left(O\right)\setminus B_{\frac{5}{4}\delta}\left(O\right)}\,\bar{v}_{s_{i}}\left(y,\sigma\right)\,dy
\]
\[
\lesssim\int_{\Pi\cap B_{\frac{5}{4}\delta^{-1}}\left(O\right)\setminus B_{\frac{3}{4}\delta}\left(O\right)}\,\bar{v}_{s_{i}}\left(y,0\right)\,dy\,\leq C\left(\delta\right)\,\bar{v}_{s_{i}}\left(Q_{\delta},\sigma_{\delta}\right).
\]
By (\ref{linear eq of extended type I graph}), (\ref{coefficients bound for extended type I eq})
and the mean value inequality (cf. \cite{AS}), for $y\in B_{\frac{2}{3}\delta^{-1}}\left(O\right)\setminus B_{\frac{4}{3}\delta}\left(O\right)$,
$\sigma_{\delta}+1\leq\sigma\leq\mathcal{T}$, we have 
\[
\bar{v}_{s_{i}}\left(y_{0},\sigma_{0}\right)\,\lesssim\fint_{B_{\frac{\delta}{12}}\left(y_{0}\right)\times\left(\sigma_{0}-\left(\frac{\delta}{12}\right)^{2},\sigma\right)}\,\bar{v}_{s_{i}}\left(y,\sigma\right)\,dy\,d\sigma
\]
\[
\lesssim\frac{1}{\delta^{2}}\,\sup_{\sigma_{0}-\left(\frac{\delta}{12}\right)^{2}<\sigma<\sigma_{0}}\,\int_{B_{\frac{3}{4}\delta^{-1}}\left(O\right)\setminus B_{\frac{5}{4}\delta}\left(O\right)}\,\bar{v}_{s_{i}}\left(y,\sigma\right)\,dy\,\leq C\left(\delta\right)\bar{v}_{s_{i}}\left(Q_{\delta},\sigma_{\delta}\right).
\]
\end{proof}
Now we are in a position to prove the unity of Huisken's density.
Our proof follows closely the arguments in \cite{LW} (see also \cite{CM}).
\begin{prop}
\label{unit Huisken density}(Unity of Huisken's Density)

Let $P$ be a limit point of $\left\{ \boldsymbol{\Sigma}_{t}\right\} _{0\leq t<T}$
as $t\nearrow T$. Then 
\[
\Theta_{\left\{ \boldsymbol{\Sigma}_{t}\right\} }\left(P,T\right)=\left\{ \begin{array}{c}
1,\quad P\in\boldsymbol{U}\\
\frac{1}{2},\quad P\in\boldsymbol{\Gamma}
\end{array}\right..
\]
\end{prop}

\begin{proof}
We will focus on the case where $P\in\boldsymbol{\Gamma}$ since the
argument for $P\in\boldsymbol{U}$ is similar. 

Note that the mean convexity of $\boldsymbol{\Gamma}$ yields $\Theta_{\left\{ \boldsymbol{\Sigma}_{t}\right\} }\left(P,T\right)\geq\frac{1}{2}$
(cf. \cite{K}). Suppose that $\Theta_{\left\{ \boldsymbol{\Sigma}_{t}\right\} }\left(P,T\right)>\frac{1}{2}$,
then from Proposition \ref{small energy implies regularity}, Proposition
\ref{local graph thm} and Lemma \ref{Huisken's density}, we know
that $P\in\mathcal{S}$ (see Section \ref{Hypotheses}). Below we
will derive a contradiction in three steps and hence prove the proposition. 

$\mathbf{Step\,\,1}$: \textit{Prove that 
\[
\int_{\mathbb{R}^{2}}\left(\left|\nabla\eta\right|^{2}-\frac{1}{2}\eta^{2}\right)\,e^{-\frac{\left|y\right|^{2}}{4}}\,dy\,\geq0
\]
for any function $\eta\left(y_{1},y_{2}\right)\in C_{c}\left(\mathbb{R}^{2}\right)\cap W^{1,2}\left(\mathbb{R}^{2}\right)$
satisfying $\eta\left(0,0\right)=0$. }

\textit{Proof of Step 1.} By approximation, it suffices to show the
following. Given $0<\delta<1$ and function $\eta\left(y_{1},y_{2}\right)\in C_{c}^{1}\left(B_{\frac{1}{2}\delta^{-1}}\left(O\right)\setminus B_{\frac{3}{2}\delta}\left(O\right)\right)$,
there holds
\[
\int_{\mathbb{R}^{2}}\left(\left|\nabla\eta\right|^{2}-\frac{1}{2}\eta^{2}\right)\,e^{-\frac{\left|y\right|^{2}}{4}}\,dy\,\geq0.
\]
For that purpose, let $\left\{ \bar{v}_{s_{i}}\right\} $ be the sequence
of functions in Lemma \ref{Harnack estimate}. Define 
\[
w_{s_{i}}\left(y,\sigma\right)=\frac{\bar{v}_{s_{i}}\left(y,\sigma\right)}{\bar{v}_{s_{i}}\left(Q_{\delta},\sigma_{\delta}\right)}.
\]
By (\ref{linear eq of extended type I graph}), we get 
\[
\partial_{\sigma}w_{s_{i}}=\partial_{k}\left(\bar{\boldsymbol{g}}_{s_{i}}^{kl}\left(y,\sigma\right)\,\partial_{l}w_{s_{i}}\right)+\left(-\frac{1}{2}y+\bar{b}_{s_{i}}\left(y,\sigma\right)\right)\cdot\nabla w_{s_{i}}+\left(\frac{1}{2}+\bar{c}_{s_{i}}\left(y,\sigma\right)\right)w_{s_{i}}
\]
Note that $w_{s_{i}}\left(y_{1},y_{2},\sigma\right)$ is an even function
in $y_{2}$. For $i\gg1$, Lemma \ref{Harnack estimate} implies that
\[
C\left(\delta,\sigma\right)^{-1}\leq w_{s_{i}}\left(y,\sigma\right)\leq C\left(\delta\right)
\]
for $y\in B_{\frac{2}{3}\delta^{-1}}\left(O\right)\setminus B_{\frac{4}{3}\delta}\left(O\right)$,
$\sigma\geq\sigma_{\delta}+1$, and 
\[
w_{s_{i}}\left(y,\sigma\right)\geq C\left(\delta\right)^{-1}
\]
for $y\in B_{\frac{2}{3}\delta^{-1}}\left(O\right)\setminus B_{\frac{4}{3}\delta}\left(O\right)$,
$\sigma_{\delta}+1\leq\sigma\leq\sigma_{\delta}+2$. By (\ref{coefficients bound for extended type I eq})
and H$\ddot{\textrm{o}}$lder estimates (cf. \cite{AS}), there exists
$0<\alpha<1$ (which is independent of $s_{i}$) so that 
\[
\left[w_{s_{i}}\right]_{\alpha}\lesssim\frac{1}{\delta^{\alpha}}\left\Vert w_{s_{i}}\right\Vert _{L^{\infty}}\leq C\left(\delta\right).
\]
Consequently, there exists a positive function $w\left(y,\sigma\right)$
so that, after passing to a subsequence, $w_{s_{i}}$ converges locally
uniformly to $w$ (by Arzel$\grave{\textrm{a}}$-Ascoli theorem) and
$\nabla w_{s_{i}}$ converges weakly in $L_{loc}^{2}$ to $\nabla w$
(by Caccioppoli estimate) on $y\in\left(B_{\frac{1}{2}\delta^{-1}}\left(O\right)\setminus B_{\frac{3}{2}\delta}\left(O\right)\right)$,
$\sigma\geq\sigma_{\delta}+2$ as $i\rightarrow\infty$. Note that
$w\left(y_{1},y_{1},\sigma\right)$ is even in $y_{2}$ and satisfies
\begin{equation}
C\left(\delta,\sigma\right)^{-1}\leq w\left(y,\sigma\right)\leq C\left(\delta\right)\label{upper bound for w}
\end{equation}
for $y\in B_{\frac{1}{2}\delta^{-1}}\left(O\right)\setminus B_{\frac{3}{2}\delta}\left(O\right)$,
$\sigma\geq\sigma_{\delta}+1$, and 
\begin{equation}
w\left(y,\sigma\right)\geq C\left(\delta\right)^{-1}\label{lower bound for w}
\end{equation}
for $y\in B_{\frac{1}{2}\delta^{-1}}\left(O\right)\setminus B_{\frac{3}{2}\delta}\left(O\right)$,
$\sigma_{\delta}+1\leq\sigma\leq\sigma_{\delta}+2$. 

It follows that $\ln w_{s_{i}}$ converges locally uniformly to $\ln w$,
and $\nabla\ln w_{s_{i}}$ converges weakly in $L_{loc}^{2}$ to $\nabla\ln w$
as $i\rightarrow\infty$. Additionally, by taking the logarithm of
the equation for $w_{s_{i}}$, we have
\[
\partial_{\sigma}\ln w_{s_{i}}-\partial_{k}\left(\bar{\boldsymbol{g}}_{s_{i}}^{kl}\left(y,\sigma\right)\,\partial_{l}\ln w_{s_{i}}\right)
\]
\[
=\bar{\boldsymbol{g}}_{s_{i}}^{kl}\left(y,\sigma\right)\,\partial_{k}\ln w_{s_{i}}\,\partial_{l}\ln w_{s_{i}}+\left(-\frac{1}{2}y+\bar{b}_{s_{i}}\left(y,\sigma\right)\right)\cdot\nabla\ln w_{s_{i}}+\left(\frac{1}{2}+\bar{c}_{s_{i}}\left(y,\sigma\right)\right).
\]
Multiplication of the above equation by $\eta^{2}\left(y\right)e^{-\frac{\left|y\right|^{2}}{4}}$
and integration gives
\[
\left.\int_{\mathbb{R}^{2}}\ln w_{s_{i}}\left(y,\sigma\right)\,\eta^{2}\left(y\right)e^{-\frac{\left|y\right|^{2}}{4}}\,dy\right|_{\sigma=\sigma_{\delta}+2}^{\mathcal{T}}\,+\,\int_{\sigma_{\delta}+2}^{\mathcal{T}}\int_{\mathbb{R}^{2}}\bar{\boldsymbol{g}}_{s_{i}}^{kl}\left(y,\sigma\right)\,\partial_{l}\ln w_{s_{i}}\,\left(\partial_{k}\eta^{2}-\frac{1}{2}y_{k}\eta^{2}\right)e^{-\frac{\left|y\right|^{2}}{4}}\,dy\,d\sigma
\]
\[
=\int_{\sigma_{\delta}+2}^{\mathcal{T}}\int_{\mathbb{R}^{2}}\bar{\boldsymbol{g}}_{s_{i}}^{kl}\left(y,\sigma\right)\,\partial_{k}\ln w_{s_{i}}\,\partial_{l}\ln w_{s_{i}}\,\eta^{2}\left(y\right)e^{-\frac{\left|y\right|^{2}}{4}}\,dy\,d\sigma
\]
\[
+\int_{\sigma_{\delta}+2}^{\mathcal{T}}\int_{\mathbb{R}^{2}}\left\{ \left(-\frac{1}{2}y+\bar{b}_{s_{i}}\left(y,\sigma\right)\right)\cdot\nabla\ln w_{s_{i}}+\left(\frac{1}{2}+\bar{c}_{s_{i}}\left(y,\sigma\right)\right)\right\} \eta^{2}\left(y\right)e^{-\frac{\left|y\right|^{2}}{4}}\,dy\,d\sigma
\]
for any $\sigma_{\delta}+2<\mathcal{T}<\infty$. Letting $i\rightarrow\infty$
gives 
\[
\left.\int_{\mathbb{R}^{2}}\ln w\left(y,\sigma\right)\,\eta^{2}e^{-\frac{\left|y\right|^{2}}{4}}\,dy\right|_{\sigma=\sigma_{\delta}+2}^{\mathcal{T}}\,+\,\int_{\sigma_{\delta}+2}^{\mathcal{T}}\int_{\mathbb{R}^{2}}\nabla\ln w\cdot\left(\nabla\eta^{2}-\frac{1}{2}y\eta^{2}\right)e^{-\frac{\left|y\right|^{2}}{4}}\,dy\,d\sigma
\]
\[
=\int_{\sigma_{\delta}+2}^{\mathcal{T}}\int_{\mathbb{R}^{2}}\left|\nabla\ln w\right|^{2}\,\eta^{2}e^{-\frac{\left|y\right|^{2}}{4}}\,dy\,d\sigma\,+\,\int_{\sigma_{\delta}+2}^{\mathcal{T}}\int_{\mathbb{R}^{2}}\left\{ -\frac{1}{2}y\cdot\nabla\ln w+\frac{1}{2}\right\} \eta^{2}e^{-\frac{\left|y\right|^{2}}{4}}\,dy\,d\sigma.
\]
After completing the square, we get 
\[
\int_{\mathbb{R}^{2}}\ln\frac{w\left(y,\mathcal{T}\right)}{w\left(y,\sigma_{\delta}+2\right)}\,\eta^{2}e^{-\frac{\left|y\right|^{2}}{4}}\,dy\,+\,\int_{\sigma_{\delta}+2}^{\mathcal{T}}\int_{\mathbb{R}^{2}}\left(\left|\nabla\eta\right|^{2}-\frac{1}{2}\eta^{2}\right)\,e^{-\frac{\left|y\right|^{2}}{4}}\,dy\,d\sigma
\]
\[
=\int_{\sigma_{\delta}+2}^{\mathcal{T}}\int_{\mathbb{R}^{2}}\left|\nabla\eta-\eta\nabla\ln w\right|^{2}\,e^{-\frac{\left|y\right|^{2}}{4}}\,dy\,d\sigma.
\]
It follows that
\[
\int_{\mathbb{R}^{2}}\left(\left|\nabla\eta\right|^{2}-\frac{1}{2}\eta^{2}\right)\,e^{-\frac{\left|y\right|^{2}}{4}}\,dy\,\geq\,\frac{1}{\mathcal{T}-\left(\sigma_{\delta}+2\right)}\int_{\mathbb{R}^{2}}\ln\frac{w\left(y,\sigma_{\delta}+2\right)}{w\left(y,\mathcal{T}\right)}\,\eta^{2}e^{-\frac{\left|y\right|^{2}}{4}}\,dy
\]
\[
=\int_{\mathbb{R}^{2}}\ln\left(\frac{w\left(y,\sigma_{\delta}+2\right)}{w\left(y,\mathcal{T}\right)}\right)^{\frac{1}{\mathcal{T}-\left(\sigma_{\delta}+2\right)}}\,\eta^{2}e^{-\frac{\left|y\right|^{2}}{4}}\,dy.
\]
By (\ref{upper bound for w}) and (\ref{lower bound for w}), we have
\[
\left(\frac{w\left(y,\sigma_{\delta}+2\right)}{w\left(y,\mathcal{T}\right)}\right)^{\frac{1}{\mathcal{T}-\left(\sigma_{\delta}+2\right)}}\geq C\left(\delta\right)^{-\frac{1}{\mathcal{T}-\left(\sigma_{\delta}+2\right)}}\rightarrow1\qquad\textrm{as}\quad\mathcal{T}\nearrow\infty
\]
Thus, we have $\int_{\mathbb{R}^{2}}\left(\left|\nabla\eta\right|^{2}-\frac{1}{2}\eta^{2}\right)\,e^{-\frac{\left|y\right|^{2}}{4}}\,dy\,\geq0.$
Q.E.D.

$\mathbf{Step\,\,2}$: \textit{Prove that 
\[
\int_{\mathbb{R}^{2}}\left(\left|\nabla\eta\right|^{2}-\frac{1}{2}\eta^{2}\right)\,e^{-\frac{\left|y\right|^{2}}{4}}\,dy\,\geq0
\]
for any function $\eta\left(y_{1},y_{2}\right)\in C_{c}\left(\mathbb{R}^{2}\right)\cap W^{1,2}\left(\mathbb{R}^{2}\right)$.
Note that $\eta$ does not need to vanish at $\left(0,0\right)$ as
in Step 1. }

\textit{Proof of Step 2.} By approximation, it suffices to show the
following. Given a function $\eta\left(y_{1},y_{2}\right)\in C_{c}^{1}\left(\mathbb{R}^{2}\right)$,
there holds
\[
\int_{\mathbb{R}^{2}}\left(\left|\nabla\eta\right|^{2}-\frac{1}{2}\eta^{2}\right)\,e^{-\frac{\left|y\right|^{2}}{4}}\,dy\,\geq0.
\]
To achieve that, for every $0<\delta\ll1$, let 
\[
\psi_{\delta}\left(\xi\right)=2\delta\,\xi^{2\delta-1}\quad\forall\;\,0<\xi<1.
\]
By a simple calculation, we have 
\[
\int_{0}^{1}\psi_{\delta}\left(\xi\right)d\xi=1\quad\textrm{and}\quad\int_{0}^{1}\psi_{\delta}^{2}\left(\xi\right)\xi\,d\xi=\delta.
\]
Now define
\[
\zeta_{\delta}\left(r\right)=\left\{ \begin{array}{c}
\int_{0}^{r}\frac{1}{\delta}\,\psi_{\delta}\left(\frac{\rho}{\delta}\right)d\rho,\qquad0\leq r\leq\delta\\
1,\qquad r>\delta
\end{array}\right..
\]
Notice that $\zeta_{\delta}\left(r\right)\in C\left[0,\infty\right)$,
$\zeta_{\delta}\left(0\right)=0$, 
\[
\int_{0}^{\infty}\left|\zeta'_{\delta}\left(r\right)\right|^{2}r\,dr=\int_{0}^{\delta}\frac{1}{\delta^{2}}\,\psi_{\delta}^{2}\left(\frac{r}{\delta}\right)r\,dr=\delta,
\]
and $\zeta_{\delta}\rightarrow1$ as $\delta\searrow0$. 

Next, let's define 
\[
\eta_{\delta}\left(y\right)=\zeta_{\delta}\left(\left|y\right|\right)\,\eta\left(y\right).
\]
Then $\eta_{\delta}\in C_{c}\left(\mathbb{R}^{2}\right)\cap W^{1,2}\left(\mathbb{R}^{2}\right)$
and $\eta_{\delta}\left(0,0\right)=0$. It follows from Step 1 that
\[
\int_{\mathbb{R}^{2}}\left(\left|\nabla\eta_{\delta}\right|^{2}-\frac{1}{2}\eta_{\delta}^{2}\right)\,e^{-\frac{\left|y\right|^{2}}{4}}\,dy\,\geq0.
\]
Note that
\[
\nabla\eta_{\delta}\left(y\right)=\zeta_{\delta}\left(\left|y\right|\right)\nabla\eta\left(y\right)+\zeta'_{\delta}\left(\left|y\right|\right)\frac{y}{\left|y\right|}\,\eta\left(y\right)
\]
and that 
\[
\int_{\mathbb{R}^{2}}\left|\zeta'_{\delta}\left(\left|y\right|\right)\eta\left(y\right)\right|^{2}\,e^{-\frac{\left|y\right|^{2}}{4}}\,dy\,\leq\,2\pi\left\Vert \eta\right\Vert _{L^{\infty}}^{2}\int_{0}^{\infty}\left|\zeta'_{\delta}\left(r\right)\right|^{2}rdr\,=\,2\pi\left\Vert \eta\right\Vert _{L^{\infty}}^{2}\delta.
\]
Letting $\delta\searrow0$ gives $\int_{\mathbb{R}^{2}}\left(\left|\nabla\eta\right|^{2}-\frac{1}{2}\eta^{2}\right)\,e^{-\frac{\left|y\right|^{2}}{4}}\,dy\,\geq0.$
Q.E.D.

$\mathbf{Step\,\,3}$: For each $R>0$, let \textit{
\[
\eta_{R}\left(y\right)=\left\{ \begin{array}{c}
1,\quad\left|y\right|\leq R\\
R+1-\left|y\right|,\quad R<\left|y\right|\leq R+1\\
0,\quad\left|y\right|>R+1
\end{array}\right..
\]
}Then we have \textit{
\[
\int_{\mathbb{R}^{2}}\left(\left|\nabla\eta_{R}\right|^{2}-\frac{1}{2}\eta_{R}^{2}\right)\,e^{-\frac{\left|y\right|^{2}}{4}}\,dy\,\rightarrow-\frac{1}{2}\int_{\mathbb{R}^{2}}\,e^{-\frac{\left|y\right|^{2}}{4}}\,dy\,=-2\pi
\]
}as $R\rightarrow\infty$. However, by Step 2 we should have \textit{
\[
\int_{\mathbb{R}^{2}}\left(\left|\nabla\eta_{R}\right|^{2}-\frac{1}{2}\eta_{R}^{2}\right)\,e^{-\frac{\left|y\right|^{2}}{4}}\,dy\,\geq0\quad\forall\,R>0.
\]
}Thus, we get the desired contradiction.
\end{proof}
Thanks to Proposition \ref{unit Huisken density} , Allard's regularity
theorem (cf. \cite{Al}) and Proposition \ref{Li-Wang curvature estimate}
, now we can improve Proposition \ref{condensation compactness for MCF}
as follows. 
\begin{cor}
\label{flat tangent flow}Given a sequence $\left\{ \lambda_{i}\searrow0\right\} _{i\in\mathbb{N}}$,
there exist a half plane $\Pi$ so that the a subsequence of (\ref{parabolic rescaling})
converges to $\left\{ \Pi\right\} _{-\infty<\tau<0}$ with multiplicity
one. The half plane $\Pi$ meets $\lim_{i\rightarrow\infty}\boldsymbol{\Gamma}^{P,\lambda_{i}}\simeq\mathbb{R}^{2}$
orthogonally. 

An analogous result holds for $P\in\mathcal{S}\cap\boldsymbol{U}$,
in which case $\Pi$ is a plane. 
\end{cor}

\begin{proof}
By Proposition \ref{condensation compactness for MCF}, Lemma \ref{condensation compactness for NMCF}
and Proposition \ref{unit Huisken density}, there exist a half plane
$\Pi$ (with free boundary on $\lim_{i\rightarrow\infty}\boldsymbol{\Gamma}^{P,\lambda_{i}}\simeq\mathbb{R}^{2}$)
so that, after passing to a subsequence, $\left\{ \boldsymbol{\Sigma}_{\tau}^{\left(P,T\right),\lambda_{i}}\right\} \rightarrow\left\{ \Pi\right\} $
with multiplicity one away from $\mathfrak{S}_{P}\times\left(-\infty,0\right)$,
where $\mathfrak{S}_{P}$ is defined by (\ref{energy-concentrated points}). 

Let $\widetilde{\boldsymbol{\Sigma}_{\tau}^{\left(P,T\right),\lambda_{i}}}$
be the reflection of $\boldsymbol{\Sigma}_{\tau}^{\left(P,T\right),\lambda_{i}}$
on $B_{\frac{1}{\lambda_{i}\kappa}}\left(O\right)$ with respect to
$\boldsymbol{\Gamma}^{P,\lambda_{i}}$, and define
\[
\overline{\boldsymbol{\Sigma}_{\tau}^{\left(P,T\right),\lambda_{i}}}=\boldsymbol{\Sigma}_{\tau}^{\left(P,T\right),\lambda_{i}}\cup\widetilde{\boldsymbol{\Sigma}_{\tau}^{\left(P,T\right),\lambda_{i}}}.
\]
By (\ref{rescaled area ratio bound}) and (\ref{rescaled mean curvature bound}),
$\overline{\boldsymbol{\Sigma}_{-1}^{\left(P,T\right),\lambda_{i}}}$
converges weakly to $\bar{\Pi}$ in the sense of varifolds as $i\rightarrow\infty$,
where $\bar{\Pi}=\Pi\cup\tilde{\Pi}$ is a plane. Suppose that $\mathfrak{S}_{P}\neq\emptyset$,
then pick $Q\in\mathfrak{S}_{P}$. By the weak convergence of varifolds
and a similar argument as in Remark \ref{isolated singularities},
for almost every $r>0$, we have
\[
\frac{\mathcal{H}^{2}\left(\overline{\boldsymbol{\Sigma}_{-1}^{\left(P,T\right),\lambda_{i}}}\cap B_{r}\left(Q\right)\right)}{\pi r^{2}}\rightarrow1\qquad\textrm{as}\quad i\rightarrow\infty.
\]
By (\ref{rescaled mean curvature bound}) and Allard's regularity
theorem (cf. \cite{Al}), it follows that $\boldsymbol{\Sigma}_{-1}^{\left(P,T\right),\lambda_{i}}$
is a graph of a function whose $C^{1,\alpha}$ norm are uniformly
bounded (independent of $i\in\mathbb{N}$) for any fixed $0<\alpha<1$.
Applying Proposition \ref{Li-Wang curvature estimate} to the MCF
$\left\{ \boldsymbol{\Sigma}_{\tau}^{\left(P,T\right),\lambda_{i}}\right\} _{-1\leq\tau<0}$,
we then get a uniform bound (independent of $i\in\mathbb{N}$) on
the second fundamental form of $\boldsymbol{\Sigma}_{-1}^{\left(P,T\right),\lambda_{i}}$
near $Q$. However, this contradicts with the choice of $Q$ (see
(\ref{energy-concentrated points})).
\end{proof}

\section{\label{Proof of the Main Theorem}Proof of the Main Theorem}

This section is devoted to prove our main theorem as stated below.
\begin{thm}
\label{main thm}The MCF $\left\{ \boldsymbol{\Sigma}_{t}\right\} _{0\leq t<T}$
given in Section \ref{Hypotheses} can be extended beyond time $T$.
\end{thm}

\begin{proof}
By \cite{S}, it suffices to show that the second fundamental form
of $\left\{ \boldsymbol{\Sigma}_{t}\right\} _{0\leq t<T}$ is uniformly
bounded. Below we will prove that by contradiction.

Suppose that 
\[
\limsup_{t\rightarrow T}\left\Vert A_{\boldsymbol{\Sigma}_{t}}\right\Vert _{L^{\infty}}=\infty.
\]
Then choose a sequence $\left\{ \left(P_{i},t_{i}\right)\right\} _{i\in\mathbb{N}}$
so that $t_{i}\nearrow T$, $P_{i}\in\boldsymbol{\Sigma}_{t_{i}}$
and 
\[
\left|A_{\boldsymbol{\Sigma}_{t_{i}}}\left(P_{i}\right)\right|=\sup_{0\leq\tau\leq T-\frac{1}{i}}\left\Vert A_{\boldsymbol{\Sigma}_{\tau}}\right\Vert _{L^{\infty}}\rightarrow\infty\qquad\textrm{as}\;\,\,i\rightarrow\infty.
\]
Since $\boldsymbol{\Sigma}_{0}$ is compact and $T<\infty$, (\ref{mean curvature bound})
implies (after passing to a subsequence) that $P_{i}\rightarrow P$.
Here we have three possibilities to consider: 
\begin{itemize}
\item \textit{Case 1: }$P\in\boldsymbol{U}$;
\item \textit{Case 2: }$P\in\boldsymbol{\Gamma}$ and $\liminf_{i\rightarrow\infty}\textrm{dist}\left(P_{i},\partial\boldsymbol{\Sigma}_{t_{i}}\right)\left|A_{\boldsymbol{\Sigma}_{t_{i}}}\left(P_{i}\right)\right|<\infty$;
\item \textit{Case 3: }$P\in\boldsymbol{\Gamma}$ and $\liminf_{i\rightarrow\infty}\textrm{dist}\left(P_{i},\partial\boldsymbol{\Sigma}_{t_{i}}\right)\left|A_{\boldsymbol{\Sigma}_{t_{i}}}\left(P_{i}\right)\right|=\infty$.
\end{itemize}
Since Case 1 has been studied in \cite{LW} by using Proposition \ref{unit Huisken density}
and White's regularity theorem (cf. \cite{Wh}), we will focus on
the remaining two cases. Actually, the argument for Case 2 is similar
to that for Case 1. 

$\mathbf{Case\;2}$ ($P\in\boldsymbol{\Gamma}$ and $\liminf_{i\rightarrow\infty}\textrm{dist}\left(P_{i},\partial\boldsymbol{\Sigma}_{t_{i}}\right)\left|A_{\boldsymbol{\Sigma}_{t_{i}}}\left(P_{i}\right)\right|<\infty$): 

By passing to a subsequence, we may assume that 
\[
\textrm{dist}\left(P_{i},\partial\boldsymbol{\Sigma}_{t_{i}}\right)A_{i}\leq R<\infty
\]
for all $i\in\mathbb{N}$, where $A_{i}=\left|A_{\boldsymbol{\Sigma}_{t_{i}}}\left(P_{i}\right)\right|$.
Choose $\mathring{P}_{i}\in\partial\boldsymbol{\Sigma}_{t_{i}}$ so
that $\textrm{dist}\left(P_{i},\partial\boldsymbol{\Sigma}_{t_{i}}\right)=\left|P_{i}-\mathring{P}_{i}\right|$.
Clearly, we have $\mathring{P}_{i}\rightarrow P$ as $i\rightarrow\infty$. 

Given $\varepsilon>0$, by Proposition \ref{unit Huisken density},
Lemma \ref{monotonicity formula for MCF}, Lemma \ref{Huisken's density}
and the continuous dependence of the Gaussian integral on the parameters,
there exists $\delta>0$ so that for $i\gg1$ we have
\[
\frac{1}{2}\leq\sup_{t_{i}-\delta^{2}\leq t<t_{i}}\,\int_{\boldsymbol{\Sigma}_{t}}e^{85\left(\boldsymbol{\kappa}^{2}\left(t_{i}-t\right)\right)^{\frac{2}{5}}}\eta_{\boldsymbol{\Gamma};\mathring{P}_{i},t_{i}}\,\Psi_{\boldsymbol{\Gamma};\mathring{P}_{i},t_{i}}\left(X,t\right)\,d\mathcal{H}^{2}\left(X\right)\leq\,\frac{1+\varepsilon}{2}
\]
(cf. \cite{E}). After passing to a subsequence, we may assume that
\[
\frac{1}{2}\leq\sup_{t_{i}-\left(\frac{1}{A_{i}}\right)^{2}\leq t<t_{i}}\,\int_{\boldsymbol{\Sigma}_{t}}e^{85\left(\boldsymbol{\kappa}^{2}\left(t_{i}-t\right)\right)^{\frac{2}{5}}}\eta_{\boldsymbol{\Gamma};\mathring{P}_{i},t_{i}}\,\Psi_{\boldsymbol{\Gamma};\mathring{P}_{i},t_{i}}\left(X,t\right)\,d\mathcal{H}^{2}\left(X\right)\leq\,\frac{1+\frac{1}{i}}{2}\quad\forall\,\,i.
\]

Let 
\[
\hat{\boldsymbol{\Sigma}}_{\tau}^{i}=A_{i}\left(\Sigma_{t_{i}+\frac{\tau}{A_{i}^{2}}}^{i}-\mathring{P}_{i}\right),\qquad\mathcal{P}_{i}=A_{i}\left(P_{i}-\mathring{P}_{i}\right).
\]
Since $\mathcal{P}_{i}\in B_{2R}\left(O\right)$ for $i\gg1$, we
may assume (by passing to a subsequence) that $\mathcal{P}_{i}\rightarrow\mathcal{P}$.
In addition, we have 
\begin{itemize}
\item $\sup_{-t_{i}A_{i}^{2}\leq\tau\leq0}\left\Vert A_{\hat{\boldsymbol{\Sigma}}_{\tau}^{i}}\right\Vert _{L^{\infty}}\leq1=\left|A_{\hat{\boldsymbol{\Sigma}}_{0}^{i}}\left(\mathcal{P}_{i}\right)\right|$;
\item $\hat{\boldsymbol{U}}_{i}\coloneqq A_{i}\left(\boldsymbol{U}-\mathring{P}_{i}\right)\rightarrow\mathbb{R}_{+}^{3},\quad\hat{\boldsymbol{\Gamma}}_{i}\coloneqq\partial\hat{\boldsymbol{U}}_{i}=A_{i}\left(\boldsymbol{\Gamma}-\mathring{P}_{i}\right)$
satisfies $\boldsymbol{\kappa}_{i}=\frac{\boldsymbol{\kappa}}{A_{i}}$
graph condition and converges to $\partial\mathbb{R}_{+}^{3}\simeq\mathbb{R}^{2}$;
\item $\frac{1}{2}\leq\int_{\hat{\boldsymbol{\Sigma}}_{\tau}^{i}}e^{85\left(\boldsymbol{\kappa}_{i}^{2}\left(-\tau\right)\right)^{\frac{2}{5}}}\eta_{\hat{\boldsymbol{\Gamma}}_{i};O,0}\,\Psi_{\hat{\boldsymbol{\Gamma}}_{i};O,0}\left(X,\tau\right)\,d\mathcal{H}^{2}\left(X\right)\leq\,\frac{1+\frac{1}{i}}{2}\quad\forall\,-1\leq\tau<0$. 
\end{itemize}
It follows from Proposition \ref{compactness} that $\left\{ \hat{\boldsymbol{\Sigma}}_{\tau}^{i}\right\} _{-t_{i}A_{i}^{2}\leq\tau\leq0}\rightarrow\left\{ \hat{\boldsymbol{\Sigma}}_{\tau}\right\} _{-\infty<\tau\leq0}$.
The limiting MCF has free boundary on $\lim_{i\rightarrow\infty}\hat{\boldsymbol{\Gamma}}_{i}\simeq\mathbb{R}^{2}$
and it satisfies 
\[
\left|A_{\hat{\boldsymbol{\Sigma}}_{0}}\left(\mathcal{P}\right)\right|=1,
\]
\[
\int_{\hat{\boldsymbol{\Sigma}}_{\tau}}\Psi_{\mathbb{R}^{2};O,0}\left(X,t\right)\,d\mathcal{H}^{2}\left(X\right)=\,\frac{1}{2}\qquad\forall\,-1\leq\tau<0.
\]
Then Lemma \ref{monotonicity formula for MCF} implies that $\left\{ \hat{\boldsymbol{\Sigma}}_{\tau}\right\} _{-1\leq\tau<0}$
is self-shrinking. Namely, it satisfies 
\[
H_{\hat{\boldsymbol{\Sigma}}_{\tau}}+\frac{X_{\tau}\cdot N_{\hat{\boldsymbol{\Sigma}}_{\tau}}}{2\left(-\tau\right)}=0\qquad\forall\,\,X_{\tau}\in\hat{\boldsymbol{\Sigma}}_{\tau},\,-1\leq\tau<0
\]
(cf. \cite{Wh,B}). In particular, we get
\[
X_{\tau}\cdot N_{\hat{\boldsymbol{\Sigma}}_{\tau}}=2\tau H_{\hat{\boldsymbol{\Sigma}}_{\tau}}\rightarrow0\qquad\textrm{as}\quad\tau\nearrow0.
\]
Consequently, $\hat{\boldsymbol{\Sigma}}_{0}$ must be a $C^{2}$
cone (with vertex at $O$) and hence a plane. This contradicts with
the property that $\left|A_{\hat{\boldsymbol{\Sigma}}_{0}}\left(\mathcal{P}\right)\right|=1$. 

$\mathbf{Case\;3}$ ($P\in\boldsymbol{\Gamma}$ and $\liminf_{i\rightarrow\infty}\textrm{dist}\left(P_{i},\partial\boldsymbol{\Sigma}_{t_{i}}\right)\left|A_{\boldsymbol{\Sigma}_{t_{i}}}\left(P_{i}\right)\right|=\infty$): 

Firstly, we claim that $\left|P-P_{i}\right|\geq\sqrt{T-t_{i}}$ for
all but finitely many $i\in\mathbb{N}$. For otherwise, after passing
to a subsequence, we would have
\[
\left|P-P_{i}\right|<\sqrt{T-t_{i}}\qquad\forall\;i.
\]
By Corollary \ref{flat tangent flow}, $\left\{ \frac{1}{\sqrt{T-t_{i}}}\left(\boldsymbol{\Sigma}_{t_{i}}-P\right)\right\} _{i\in\mathbb{N}}$
converges to a half-plane and hence
\[
\sqrt{T-t_{i}}\,\left|A_{\boldsymbol{\Sigma}_{t_{i}}}\left(P_{i}\right)\right|\rightarrow0\qquad\textrm{as}\quad i\rightarrow\infty,
\]
which implies
\[
\textrm{dist}\left(\partial\boldsymbol{\Sigma}_{t_{i}},P\right)\geq\textrm{dist}\left(\partial\boldsymbol{\Sigma}_{t_{i}},P_{i}\right)-\left|P-P_{i}\right|\gg\left|A_{\boldsymbol{\Sigma}_{t_{i}}}\left(P_{i}\right)\right|^{-1}-\sqrt{T-t_{i}}\gg\sqrt{T-t_{i}}
\]
for $i\gg1$. However, (\ref{mean curvature bound}) yields that
\begin{equation}
\textrm{dist}\left(\partial\boldsymbol{\Sigma}_{t_{i}},P\right)\leq\boldsymbol{\Lambda}\left(T-t_{i}\right)\label{traveling speed}
\end{equation}
(cf. \cite{K}). Thus, we get a contradiction. 

By passing to a subsequence, let's assume that 
\begin{equation}
\left|P-P_{i}\right|\geq\sqrt{T-t_{i}}\qquad\forall\;i.\label{pseudo-interior seq}
\end{equation}
It follows from (\ref{traveling speed}) and (\ref{pseudo-interior seq})
that
\[
\textrm{dist}\left(P,\partial\boldsymbol{\Sigma}_{t_{i}}\right)\leq\boldsymbol{\Lambda}\left(T-t_{i}\right)\ll\sqrt{T-t_{i}}\leq\left|P-P_{i}\right|
\]
for $i\gg1$, which implies
\[
\left|P_{i}-P\right|\geq\textrm{dist}\left(P_{i},\partial\boldsymbol{\Sigma}_{t_{i}}\right)-\textrm{dist}\left(P,\partial\boldsymbol{\Sigma}_{t_{i}}\right)\geq\textrm{dist}\left(P_{i},\partial\boldsymbol{\Sigma}_{t_{i}}\right)-\frac{1}{2}\left|P_{i}-P\right|.
\]
Thus we have
\begin{equation}
\left|P_{i}-P\right|\,\geq\,\frac{2}{3}\,\textrm{dist}\left(P_{i},\partial\boldsymbol{\Sigma}_{t_{i}}\right).\label{distance estimate}
\end{equation}

Let $C\geq1$ be a large number to be determined. For $i\gg1$, let
$\lambda_{i}=C\left|P-P_{i}\right|$, then Corollary \ref{flat tangent flow}
implies that
\[
\boldsymbol{\Sigma}_{\tau}^{\left(P,T\right),\lambda_{i}}\coloneqq\frac{1}{\lambda_{i}}\left(\boldsymbol{\Sigma}_{T+\lambda_{i}^{2}\tau}-P\right)
\]
converges to a half-plane $\Pi$ with multiplicity one as $i\rightarrow\infty$.
In particular, we have
\[
\left\Vert A_{\boldsymbol{\Sigma}_{-1}^{\left(P,T\right),\lambda_{i}}}\right\Vert _{L^{\infty}\left(B_{1}\left(O\right)\right)}\leq1
\]
for $i\gg1$. Moreover, since
\[
\sup_{-\frac{T}{\lambda_{i}^{2}}\leq\tau<0}\left\Vert H_{\boldsymbol{\Sigma}_{\tau}^{\left(P,T\right),\lambda_{i}}}\right\Vert _{L^{\infty}}\leq\lambda_{i}\boldsymbol{\Lambda}
\]
and $\boldsymbol{\Gamma}^{P,\lambda_{i}}$ satisfies a $\lambda_{i}\boldsymbol{\kappa}-$graph
condition, Proposition \ref{unit area ratio of MCF} implies that
there is $0<\delta<1$ so that 
\begin{equation}
\sup_{0<r\leq\delta}\frac{\mathcal{H}^{2}\left(\boldsymbol{\Sigma}_{\tau}^{\left(P,T\right),\lambda_{i}}\cap B_{r}\left(Q\right)\right)_{Q}+\mathcal{H}^{2}\left(\left(\boldsymbol{\Sigma}_{\tau}^{\left(P,T\right),\lambda_{i}}\cap B_{r}\left(Q\right)\right)_{Q}\cap\tilde{B}_{r}\left(Q\right)\right)}{\pi r^{2}}\leq1+\vartheta\label{unit area ratio condition}
\end{equation}
\[
\forall\;\,Q\in\boldsymbol{\Sigma}_{\tau}^{\left(P,T\right),\lambda_{i}}\cap B_{\delta}\left(O\right),\quad-1\leq\tau<\min\left\{ 0,-1+\frac{\delta}{\lambda_{i}\Lambda}\right\} =0,\quad i\gg1,
\]
where $\vartheta>0$ is the constant in Lemma \ref{rigidity}. Note
that the reflection $\tilde{B}_{r}\left(Q\right)$ is with respect
to $\boldsymbol{\Gamma}^{P,\lambda_{i}}\coloneqq\frac{1}{\lambda_{i}}\left(\boldsymbol{\Gamma}-P\right)$
(see Definition \ref{complementary ball}). 

Let $A_{i}=\left|A_{\boldsymbol{\Sigma}_{t_{i}}}\left(P_{i}\right)\right|$
and define 
\[
\hat{\boldsymbol{\Sigma}}_{\tau}^{i}=A_{i}\left(\boldsymbol{\Sigma}_{t_{i}+\frac{\tau}{A_{i}^{2}}}-P_{i}\right),
\]
then we have 
\[
\sup_{t_{i}A_{i}^{2}\leq\tau\leq0}\left\Vert A_{\hat{\boldsymbol{\Sigma}}_{\tau}^{i}}\right\Vert _{L^{\infty}}\leq1=\left|A_{\hat{\boldsymbol{\Sigma}}_{0}^{i}}\left(O\right)\right|,
\]
\[
\sup_{t_{i}A_{i}^{2}\leq\tau\leq0}\left\Vert H_{\hat{\boldsymbol{\Sigma}}_{\tau}^{i}}\right\Vert _{L^{\infty}}\leq\frac{\boldsymbol{\Lambda}}{A_{i}}.
\]
Note that 
\begin{equation}
\hat{\boldsymbol{\Sigma}}_{0}^{i}=\lambda_{i}A_{i}\,\boldsymbol{\Sigma}_{\tau_{i}}^{\left(P,T\right),\lambda_{i}}+A_{i}\left(P-P_{i}\right),\label{correspondence of two rescalings}
\end{equation}
where 
\[
\lambda_{i}A_{i}=C\left|P-P_{i}\right|\left|A_{\boldsymbol{\Sigma}_{t_{i}}}\left(P_{i}\right)\right|\geq\frac{2}{3}C\,\textrm{dist}\left(\partial\boldsymbol{\Sigma}_{t_{i}},P_{i}\right)\left|A_{\boldsymbol{\Sigma}_{t_{i}}}\left(P_{i}\right)\right|\rightarrow\infty,
\]
\[
A_{i}\left|P-P_{i}\right|\geq\frac{2}{3}\,\textrm{dist}\left(P_{i},\partial\boldsymbol{\Sigma}_{t_{i}}\right)\left|A_{\boldsymbol{\Sigma}_{t_{i}}}\left(P_{i}\right)\right|\rightarrow\infty
\]
by (\ref{distance estimate}), and 
\[
\tau_{i}=-\frac{T-t_{i}}{\lambda_{i}^{2}}\,\in\left[-C^{-2},0\right)
\]
by (\ref{pseudo-interior seq}). Furthermore, given $\mathcal{Q}\in\hat{\boldsymbol{\Sigma}}_{0}^{i}\cap B_{\frac{1}{2}A_{i}\left|P-P_{i}\right|}\left(O\right)$
and $R>0$, (\ref{unit area ratio condition}) and (\ref{correspondence of two rescalings})
yield
\[
\frac{\mathcal{H}^{2}\left(\hat{\boldsymbol{\Sigma}}_{0}^{i}\cap B_{R}\left(\mathcal{Q}\right)\right)_{\mathcal{Q}}+\mathcal{H}^{2}\left(\left(\hat{\boldsymbol{\Sigma}}_{0}^{i}\cap B_{R}\left(\mathcal{Q}\right)\right)_{\mathcal{Q}}\cap\tilde{B}_{R}\left(\mathcal{Q}\right)\right)}{\pi R^{2}}
\]
\[
=\frac{\mathcal{H}^{2}\left(\boldsymbol{\Sigma}_{\tau_{i}}^{\left(P,T\right),\lambda_{i}}\cap B_{\frac{R}{\lambda_{i}A_{i}}}\left(Q\right)\right)_{Q}+\mathcal{H}^{2}\left(\left(\boldsymbol{\Sigma}_{\tau_{i}}^{\left(P,T\right),\lambda_{i}}\cap B_{\frac{R}{\lambda_{i}A_{i}}}\left(Q\right)\right)_{Q}\cap\tilde{B}_{\frac{R}{\lambda_{i}A_{i}}}\left(Q\right)\right)}{\pi\left(\frac{R}{\lambda_{i}A_{i}}\right)^{2}}\leq1+\vartheta
\]
for $i\gg1$, where 
\[
Q=\frac{1}{\lambda_{i}}\left(P_{i}-P+\frac{\mathcal{Q}}{A_{i}}\right)\,\in\boldsymbol{\Sigma}_{\tau_{i}}^{\left(P,T\right),\lambda_{i}}\cap\left(B_{\frac{3}{2}C^{-1}}\left(O\right)\setminus B_{\frac{1}{2}C^{-1}}\left(O\right)\right).
\]
Note that $\left(B_{\frac{3}{2}C^{-1}}\left(O\right)\setminus B_{\frac{1}{2}C^{-1}}\left(O\right)\right)=B_{\delta}\left(O\right)\setminus B_{\frac{\delta}{3}}\left(O\right)$
if we choose $C=\frac{3}{2}\delta^{-1}$, and that the reflection
in the first ratio (i.e. $\tilde{B}_{R}\left(\mathcal{Q}\right)$)
is with respect to $\hat{\boldsymbol{\Gamma}}^{i}=A_{i}\left(\boldsymbol{\Gamma}-P_{i}\right)$.
Passing to a subsequence, we have $\left\{ \hat{\boldsymbol{\Sigma}}_{\tau}^{i}\right\} \rightarrow\left\{ \hat{\boldsymbol{\Sigma}}\right\} _{-\infty<\tau\leq0}$,
where $\hat{\boldsymbol{\Sigma}}$ is a complete minimal surface (with
free boundary on $\lim_{i\rightarrow\infty}\hat{\boldsymbol{\Gamma}}^{i}$
if $\lim_{i\rightarrow\infty}\hat{\boldsymbol{\Gamma}}^{i}\neq\emptyset$).
The limiting minimal surface satisfies 
\[
\sup_{R>0}\,\sup_{\mathcal{Q}\in\hat{\boldsymbol{\Sigma}}}\frac{\mathcal{H}^{2}\left(\hat{\boldsymbol{\Sigma}}\cap B_{R}\left(\mathcal{Q}\right)\right)_{\mathcal{Q}}+\mathcal{H}^{2}\left(\left(\hat{\boldsymbol{\Sigma}}\cap B_{R}\left(\mathcal{Q}\right)\right)_{\mathcal{Q}}+\tilde{B}_{R}\left(\mathcal{Q}\right)\right)}{\pi R^{2}}\leq1+\vartheta,
\]
\[
\left\Vert A_{\hat{\boldsymbol{\Sigma}}}\right\Vert _{L^{\infty}}\leq1=\left|A_{\hat{\boldsymbol{\Sigma}}}\left(O\right)\right|.
\]
A contradiction follows from Lemma \ref{rigidity}.
\end{proof}

\vspace{0.25in}
\email{
\noindent Department of Mathematics, Indiana University - Rawles Hall\\
831 East 3rd St.
Bloomington, IN 47405-7106\\\\
E-mail address: \textsf{siaoguo@iu.edu}
}


\begin{thebibliography}{LZZ}
\bibitem[Al]{Al}W. Allard: On the first variation of a varifold,
Ann. Math., 95 (1972), 417-491.

\bibitem[An]{An}B. Andrews: Fully nonlinear parabolic equations in
two space variables, arXiv:math/0402235v1.

\bibitem[AS]{AS}D. Aronson, J. Serrin: Local behavior of solutions
of quasilinear parabolic equations, Arch. Rational Mech. Anal. 25
1967, 81-122.

\bibitem[B]{B}J. Buckland: Mean curvature flow with free boundary
on smooth hypersurfaces, J. Reine Angew. Math. 586 (2005), 71\textendash 90.

\bibitem[C]{C}A. Cooper: A characterization of the singular time
of the mean curvature flow, Proc. Amer. Math. Soc. 139 (2011), no.
8, 29332942.

\bibitem[CM]{CM}T. Colding, W. Minicozzi II: Smooth compactness of
self-shrinkers, Comment. Math. Helv. 87 (2012), no. 2, 463-475. 

\bibitem[E]{E}K. Ecker: Regularity theory for mean curvature flow,
Birkh$\ddot{\textrm{a}}$user, Boston (2004).

\bibitem[EH]{EH}K. Ecker, G. Huisken: Interior estimates for hypersurfaces
moving by mean curvature, Invent. math. 105, 547-569 (1991).

\bibitem[GJ]{GJ}M. Gr$\ddot{\textrm{u}}$ter, J. Jost: Allard type
regularity results for varifolds with free boundaries, Ann. Scu. Norm.
Sup. Pisa Cal. Sci (4) 13 (1986), n. 1, 129-169.

\bibitem[H1]{H1}G. Huisken: Non-parametric mean curvature evolution
with boundary conditions, J. Differ. Eqns.77 (1989) 369-378.

\bibitem[H2]{H2}G. Huisken: Asymptotic behaviour for singularities
of the mean curvature flow, J. Diff. Geom. 31, 285-299 (1990).

\bibitem[HS]{HS}G. Huisken, C. Sinestrari: Convexity estimates for
mean curvature flow and singularities of mean convex surfaces, Acta
Math. 183 (1999), no. 1, 45-70. 

\bibitem[I]{I}T. Ilmanen: Singularities of mean curvature flow of
surfaces, preprint.

\bibitem[K]{K}A. Koeller: Regularity of mean curvature flows with
Neumann free boundary conditions, Calc. Var. (2012) 43:265-309.

\bibitem[LS]{LS}N. Le, N. Sesum: The mean curvature at the first
singular time of the mean curvature flow, Ann. Inst. H. Poincar$\acute{\textrm{e}}$
Anal. Non Lin$\acute{\textrm{e}}$aire 27 (2010), no. 6, 14411459.

\bibitem[LW]{LW}H. Li, B. Wang: The extension problem of the mean
curvature flow (I), arXiv:1608.02832.

\bibitem[LZZ]{LZZ}Y. Leng, E. Zhao, H. Zhao: Notes on the extension
of the mean curvature flow, Pacific Journal of Mathematics 269 (2),
385-392, 2014. 1, 2014.

\bibitem[PR]{PR}J. P$\acute{\textrm{e}}$rez, A. Ros: Properly embedded
minimal surfaces with finite total curvature. The global theory of
minimal surfaces in flat spaces (Martina Franca, 1999), 15\textendash 66,
Lecture Notes in Math., 1775, Fond. CIME/CIME Found. Subser., Springer,
Berlin, 2002.

\bibitem[S]{S}A. Stahl: Regularity estimates for solutions to the
mean curvature flow with a Neumann boundary condition, Calc. Var.
4, 385-407 (1996).

\bibitem[Wa]{Wa}W. Walter: On the strong maximum principle for parabolic
differential equations, Proc. Edinburgh Math. Soc. (2) 29 (1986),
no. 1, 93\textendash 96.

\bibitem[Wh]{Wh}B. White: A local regularity theorem for mean curvature
flow, Ann. of Math. (2) 161 (2005), no. 3, 1487-1519.
\end{thebibliography}
\end{document}